\documentclass[
a4paper 
,11pt 
,leqno 
,english 
]{scrartcl} 

\usepackage{my-preamble-article} 
\usepackage{my-macros} 
\usepackage{my-specific-macros} 
\usepackage{my-counters-related-preprints} 

\title{Global convergence towards pushed travelling fronts for parabolic gradient systems}
\author{\RamonsName{} and \myName}
\begin{document}
\hypersetup{pageanchor=false} 
\maketitle
\nnfootnote{%
}
\begin{abstract}
This article addresses the issue of global convergence towards pushed travelling fronts for solutions of parabolic systems of the form
\[
u_t = - \nabla V(u) + u_{xx}
\,,
\]
where 
the potential $V$ is coercive at infinity. It is proved that, if an initial condition $x\mapsto u(x,t=0)$ approaches, rapidly enough, a critical point $e$ of $V$ to the right end of space, and if, for some speed $c_0$ greater than the linear spreading speed associated with $e$, the energy of this initial condition in a frame travelling at the speed $c_0$ is negative --- with symbols, 
\[
\int_{\rr} e^{c_0 x}\left(\frac{1}{2} u_x(x,0)^2 + V\bigl(u(x,0)\bigr)- V(e)\right)\, dx < 0 
\,,
\]
then the corresponding solution invades $e$ at a speed $c$ greater than $c_0$, and approaches, around the leading edge and as time goes to $+\infty$, profiles of pushed fronts (in most cases a single one) travelling at the speed $c$. A necessary and sufficient condition for the existence of pushed fronts invading a critical point at a speed greater than its linear spreading speed follows as a corollary. In the absence of maximum principle, the arguments are purely variational. 
The key ingredient is a Poincaré inequality showing that, in frames travelling at speeds exceeding the linear spreading speed, the variational landscape does not differ much from the case where the invaded equilibrium $e$ is stable. 
The proof is notably inspired by ideas and techniques introduced by Th. Gallay and R. Joly, and subsequently used by C. Luo, in the setting of nonlinear damped wave equations. 
\end{abstract}
\thispagestyle{empty}
\pagestyle{empty}
\newpage%
\tableofcontents
\newpage
\hypersetup{pageanchor=true} 
\pagestyle{plain}
\setcounter{page}{1}
\section{Introduction and statements of the main results}
\label{sec:introduction}
\subsection{System, semi-flow}
\label{subsec:system_semi_flow}
Let us consider the nonlinear parabolic system
\begin{equation}
\label{parabolic_system}
u_t=-\nabla V (u) + u_{xx}
\,,
\end{equation}
where the time variable $t$ and the space variable $x$ are real, the spatial domain is the whole real line, the function $(x,t)\mapsto u(x,t)$ takes its values in $\rr^d$ with $d$ a positive integer, and the nonlinearity is the gradient of a \emph{potential} function $V:\rr^d\to\rr$, of class $\ccc^2$, and \emph{coercive at infinity} in the following sense:
\begin{gather} 
\tag{$\text{H}_\text{coerc}$}
\lim_{R\to+\infty}\quad  \inf_{\abs{u}\ge R}\ \frac{u\cdot \nabla V(u)}{\abs{u}^2} >0
\,.
\label{hyp_coerc}
\end{gather}
The uniformly local Sobolev space $\HoneulofR$ (see \cref{subsec:global_existence_solutions_regularization} for references) provides a natural framework for the study of system \cref{parabolic_system} on the whole real line. This system defines a local semi-flow in this space, and according the assumption \cref{hyp_coerc}, this semi-flow is actually global (\cref{prop:attr_ball} below). Let us denote by $(S_t)_{t\ge0}$ this semi-flow. In the following, a \emph{solution of system \cref{parabolic_system}} will refer to a function 
\[
\rr\times[0,+\infty)\to\rr^d\,, \quad (x,t)\mapsto u(x,t)
\,,
\]
such that the function $u_0:x\mapsto u(x,t=0)$ (initial condition) is in $\HoneulofR$ and $u(\cdot,t)$ equals $(S_t u_0)(\cdot)$ for every nonnegative time $t$. 
\subsection{Invaded critical point}
\label{subsec:invaded_critical_point}
According to assumption \cref{hyp_coerc}, we may consider the quantity $\Vmin$ defined as
\begin{equation}
\label{def_Vmin}
\Vmin = \min_{u\in\rr^d}V(u)
\,.
\end{equation}
Let us consider a point $e$ of $\rr^d$, and let us assume that:
\begin{gather} 
\tag{$\text{H}_{\text{crit},\,e}$}
\nabla V(e)=0
\quad\text{and}\quad
V(e)=0
\quad\text{and}\quad
\Vmin<0
\,.
\label{hyp_crit_point}
\end{gather}
In other words, $e$ is assumed to be a critical point which is \emph{not} a global minimum of $V$, and $V$ is normalized so that it takes the value $0$ at $e$. The aim of this paper is mainly to address the case where $e$ is \emph{not} a nondegenerate minimum point of $V$; that is, if $D^2V(u)$ denotes the Hessian matrix of $V$ at some point $u$ of $\rr^d$ and $\sigma\bigl(D^2V(u)\bigr)$ denotes the spectrum of this Hessian matrix, the case where 
\begin{equation}
\label{hyp_e_not_local_min_point_of_V}
\min\Bigl(\sigma\bigl(D^2V(e)\bigr)\Bigr)\le 0
\,.
\end{equation}
Indeed, if $e$ \emph{is} a local minimum point of $V$, global convergence towards travelling fronts invading $e$ can be addressed by differing techniques leading to stronger results \cite{Risler_globCVTravFronts_2008,Risler_globalBehaviour_2016}. As a consequence, the reader may assume that, in addition to assumption \cref{hyp_crit_point}, inequality \cref{hyp_e_not_local_min_point_of_V} holds, even if this inequality will nowhere be formally required. 
\subsection{Travelling waves/fronts}
\subsubsection{Definition}
Let $c$ be a \emph{positive} real quantity. A \emph{wave travelling at the speed $c$} for system \cref{parabolic_system} is a function of the form $(x,t)\mapsto \phi(x-ct)$, where $\phi$ is a solution (with values in $\rr^d$) of the second order differential system
\begin{equation}
\label{syst_trav_front_order_2}
\phi''=-c\phi'+\nabla V(\phi) 
\,,
\end{equation}
which is equivalent to the first order differential system 
\begin{equation}
\label{syst_trav_front_order_1}
\begin{pmatrix} \phi' \\ \psi' \end{pmatrix} = \begin{pmatrix}\psi \\ - c \psi + \nabla V(\phi)  \end{pmatrix}
\,.
\end{equation}
The function $\phi$ is called the \emph{profile} of the travelling wave. Notice that the solutions of systems \cref{syst_trav_front_order_2,syst_trav_front_order_1} may blow up in finite time, so that a travelling wave is, formally, not necessarily a solution of system \cref{parabolic_system} as defined in \cref{subsec:system_semi_flow}; its profile $\phi$ is a function $(T_-,T_+)\to\rr^d$, where $(T_-,T_+)$ is the (maximal) time of existence of $\phi$ as a solution of systems \cref{syst_trav_front_order_2,syst_trav_front_order_1} (thus $T_-$ is in $\{-\infty\}\cup\rr$ and $T_+$ is in $\rr\cup\{+\infty\}$, and the travelling wave $(x,t)\mapsto \phi(x-ct)$ itself is defined accordingly).

Let $\SigmaCrit(V)$ denote the set of critical points of $V$; with symbols, 
\[
\SigmaCrit(V) = \{u\in\rr^d: \nabla V(u) = 0\}
\,.
\] 
\begin{definition}[travelling wave/front invading $e$]
\label{def:travelling_wave_front_invading_e_speed_c}
Let us call \emph{travelling wave invading $e$ at the speed $c$} a wave travelling at the speed $c$ such that its profile $\phi$ is \emph{nonconstant}, defined on a maximal interval of the form $(T_-,+\infty)$, and satisfies:
\[
\phi(\xi)\xrightarrow[\xi\to +\infty]{} e
\,;
\]
and let us call \emph{front} this wave if $\phi$ is defined up to $-\infty$ and if there exists a negative quantity $V_{-\infty}$ such that the following limit holds:
\[
\dist\Bigl(\phi(\xi),\SigmaCrit(V)\cap V^{-1}\bigl(\{V_{-\infty}\}\bigr)\Bigr)\xrightarrow[\xi\to -\infty]{} 0
\,.
\]
Finally, let us call \emph{travelling wave (front) invading $e$} a travelling wave (front) invading $e$ at some positive speed. 
\end{definition}
\begin{remarks}
\begin{enumerate}
\item As stated in conclusion \cref{item:lem_asymptotics_at_the_two_ends_of_space_left_end} of \cref{lem:asymptotics_at_the_two_ends_of_space}, in order a travelling wave (invading $e$ at the speed $c$) to be a travelling \emph{front} in the sense of \cref{def:travelling_wave_front_invading_e_speed_c}, it is sufficient that its profile $\phi(\cdot)$ be defined up to $-\infty$ and bounded on $\rr$. 
\item Voluntarily, this definition a travelling front slightly differs from the usual one, since it does not require that $\phi(\xi)$ approach a single critical point of $V$ as $\xi$ goes to $-\infty$. However, in most cases (at least if the critical points of $V$ are isolated --- which is true for a generic potential $V$, or if $V$ is analytic, see for instance \cite{Haraux_someApplicationsLojasiewiczGradientIneq_2012}), then the profile of a travelling front in the sense of \cref{def:travelling_wave_front_invading_e_speed_c} \emph{does} approach a single critical point of $V$ to the left end of space. 
\end{enumerate}
\end{remarks}
\subsubsection{Linearization at the invaded critical point}
The linearization of the (equivalent) differential systems \cref{syst_trav_front_order_2,syst_trav_front_order_1} at the point $(e,0_{\rr^d})$ reads:
\begin{equation}
\label{syst_trav_front_order_1_2_linearized}
\phi'' = -c\phi' + D^2V(e)\cdot\phi
\quad\text{and}\quad
\frac{d}{d\xi}\begin{pmatrix}\phi \\ \psi\end{pmatrix} = \begin{pmatrix} 0 & I_d \\ D^2V(e) & -c I_d\end{pmatrix}\cdot\begin{pmatrix}\phi \\ \psi\end{pmatrix}
\,.
\end{equation}
For every real quantity $\mu$, let us consider the quantities $\lambda_{c,\pm}(\mu)$ defined as
\begin{equation}
\label{def_lambda_c_pm_of_mu}
\lambda_{c,\pm}(\mu) = \left\{
\begin{aligned}
-\frac{c}{2}\pm\sqrt{\frac{c^2}{4}+\mu} \quad\text{if}\quad -\frac{c^2}{4} \le \mu \,, \\
-\frac{c}{2}\pm i\sqrt{-\frac{c^2}{4}-\mu} \quad\text{if}\quad \mu \le -\frac{c^2}{4} \,,  
\end{aligned}
\right.
\end{equation}
see \cref{fig:eigenvalues}; these two quantities are the roots of the polynomial equation $\lambda^2 = -c\lambda +\mu$. Let $\mu_1,\dots,\mu_d$ denote the (real) eigenvalues of $D^2V(e)$, counted with algebraic multiplicity, and ordered, so that
\begin{equation}
\label{def_mu_1_mu_d}
\mu_1\le\dots\le\mu_d
\,.
\end{equation}
These quantities will also be referred to as the \emph{curvatures} of the potential $V$ at the critical point $e$. No assumption is made concerning their signs (see the remark following inequality \cref{hyp_e_not_local_min_point_of_V}). According to this notation, the eigenvalues of the linearized differential systems \cref{syst_trav_front_order_1_2_linearized} (counted with algebraic multiplicity) are the $2d$ quantities:
\begin{equation}
\label{def_sigma_c_2d_of_e}
\lambda_{c,-}(\mu_1),\,\lambda_{c,+}(\mu_1),\,\dots,\,\lambda_{c,-}(\mu_d),\,\lambda_{c,+}(\mu_d)
\,, 
\end{equation}
see \cref{fig:eigenvalues}. 
\begin{figure}[!htbp]
\centering
\includegraphics[width=.9\textwidth]{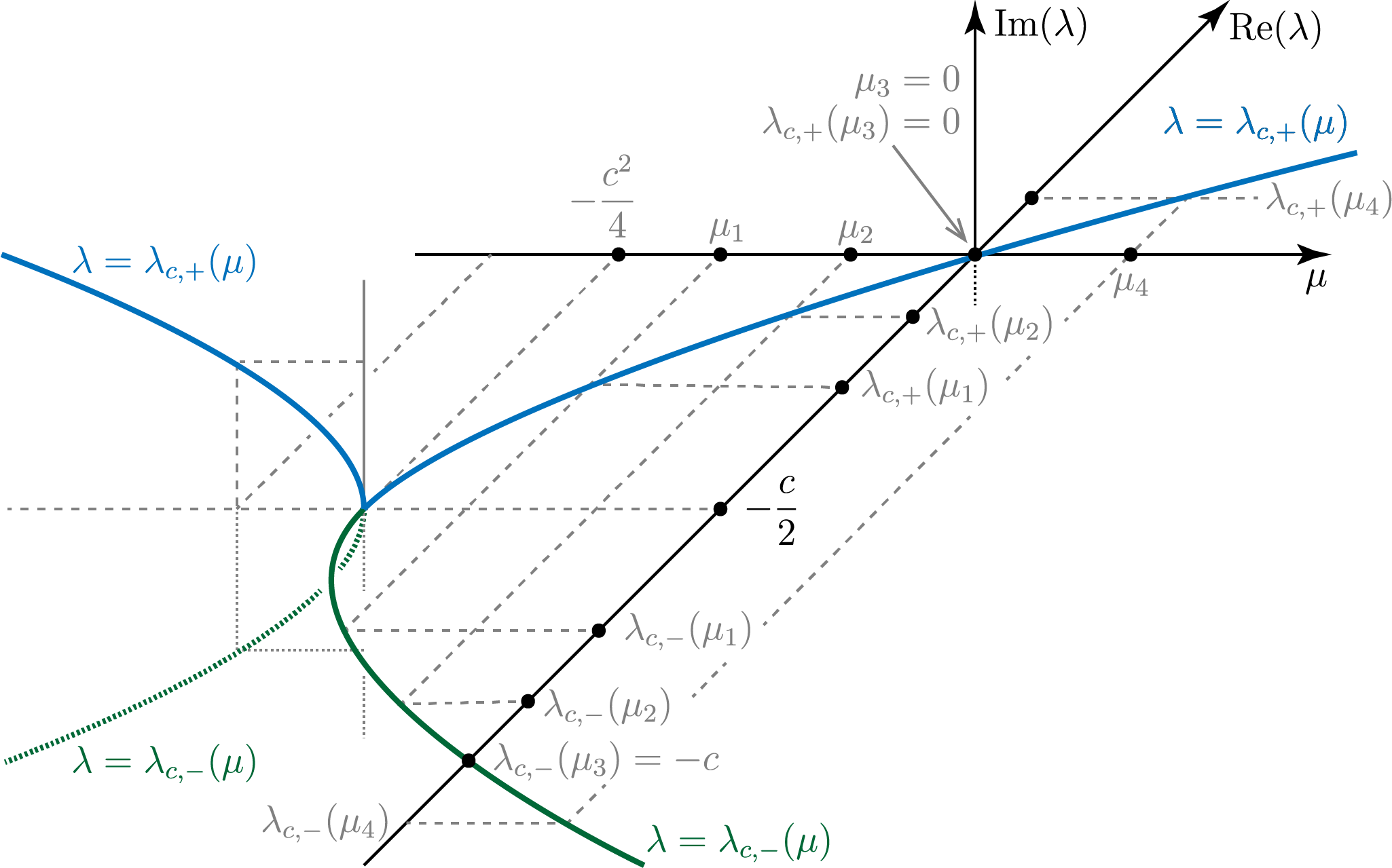}
\caption{Functions $\mu\mapsto\lambda_{c,\pm}(\mu)$ and eigenvalues of the linearized differential system \cref{syst_trav_front_order_1_2_linearized}. Here, as an example, the spectrum of $D^2V(e)$ is equal to $\{\mu_1,\dots,\mu_4\}$ with $-c^2/4 < \mu_1 < \mu_2 < \mu_3 = 0 < \mu_4$, but in the article $\mu_1$ is not necessarily nonpositive.}
\label{fig:eigenvalues}
\end{figure}
Let $(u_1,\dots,u_d)$ denote an orthonormal basis of $\rr^d$ such that, for every $j$ in $\{1,\dots,d\}$, $u_j$ is an eigenvector of $D^2V(e)$ for the eigenvalue $\mu_j$. Then, for every $j$ in $\{1,\dots,d\}$, the vectors
\begin{equation}
\label{eigenvectors}
\begin{pmatrix} u_j \\ \lambda_{c,-}(\mu_j) \, u_j\end{pmatrix}
\quad\text{and}\quad
\begin{pmatrix} u_j \\ \lambda_{c,+}(\mu_j) \, u_j\end{pmatrix}
\end{equation}
are eigenvectors of the linearized systems \cref{syst_trav_front_order_1_2_linearized}, for the eigenvalues $\lambda_{c,-}(\mu_j)$ and $\lambda_{c,+}(\mu_j)$, respectively. 
\subsubsection{Maximal linear invasion speed}
For every real quantity (curvature) $\mu$, let us consider the (nonnegative) quantity $\cLinNoIndex(\mu)$ defined as
\begin{equation}
\label{def_cLinNoIndex}
\cLinNoIndex(\mu) = 
\left\{
\begin{aligned}
2\sqrt{-\mu} = 2\sqrt{\abs{\mu}} \quad&\text{if}\quad \mu<0 \,,\\
0 \quad&\text{if}\quad \mu\ge 0\,,
\end{aligned}
\right.
\end{equation}
see \cref{fig:correspondence_mu_c}, and, for every $j$ in $\{1,\dots,d\}$, let us denote by $\cLin{j}$ the quantity $\cLinNoIndex(\mu_j)$. 
\begin{figure}[!htbp]
\centering
\includegraphics[width=.9\textwidth]{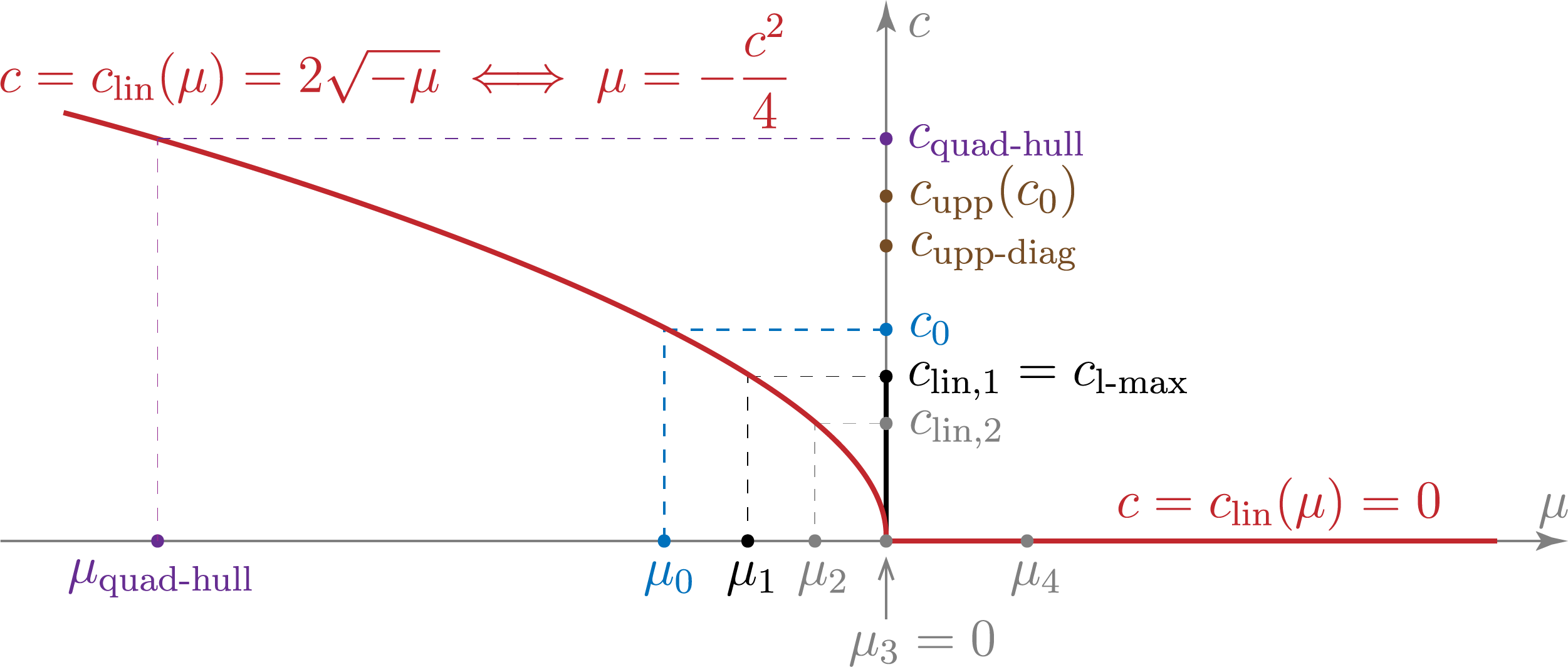}
\caption{Correspondence $\mu\mapsto\cLinNoIndex(\mu)$ between the curvatures of $V$ at $e$ and the corresponding linear invasion speeds. As on \cref{fig:correspondence_mu_c}, the spectrum of $D^2V(e)$ is equal to $\{\mu_1,\dots,\mu_4\}$ with $-c^2/4 < \mu_1 < \mu_2 < \mu_3 = 0 < \mu_4$. Only invasion at speeds greater than $\cLinMax$ are considered in this article. The quantity $c_0$ denotes a generic speed which is greater than (sometimes greater than or equal to) $\cLinMax$, and $\mu_0$ denotes its companion curvature. The quantities $\muQuadHull$ and $\cQuadHull$ will be introduced in \cref{subsubsec:lower_quadratic_hull}, and the quantities $\cUppDiag$ and $\cUpp(c_0)$ will be introduced in \cref{sec:upper_bound_speed_pushed_front}.}
\label{fig:correspondence_mu_c}
\end{figure}
\begin{definition}[maximal linear invasion speed]
\label{def:max_linear_invasion_speed}
Let us call \emph{maximal linear invasion speed} (associated with the critical point $e$) the (nonnegative) quantity $\cLinMax$ defined as
\begin{equation}
\label{def_max_linear_invasion_speed}
\cLinMax = \max(\cLin{1},\dots,\cLin{d})\,,
\quad\text{or equivalently}\quad
\cLinMax = \cLin{1}
\,.
\end{equation}
\end{definition}
Accordingly, the quantities $\cLin{1},\dots,\cLin{d}$ may be called the \emph{linear invasion speeds} associated with the eigenvalues $\mu_1,\dots,\mu_d$ of $D^2V(e)$, but only the maximal linear invasion speed $\cLinMax$ will play a significant role in the following. This quantity $\cLinMax$ is usually called \emph{linear spreading speed} in the literature, see for instance \cite{VanSaarloos_frontPropagationUnstableStates_2003}, and is referred as such in the abstract of this article. In the following only the denomination ``maximal linear invasion speed'' will be used to denote this quantity: the adjective ``maximal'' is relevant for systems, and the terms ``invasion/invaded/invading'', which fit with the phenomenon of propagation (to the right) into the state $e$ considered here, are ubiquitous. In addition, a ``maximal \emph{nonlinear} invasion speed'' $\cNonLinMax$ will be introduced in \cref{subsec:variational_structure_trave_frames}. According to the expression \cref{def_cLinNoIndex} of $\cLinNoIndex(\cdot)$, this maximal linear invasion speed $\cLinMax$ is nonnegative (but might vanish).

It follows from the expression \cref{def_lambda_c_pm_of_mu} of $\lambda_{c,\pm}(\cdot)$ that, for every $j$ in $\{1,\dots,d\}$, if 
\[
c>\cLin{j}
\,,
\quad\text{or equivalently if}\quad
-\frac{c^2}{4} < \mu_j
\,,
\]
then the eigenvalues $\lambda_{c,-}(\mu_j)$ and $\lambda_{c,+}(\mu_j)$ of the linear system \cref{syst_trav_front_order_1_2_linearized} are real and distinct, and the corresponding eigenvectors \cref{eigenvectors} are real (see \cref{fig:eigenvalues}).
\subsubsection{Pushed travelling waves and fronts invading a critical point}
Let us keep the previous notation, and let $\phi$ denote the profile of a wave invading the critical point $e$ at the speed $c$. According to the previous considerations, there must exist some nonpositive eigenvalue $\lambda$ among the quantities \cref{def_sigma_c_2d_of_e} such that
\begin{equation}
\label{def_steepness}
\frac{\ln\abs{\phi(\xi)-e}}{\xi}\to \lambda
\quad\text{as}\quad 
\xi\to+\infty
\,.
\end{equation}
\begin{definition}[steepness of a travelling wave invading $e$]
\label{def:steepness}
Let us call \emph{steepness} of the wave under consideration the quantity $\lambda$ defined by the limit \cref{def_steepness}. 
\end{definition}
\begin{definition}[pushed travelling wave/front invading $e$]
\label{def:pushed_travelling_wave_front}
A travelling wave (front) invading the critical point $e$ at some positive speed $c$ is said to be \emph{pushed} if its steepness $\lambda$ (\cref{def:steepness}) satisfies the inequality
\[
\lambda<-\frac{c}{2}\,,
\quad\text{or equivalently}\quad
\frac{c}{2}<\abs{\lambda}
\,,
\]
or equivalently if the following limit holds:
\[
\phi(\xi) -e  = o\bigl(e^{-\frac{1}{2}c\xi}\bigr)
\quad\text{as}\quad
\xi\to+\infty
\,.
\]
\end{definition}
\begin{remark}
As shown on \cref{fig:phase_portrait}, if the least eigenvalue $\mu_1$ of $D^2V(e)$ is negative then the pushed (or ``steep'') character is non-generic among the solutions of the differential system \cref{syst_trav_front_order_1} (with $c$ greater than $\cLinMax$) approaching $e$ at the right end of space. However, when pushed travelling fronts exist, these fronts are approached by solutions of the parabolic system \cref{parabolic_system} for a ``large and relevant'' set of initial conditions, at least if their speed is greater than the maximal linear invasion speed $\cLinMax$ (see for instance \cite{Rothe_convergenceToPushedFronts_1981} in the scalar case $d$ equals $1$, and \cref{thm:main} below --- the main result of this paper --- in the vector case $d$ larger than $1$). On the other hand, if $\mu_1$ is nonnegative (in particular if $e$ is a local minimum point of $V$), then all waves (fronts) invading $e$ at a positive speed are ``pushed'' in the sense of \cref{def:pushed_travelling_wave_front}, although usually not qualified as such. 
\end{remark}
\begin{figure}[!htbp]
\centering
\includegraphics[width=.4\textwidth]{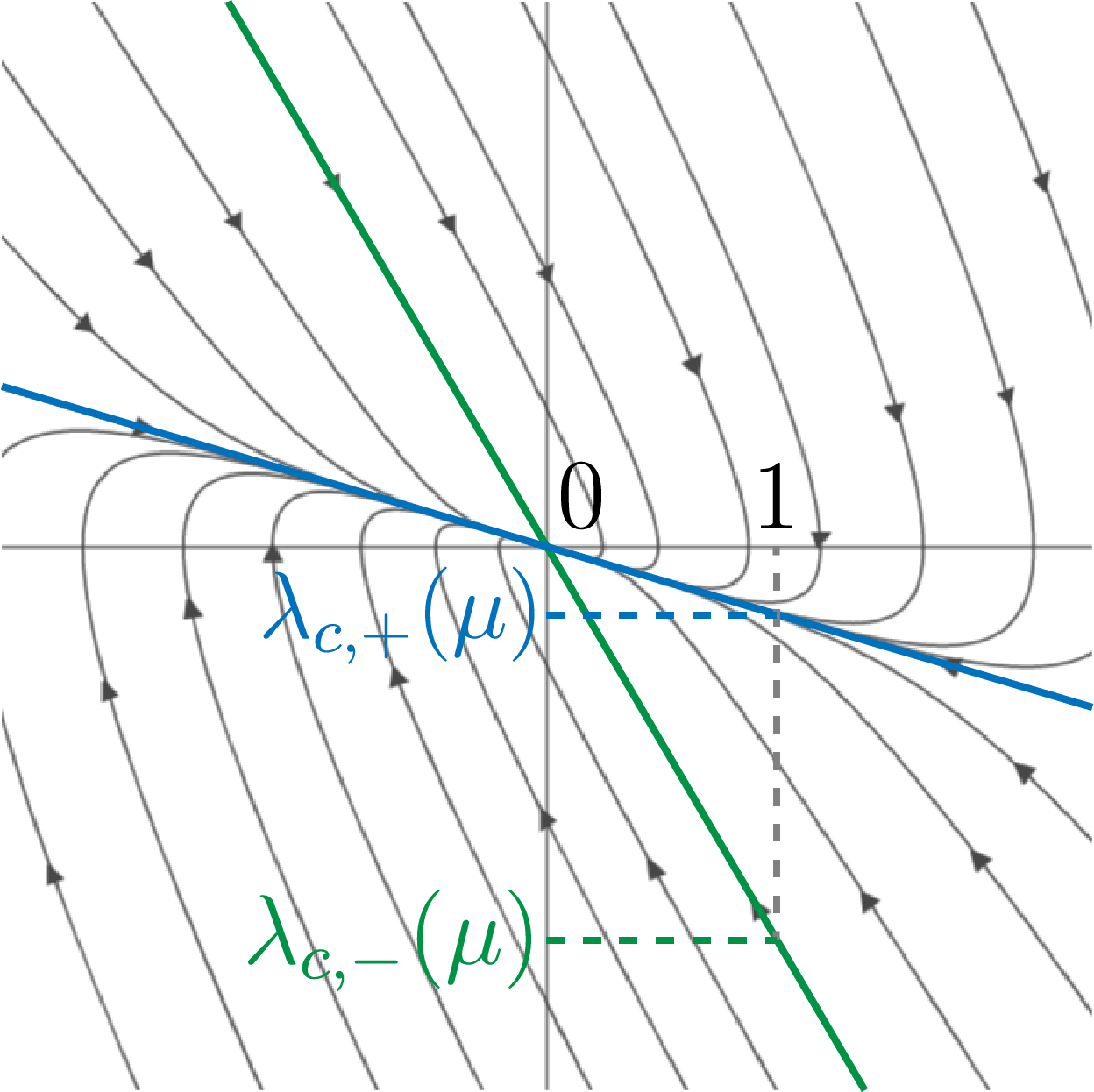}
\caption{Phase portrait of the linearized differential system \cref{syst_trav_front_order_1_2_linearized} for $d=1$, $c=2$, $D^2V(e)=\mu=-1/2$. The (real) eigenvalues are $\lambda_{c,\pm}(-1/2)=-1\pm 1/\sqrt{2}$. In this case, a pushed travelling wave invading $e$ corresponds to a solution of the (nonlinear) differential system \cref{syst_trav_front_order_1} approaching the origin tangentially to the eigenspace corresponding to the ``most stable'' eigenvalue $\lambda_{c,-}(-1/2)$ (in green on the figure). When these travelling waves are travelling fronts, they are sometimes qualified as \emph{nonlinear} while the fronts with profiles approaching the origin tangentially to the eigenspace corresponding to the ``least stable'' eigenvalue $\lambda_{c,+}(-1/2)$ (in blue of the figure) are qualified as \emph{linear}, \cite{CrossHohenberg_patternFormationOutsideEqu_1993,VanSaarloos_frontPropagationUnstableStates_2003}.}
\label{fig:phase_portrait}
\end{figure}
\subsubsection{Genericity/counting arguments related to the profiles of pushed travelling fronts}
\label{subsubsec:genericity_counting_arguments}
The following statements are proved in \cite{JolyOliverBRisler_genericTransvPulledPushedTFParabGradSyst_2023} (see also \cite{JolyRisler_genericTransversalityTravStandFrontsPulses_2023}): for a generic potential $V$, 
\begin{enumerate}
\item the set of profiles of pushed travelling fronts is discrete, 
\item and, for every pushed travelling front, 
\begin{enumerate}
\item the speed of the front is greater than the maximal linear invasion speed of the critical point invaded by the front, 
\item the profile of the front approaches a single ``invading'' critical point at the left end of space (compare with \cref{def:travelling_wave_front_invading_e_speed_c}), this invading critical point is a nondegenerate local minimum point of $V$,
\item the profile of the front approaches the invaded critical point at the right end of space (the invading critical point at the left end of space) tangentially to the least eigenvalue of the Hessian $D^2V(\cdot)$ at this invaded (invading) critical point. 
\label{item:conclusion_genericity_approach_through_least_eigenvalue}
\end{enumerate}
\end{enumerate}
A rough justification of these statements is provided by the following counting arguments. Let us recall that the Morse index of a nondegenerate critical point of $V$ is the number of negative eigenvalues of the Hessian $D^2V$ at this critical point. If $e_-$ and $e_+$ are two nondegenerate critical points of $V$ and $c$ is a positive speed, 
\begin{itemize}
\item the dimension of the unstable manifold of the equilibrium $(e_-,0_{\rr^d})$ for the differential system \cref{syst_trav_front_order_1} is equal to $d-m(e_-)$, 
\item and if $\jss$ denotes the smallest index $j$ in $\{1,\dots,d\}$ such that the $j$th eigenvalue $\mu_j$ of $D^2V(e_+)$ (by increasing order as in \cref{def_mu_1_mu_d}) is greater than $-c^2/4$, then the dimension of the ``steep stable'' manifold of $(e_+,0_{\rr^d})$ is equal to $d-\jss+1$ (see \cref{fig:eigenvalues}).
\end{itemize}
The profiles of pushed fronts connecting $e_-$ (the invading critical point) to $e_+$ (the invaded one) live in the intersection between these two manifolds, which is at least one-dimensional if nonempty. As a consequence, with the additional freedom provided by the speed parameter $c$, a transverse intersection between these two manifolds cannot occur unless $m(e_-)$ is equal to $0$ and $\jss$ is equal to $1$, and in this case this (transverse) intersection is made of isolated trajectories; and a similar counting argument supports conclusion \cref{item:conclusion_genericity_approach_through_least_eigenvalue} (see \cite{JolyOliverBRisler_genericTransvPulledPushedTFParabGradSyst_2023,JolyRisler_genericTransversalityTravStandFrontsPulses_2023} for details).

These statements are not directly involved in the results of this paper nor in their proofs, but they will be called upon in some comments, and they shed light on the general picture. In particular, they show that, among pushed travelling fronts, the most relevant ones are those with a speed \emph{exceeding the maximal linear invasion speed of the invaded critical point} (the others do generically not exist). It is precisely the convergence towards these pushed travelling fronts that the variational techniques used in this paper most easily apply to, and, as a matter of fact, only this convergence will be addressed. 
\subsubsection{Parametrization of pushed waves and fronts invading a critical point and travelling at a speed exceeding its maximal linear invasion speed}
\label{subsubsec:parametrization_pushed_tw_and_pushed_tf}
Let us keep the notation introduced before the previous \cref{subsubsec:genericity_counting_arguments}.
\begin{proposition}[local ``steep'' stable manifold of $e$ for a speed $c$ exceeding the maximal linear invasion speed]
\label{prop:local_steep_stable_manifold}
For every speed $c$ in $(\cLinMax,+\infty)$ and for every small enough positive quantity $\delta$, there exists a map
\begin{equation}
\label{def_wssloc}
\wssloc{e}{\delta}{c}:\widebar{B}_{\rr^d}(e,\delta)\to \rr^d
\end{equation}
such that, for every $(u,v)$ in $\widebar{B}_{\rr^d}(e,\delta)\times\rr^d$, the following two conditions are equivalent:
\begin{enumerate}
\item $v = \wssloc{e}{\delta}{c}(u)$,
\item the solution $\xi\mapsto\bigl(\phi(\xi),\psi(\xi)\bigr)$ of system \cref{syst_trav_front_order_1} with initial condition $(u,v)$ at $\xi$ equals $0$ remains in $\widebar{B}_{\rr^d}(e,\delta)\times\rr^d$ for all $\xi$ in $[0,+\infty)$ and is the profile of a pushed travelling wave invading $e$.
\end{enumerate}
\end{proposition}
In other words, the set $\Wssloc{e}{\delta}{c}$ defined as the graph, over $\widebar{B}_{\rr^d}(e,\delta)$, of the map $\wssloc{e}{\delta}{c}(\cdot)$:
\[
\Wssloc{e}{\delta}{c} = \Bigr\{\bigl(u,\wssloc{e}{\delta}{c}(u)\bigr):u\in \widebar{B}_{\rr^d}(e,\delta)\Bigr\}
\]
can be seen as a ``steep'' local stable manifold of the equilibrium $(e,0_{\rr^d})$ for the differential system \cref{syst_trav_front_order_1}, where ``steep'' means ``at an exponential rate which is greater than $c/2$''. 
\begin{notation}
For every $u$ in $\partial B_{\rr^d}(e,\delta)$, let us denote by $\phi_{c,u}(\cdot)$ the (maximal) solution of the differential system \cref{syst_trav_front_order_2} for the initial condition
\begin{equation}
\label{initial_condition_phi_c_u}
\bigl(\phi_{c,u}(0),\phi_{c,u}'(0)\bigr) = \bigl(u,\wssloc{e}{\delta}{c}(u)\bigr)
\end{equation}
(so that $\phi_{c,u}(\cdot)$ is the profile of a pushed travelling wave invading $e$ at the speed $c$), and let us consider the set
\begin{equation}
\label{def_uPushedFront}
\begin{aligned}
\uPushedFront{e}{\delta}{c} = \bigl\{ u\in\partial B_{\rr^d}(e,\delta): & \ \phi_{c,u}\text{ is the profile of a pushed front}\\
& \text{ invading $e$ at the speed $c$} \bigr\}
\,.
\end{aligned}
\end{equation}
\end{notation}
\begin{remark}
Depending on the value of the speed $c$, an explicit value of a quantity $\delta$ ensuring that the conclusions of \cref{prop:local_steep_stable_manifold} hold is provided by \vref{prop:extension_local_steep_stable_manifold}.
\end{remark}
\subsection{Invasion through profiles of pushed fronts}
Let us keep the previous notation, and in particular let us still consider a speed $c$ which is greater than the maximal linear invasion speed $\cLinMax$. 
\begin{definition}[invasion through profiles of pushed fronts]
\label{def:invasion_through_profiles_pushed_fronts}
A solution $(x,t)\mapsto u(x,t)$ of the parabolic system \cref{parabolic_system} is said to \emph{invade $e$ at the speed $c$ through profiles of pushed fronts} if there exist a positive quantity $\delta$ and a function $t\mapsto\tilde{x}(t)$ in $\ccc^1\bigl([0,+\infty),\rr\bigr)$ such that the following conclusions hold: 
\begin{enumerate}
\item the conclusions of \cref{prop:local_steep_stable_manifold} hold for the speed $c$ and the parameter $\delta$;
\item the set $\uPushedFront{e}{\delta}{c}$ is nonempty;
\item for $t$ large enough positive, $u\bigl(\tilde{x}(t),t\bigr)$ is in $\partial B_{\rr^d}(e,\delta)$;
\item the following three limits hold as $t$ goes to $+\infty$:
\begin{enumerate}
\item $\tilde{x}'(t)\to c$, 
\item $\dist\Bigl(u\bigl(\tilde{x}(t),t\bigr),\uPushedFront{e}{\delta}{c}\Bigr) \to 0$,
\label{item:def_invasion_through_profiles_pushed_fronts_approach_to_uPushedFront}
\item for every positive quantity $L$, 
\begin{equation}
\label{def_invasion_approach_profile}
\sup_{y\in[-L,+\infty)}\abs{u\bigl(\tilde{x}(t)+y,t\bigr) - \phi_{c,u(\tilde{x}(t),t)}(y)} \to 0
\,.
\end{equation}
\end{enumerate}
\end{enumerate}
\end{definition}
\begin{remarks}
\begin{enumerate}
\item The profiles $\phi_{c,u(\tilde{x}(t),t)}(\cdot)$ involved in the limit \cref{def_invasion_approach_profile} are profiles of pushed \emph{waves} invading $e$, but not necessarily of pushed \emph{fronts} (these profiles are therefore not necessarily defined up to $-\infty$). However, due to the limit provided by conclusion \cref{item:def_invasion_through_profiles_pushed_fronts_approach_to_uPushedFront}, these profiles get closer and closer to profiles of actual pushed fronts as time goes to $+\infty$.  \Cref{def:invasion_through_profiles_pushed_fronts} is actually unchanged if, in the limit \cref{def_invasion_approach_profile}, the vector $u(\tilde{x}(t),t)$ parametrizing the profile $\phi_{c,u(\tilde{x}(t),t)}(\cdot)$ is replaced with a ``close'' vector in the set $\uPushedFront{e}{\delta}{c}$, so that this profile is replaced with the profile of an actual pushed travelling front (this alternative approach is used in Theorem~1 of \cite{Risler_globCVTravFronts_2008}).
\item Whatever the choice made for its formulation, \cref{def:invasion_through_profiles_pushed_fronts} does not ensure the convergence towards a single pushed travelling front. More precisely, the $\omega$-limit set $\llll$ of the function $t\mapsto u(\tilde{x}(t),t)$, defined as
\[
\llll = \bigcap_{t>0} \overline{\bigcup_{s>t} \bigl\{u\bigl(\tilde{x}(s),s\bigr)\bigr\}}
\,,
\]
is a nonempty compact connected subset of $\partial B_{\rr^d}(e,\delta)$, and is, according to condition \cref{item:def_invasion_through_profiles_pushed_fronts_approach_to_uPushedFront}, included in the set $\uPushedFront{e}{\delta}{c}$. If this set $\llll$ is reduced to a singleton $\{\uPushedFrontSingleton\}$, then the limit \cref{def_invasion_approach_profile} can be replaced with the simpler and more precise limit
\[
\sup_{y\in[-L,+\infty)}\abs{u\bigl(\tilde{x}(t)+y,t\bigr) - \phi_{c,\uPushedFrontSingleton}(y)} \to 0 
\quad\text{as}\quad 
t\to+\infty
\,,
\]
involving the profile $\phi_{c,\uPushedFrontSingleton}$ of a single pushed travelling front. Notice that the set $\llll$ is \emph{necessarily} reduced to such a singleton if the set $\uPushedFront{e}{\delta}{c}$ is discrete, which, according to the statements of \cref{subsubsec:genericity_counting_arguments}, is true for a generic potential $V$. 
\item \Cref{def:invasion_through_profiles_pushed_fronts} does not provide any information concerning the behaviour of the solution in the wake (to the left) of the invasion. Let us however briefly mention that, if a solution invading $e$ at the speed $c$ through the profiles of pushed fronts is, in addition, close to a nondegenerate minimum point of $V$ at the left end of space, then, under generic assumptions on the potential $V$, the global behaviour of this solution is rather well understood \cite{Risler_globalBehaviour_2016}. 
\end{enumerate}
\end{remarks}
\subsection{Energy in a travelling frame}
Let $(x,t)\mapsto u(x,t)$ denote a solution of system \cref{parabolic_system}, let $c$ denote a \emph{positive} quantity, and let us consider the function $(\xi,t)\mapsto v(\xi,t)$ defined as:
\begin{equation}
\label{change_of_variable_stand_trav_frame}
v(\xi,t) = u(x,t)
\quad\text{for}\quad 
x = ct + \xi
\,.
\end{equation}
This function $v(\cdot,\cdot)$ is a solution of the system
\begin{equation}
\label{parabolic_system_trav_frame}
v_t - cv_\xi = - \nabla V(v) + v_{\xi\xi}
\,.
\end{equation}
Multiplying this equation by $e^{c\xi} v_t$ and integrating over $\rr$ leads us to introduce the energy below (\cref{def:eee_c}), and for that purpose the following weighted Sobolev spaces:
\begin{align}
\label{H1c}
H^1_c(\rr,\rr^d) &= \bigl\{ w \in \HoneLoc(\rr,\rr^d) : \text{ the functions } \xi\mapsto e^{\frac{1}{2}c\xi} w(\xi) \\
\nonumber
&\qquad\qquad\text{and } \xi\mapsto e^{\frac{1}{2}c\xi} w'(\xi) \text{ are in } L^2(\rr,\rr^d)\bigr\}\,, \quad \text{and} \\
\label{H1ce}
\HoneOfTwo{c}{e} &= \bigl\{ w \in \HoneLoc(\rr,\rr^d) : \text{ the function } \xi\mapsto w(\xi) - e \text{ is in } H^1_c(\rr,\rr^d) \bigr\}
\,.
\end{align}
\begin{definition}[energy in a frame travelling at the speed $c$]
\label{def:eee_c}
For every $w$ in $H^1_{c,e}(\rr,\rr^d)$, let us call \emph{energy (Lagrangian) of $w$ in the frame travelling as speed $c$}, and let us denote by $\eee_c[w]$ the quantity defined by the integral:
\begin{equation}
\label{def_eee_c}
\eee_c[w] = \int_{\rr}e^{c\xi}\left(\frac{1}{2} w'(\xi)^2 + V\bigl(w(\xi)\bigr)\right)\, d\xi
\,.
\end{equation}
\end{definition}
\begin{remark}
Everywhere in the paper (including in the definition above), if $u$ is a vector of $\rr^d$, the square of the euclidean norm $\abs{u}^2$ of $u$ is simply denoted by $u^2$. 
\end{remark}
For $w$ in $H^1_{c,e}(\rr,\rr^d)$, the integral to the right hand of \cref{def_eee_c} converges at the right end of $\rr$, and according to assumption \cref{hyp_coerc}, it either converges to the left end of $\rr$ or is equal to $+\infty$. Thus the expression \cref{def_eee_c} defines a functional: $H^1_{c,e}(\rr,\rr^d)\to\rr\cup\{+\infty\}$. For $w$ in $H^1_{c,e}(\rr,\rr^d)\cap\HoneulofR$ (in this case $w$ is bounded), the same integral also converges at the left end of $\rr$ and $\eee_c[w]$ is in $\rr$. The following proposition, proved in \cref{subsec:proof_time_derivative_energy_travelling_frame}, shows that this energy \cref{def_eee_c} defines a Lyapunov function for the solutions of system \cref{parabolic_system_trav_frame} that belong to $H^1_{c,e}(\rr,\rr^d)$. Recall that $(S_t)_{t\ge0}$ denotes the semi-flow of this system in $\HoneulofR$. 
\begin{proposition}[time derivative of energy in a travelling frame]
\label{prop:decrease_energy_trav_frame}
For every positive quantity $c$ and every solution $(x,t)\mapsto u(x,t)$ of the parabolic system \cref{parabolic_system}, if the initial condition $x\mapsto u(x,t=0)$ is in $\HoneOfTwo{c}{e}$ (in addition to being in $\HoneulofR$), then the same is true for the profile $x\mapsto u(x,t)$ of the solution at every nonnegative time $t$. In this case, the function $t\mapsto \eee_c[v(\cdot,t)]$ (for the function $v(\cdot,\cdot)$ defined in \cref{parabolic_system_trav_frame} and the functional $\eee_c[\cdot]$ defined in \cref{def_eee_c}) is continuous on $[0,+\infty)$, differentiable on $(0,+\infty)$, and, for every positive time $t$, the integral 
\begin{equation}
\label{def_Dc_of_t}
D_c(t) = \int_{\rr}e^{c\xi} v_t(\xi,t)^2 \, d\xi
\end{equation}
is finite, and the following equality holds:
\begin{equation}
\label{decrease_energy_trav_frame}
\frac{d}{dt}\eee_c[v(\cdot,t)] = - D_c(t)
\,.
\end{equation}
In addition, for every nonnegative time $t$, the restriction of the map $S_t$ (the semi-flow at time $t$) to the space $H^1_{c,e}(\rr,\rr^d)\cap\HoneulofR$ defines a continuous map from this space to itself, for the sum of the $H^1_{c,e}(\rr,\rr^d)$-norm and the $\HoneulofR$-norm. 
\end{proposition}
\subsection{Variational structure in travelling frames}
\label{subsec:variational_structure_trave_frames}
\subsubsection{Infimum of the energy in a travelling frame}
\label{subsubsec:infimum_energy_trav_frame}
For every positive quantity $c$, let us consider the quantity
\[
\mathfrak{I}(c) = \inf_{w\in\HoneOfTwo{c}{e}} \eee_c[w]
\,;
\]
and, for every function $w$ in $\HoneOfTwo{c}{e}$, and every real quantity $\xi_0$, let us consider the function $T_{\xi_0}w$ defined as
\[
T_{\xi_0}w(\xi) = w(\xi-\xi_0)
\,.
\]
It follows from the equality
\[
\eee_c[T_{\xi_0}w] = e^{c\xi_0}\eee_c[w]
\] 
that the quantity $\mathfrak{I}(c)$ is equal to either $0$ or $-\infty$. In other words, the subsets $\ccc_{-\infty}$ and $\ccc_0$ of $(0,+\infty)$, defined as
\begin{equation}
\label{def_ccc_minus_infty_ccc_0}
\ccc_{-\infty} = \{c\in(0,+\infty): \mathfrak{I}(c) = -\infty \}
\quad\text{and}\quad
\ccc_0 = \{c\in(0,+\infty): \mathfrak{I}(c) = 0 \}
\end{equation}
are complementary subsets of $(0,+\infty)$; with symbols,
\[
\ccc_{-\infty} \sqcup \ccc_0 = (0,+\infty)
\,.
\]
\subsubsection{Lower quadratic hull of the potential at the invaded critical point}
\label{subsubsec:lower_quadratic_hull}
According to assumptions \cref{hyp_coerc} and \cref{hyp_crit_point}, the set
\[
\left\{\mu\in\rr: \text{ for every $u$ in $\rr^d$, } V(u)\ge\frac{1}{2}\mu (u-e)^2 \right\}
\]
is nonempty, and bounded from above by $\mu_1$. Let us denote by $\muQuadHull$ the supremum of this set (in this notation, the index ``quad-hull'' refers to ``the curvature of the lower quadratic hull'' of the graph of $V$, centred at $(e,0)$), see \cref{fig:graph_of_V_notation_muQuadHull}. 
\begin{figure}[!htbp]
\centering
\includegraphics[width=.8\textwidth]{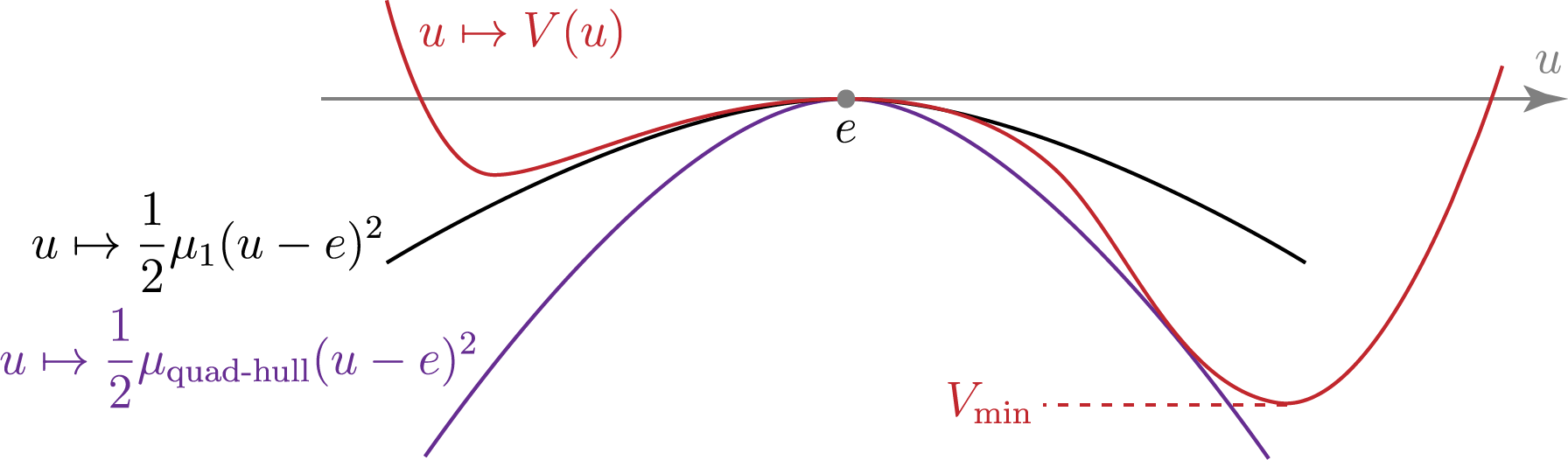}
\caption{Illustration of the notation $\mu_1$, $\muQuadHull$, and $\Vmin$ on the graph of the potential $V$. On this figure $\mu_1$ is negative, but in the article this quantity is not necessarily even nonpositive; and $\muQuadHull$ is less than $\mu_1$, which is a necessary (not sufficient) condition in order $\cNonLinMax$ to exceed $\cLinMax$ (inequality \cref{cLinMax_le_cNonLinMax_le_cQuadHull}).}
\label{fig:graph_of_V_notation_muQuadHull}
\end{figure}
Then, according to this definition,
\begin{equation}
\label{V_greater_than_lower_quadratic_hull}
-\infty<\muQuadHull\le\mu_1 
\quad\text{and, }\text{for every $u$ in $\rr^d$,}\quad 
V(u)\ge\frac{1}{2}\muQuadHull (u-e)^2
\,,
\end{equation}
and since according to \cref{hyp_crit_point} the quantity $\Vmin$ is negative, it follows that
\[
\muQuadHull < 0
\,.
\]
Let us write 
\[
\cQuadHull = 2\sqrt{-\muQuadHull}\,,
\quad\text{so that}\quad
\muQuadHull = - \frac{\cQuadHull^2}{4}
\quad\text{and}\quad
\cQuadHull \ge \cLinMax
\,.
\]
\subsubsection{Maximal nonlinear invasion speed}
\label{subsubsec:maximal_nonlinear_invasion_speed}
Some basic properties of the sets $\ccc_{-\infty}$ and $\ccc_0$ are stated in \cref{prop:basic_properties_variational_structure,cor:ccc_minus_infty_nonempty} (\cref{subsec:basic_properties_variational_structure}). Three of these properties are:
\begin{equation}
\label{subsets_ccc_minus_infty_and_ccc_0}
\ccc_{-\infty}\not=\emptyset
\quad\text{and}\quad
(0,\cLinMax) \subset \ccc_{-\infty} 
\quad\text{and}\quad
[\cQuadHull,+\infty) \subset \ccc_0
\,;
\end{equation}
notice that the quantity $\cLinMax$ may vanish, so that the first of these properties is not a consequence of the second. These properties set the ground for the next definition. 
\begin{definition}[maximal nonlinear invasion speed]
\label{def:max_nonlin_invasion_speed}
Let us call \emph{maximal nonlinear invasion speed} of the critical point $e$ the quantity $\cNonLinMax$ defined as
\begin{equation}
\label{def_maximal_nonlinear_invasion_speed}
\cNonLinMax = \sup(\ccc_{-\infty})
\end{equation}
(this quantity is noted as $c^\dag$ in \cite[Theorem 1.1]{LuciaMuratovNovaga_existTWInvasionGLCylinders_2008}).
\end{definition}
It follows from this definition and from the properties \cref{subsets_ccc_minus_infty_and_ccc_0} that
\begin{equation}
\label{cLinMax_le_cNonLinMax_le_cQuadHull}
0\le\cLinMax\le\cNonLinMax \le \cQuadHull
\quad\text{and}\quad
0<\cNonLinMax
\,.
\end{equation}
It will turn out, as a consequence of \cref{thm:main,thm:characterization_existence_pushed_front} below (the main results of this paper) that the two sets $\ccc_{-\infty}$ and $\ccc_0$ are actually nothing but the intervals $(0,\cNonLinMax)$ and $[\cNonLinMax,+\infty)$ (see \cref{cor:full_properties_ccc_minus_infty_and_ccc_0,fig:line_of_speeds} below); this additional information will not be used until then. 
\subsubsection{Decay speed and variational invasion speed of a solution}
\label{subsubsec:decay_variational_invasion_speeds}
Let $(x,t)\mapsto u(x,t)$ denote a solution of system \cref{parabolic_system} (in the space $\HoneulofR$). For every nonnegative time $t$, let us consider (following \cite[Definition 4.5, Proposition 4.6, Theorem 4.7]{Muratov_globVarStructPropagation_2004}) the sets $\cccDecay(t)$ and $\cccVar(t)$ defined as
\[
\begin{aligned}
\cccDecay(t)  &= \bigl\{c\in(0,+\infty): u(\cdot,t)\in \HoneOfTwo{c}{e}\bigr\} \, \\
\text{and}\quad
\cccVar(t) &= \bigl\{c\in(0,+\infty): u(\cdot,t)\in \HoneOfTwo{c}{e} \text{ and }\eee_c[u(\cdot,t)]<0\bigr\} 
\,,
\end{aligned}
\]
so that $\cccVar(t)$ is a subset of $\cccDecay(t)$. According to the definition \cref{def_maximal_nonlinear_invasion_speed} of the maximal nonlinear invasion speed $\cNonLinMax$, 
\[
\cccVar(t) = \bigl\{c\in(0,\cNonLinMax): u(\cdot,t)\in \HoneOfTwo{c}{e} \text{ and }\eee_c[u(\cdot,t)]<0\bigr\} \subset (0,\cNonLinMax)
\,.
\]
Let us consider the suprema $\sup\bigl(\cccDecay(t)\bigr)$ and $\sup\bigl(\cccVar(t)\bigr)$ of the two sets $\cccDecay(t)$ and $\cccVar(t)$, with the convention that each supremum equals $0$ if the corresponding set is empty. According to \cref{prop:decrease_energy_trav_frame}, both sets are non-decreasing for inclusion with respect to $t$, so that both suprema are non-decreasing. This leads to the following definition. 
\begin{definition}[decay speed, variational invasion speed]
\label{def:decay_variational_invasion_speed}
Let us call, respectively, \emph{decay speed} and \emph{variational invasion speed} of the solution $u$ the quantities $\cDecay[u]$ and $\cVar[u]$ defined as
\begin{equation}
\label{decay_variational_invasion_speed}
\cDecay[u] = \lim_{t\to+\infty} \sup\bigl(\cccDecay(t)\bigr)
\quad\text{and}\quad
\cVar[u] = \lim_{t\to+\infty} \sup\bigl(\cccVar(t)\bigr)
\,.
\end{equation}
\end{definition}
It follows from this definition that
\begin{equation}
\label{basic_inequalities_cVar_cDecay}
0\le \cVar[u]\le \cNonLinMax
\quad\text{and}\quad
\cVar[u]\le \cDecay[u] \le +\infty
\,.
\end{equation}
Similarly, if $w$ is a function in $\HoneulofR$ let us denote by $\cDecay[w]$ and $\cVar[w]$ the quantities $\cDecay[u]$ and $\cVar[u]$ defined by the solution $u$ of system \cref{parabolic_system} for the initial condition $w$ at time $0$. 
The following proposition, proved in \cref{subsubsec:lower_semi_continuity_variational_invasion_speed}, is a direct consequence of the definition of the variational invasion speed. 
\begin{proposition}[lower semi-continuity of the variational invasion speed]
\label{prop:lower_semi_continuity_variational_invasion_speed}
For every $c$ in $(\cNonLinMax,+\infty)$, the map
\[
\HoneulofR\cap\HoneOfTwo{c}{e} \to [0,\cNonLinMax] \,,
\quad
w\mapsto \cVar[w]
\]
is lower semi-continuous with respect to the topology induced by the sum of the $\HoneulofR$-norm and the $H^1_{c}(\rr,\rr^d)$-norm.
\end{proposition}
\begin{remark}
For a stronger version of this result (lower semi-continuity of the invasion speed with respect to a wider class of functions) when the invaded critical point $e$ is a nondegenerate local minimum point of $V$, see \cite[Theorem~2]{Risler_globCVTravFronts_2008}.
\end{remark}
\subsection{Main results}
\label{subsec:main_results}
Let $V$ denote a potential function in $\ccc^2(\rr^d,\rr)$ and $e$ denote a critical point of $V$, and let us assume that assumptions \cref{hyp_coerc} and \cref{hyp_crit_point} hold. The following two statements call upon:
\begin{itemize}
\item the notation $\cLinMax$ (\cref{def:max_linear_invasion_speed}) and $\cNonLinMax$ (\cref{def:max_nonlin_invasion_speed}) denoting, respectively, the maximal linear invasion speed and the maximal nonlinear invasion speed of the critical point $e$, 
\item the notation $\cDecay[u]$ and $\cVar[u]$ (\cref{def:decay_variational_invasion_speed}) denoting, respectively, the decay speed and the variational speed of a solution $u$ of system \cref{parabolic_system},
\item \Cref{def:pushed_travelling_wave_front} of a pushed travelling front invading $e$, and \cref{def:invasion_through_profiles_pushed_fronts} of invasion through profiles of such pushed fronts.
\end{itemize}
The first of these two statements (\cref{thm:main}) is the main result of this paper, and the second (\cref{thm:characterization_existence_pushed_front}, proved in \cref{subsec:proof_cor_main}), is mainly a consequence of \cref{thm:main}.
\begin{theorem}[global convergence towards pushed fronts]
\label{thm:main}
Every solution $u$ of the parabolic system (1.1) satisfying the condition 
\begin{equation}
\label{thm_main_cLinMax_smaller_than_cVar_smaller_than_cDecay}
\cLinMax < \cVar[u] < \cDecay[u]
\end{equation}
invades the critical point $e$ at its variational speed $\cVar[u]$ through profiles of pushed fronts. 
\end{theorem}
\begin{theorem}[existence of a pushed front invading $e$ at a speed greater than the maximal linear invasion speed]
\label{thm:characterization_existence_pushed_front}
The following two conditions are equivalent:
\begin{enumerate}
\item $\cLinMax < \cNonLinMax$;
\label{item:thm_characterization_existence_pushed_front_cLinMax_le_cNonLinMax}
\item there exists a pushed front invading $e$ at a speed which is greater than $\cLinMax$.
\label{item:thm_characterization_existence_pushed_front_existence}
\end{enumerate}
Moreover, if these conditions hold then the following two additional conclusions also hold:
\begin{enumerate}
\setcounter{enumi}{2}
\item there exists a pushed front invading $e$ at the speed $\cNonLinMax$;
\label{item:thm_characterization_existence_pushed_front_existence_speed_cNonLinMax}
\item the profile of every pushed front invading $e$ at the speed $\cNonLinMax$ is a global minimizer of the energy $\eee_{\cLinMax}[\cdot]$ in $\HoneOfTwo{\cNonLinMax}{e}$. 
\label{item:thm_characterization_existence_pushed_front_global_minimizer}
\end{enumerate}
\end{theorem}
\begin{remarks}
\begin{enumerate}
\item A sufficient condition for the condition \cref{thm_main_cLinMax_smaller_than_cVar_smaller_than_cDecay} of \cref{thm:main} to hold is: there exist speeds $c$ and $c'$ such that
\[
\cLinMax \le c 
\quad\text{and}\quad
\cNonLinMax < c'
\quad\text{and}\quad
\eee_c[u(\cdot,0)]< 0
\quad\text{and}\quad
u(\cdot,0)\in \HoneOfTwo{c'}{e}
\,.
\]
This condition is, however, more demanding that condition \cref{thm_main_cLinMax_smaller_than_cVar_smaller_than_cDecay}, especially concerning the rate at which the initial condition $u(\cdot,0)$ approaches $e$ at the right end of space. 
\item The key (and costly, if $\mu_1$ is negative or equivalently $\cLinMax$ is positive) assumption of \cref{thm:main} is the inequality $\cLinMax < \cVar[u]$ of \cref{thm_main_cLinMax_smaller_than_cVar_smaller_than_cDecay}, which ensures that invasion occurs at a speed which is greater than the maximal linear invasion speed $\cLinMax$, and implicitly requires the strict inequality $\cLinMax < \cNonLinMax$ (see \cref{fig:line_of_speeds}). 
\item By contrast, if $\mu_1$ is nonnegative (or equivalently if $\cLinMax$ equals $0$), then it follows from \cref{prop:sufficient_condition_invasion} that the condition \cref{thm_main_cLinMax_smaller_than_cVar_smaller_than_cDecay} is satisfied for a large set of solutions, and it follows from \cref{thm:characterization_existence_pushed_front} that there exists (at least) one (pushed) travelling front invading $e$ at the (positive) speed $\cNonLinMax$. This generalizes (in particular) \cite[Corollary~1]{Risler_globCVTravFronts_2008}, \cite[Theorem~2]{AlikakosKatzourakis_heteroclinicTW_2011}, and \cite[Theorem~1]{OliverBonafoux_heteroclinicTW1dParabSystDegenerate_2021}.
\label{item:remark_main_corollary_mu1_nonnegative}
\item The second inequality of \cref{thm_main_cLinMax_smaller_than_cVar_smaller_than_cDecay} is required for the variational arguments involved in the proof. This condition is stronger than the one required in the setting of parabolic scalar equations \cite{Rothe_convergenceToPushedFronts_1981}, where the speed $c$ of the (unique) pushed front is a priori known, and global convergence towards this pushed front only requires that the initial condition approach the critical point at an exponential rate which is larger than $\abs{\lambda_{c,+}(\mu_1)}$, \cite[condition~$(\varphi 4)$]{Rothe_convergenceToPushedFronts_1981}. Unfortunately \cref{thm:main} says nothing concerning the behaviour of solutions for which the first inequality of \cref{thm_main_cLinMax_smaller_than_cVar_smaller_than_cDecay} is fulfilled, but the profile is say only in $\HoneOfTwo{\abs{\lambda_{\cVar[u],+}(\mu_1)}}{e}$. To the best knowledge of the authors, this is an open question.
\item \Cref{thm:main} only deals with convergence towards pushed front travelling at speeds that are greater than the maximal linear invasion speed $\cLinMax$, and \cref{thm:characterization_existence_pushed_front} deals with the existence of those pushed fronts only. Pushed fronts travelling at speeds not greater than $\cLinMax$ may exist for certain potentials (an easy way to build such an example is to consider two uncoupled scalar equations, see for instance \cite{JolyOliverBRisler_genericTransvPulledPushedTFParabGradSyst_2023}), but as already mentioned in \cref{subsubsec:genericity_counting_arguments}, they do not exist for a generic potential, \cite{JolyOliverBRisler_genericTransvPulledPushedTFParabGradSyst_2023}. 
\item \Cref{thm:characterization_existence_pushed_front} is the analogue, in the simpler setting of a spatial domain equal to $\rr$ considered here, of \cite[Theorem~1.1]{LuciaMuratovNovaga_existTWInvasionGLCylinders_2008} which is concerned with gradient systems in infinite cylinders (see also \cite[Theorem~2.8]{LuciaMuratovNovaga_linNonlinSelectionPropSpeedInvasionUnstableEquil_2004} for scalar parabolic equations and \cite[Theorem~3.3]{MuratovNovaga_frontPropIVariational_2008} for scalar equations in cylinders). The parameter $\nu_0$ of \cite{LuciaMuratovNovaga_existTWInvasionGLCylinders_2008,MuratovNovaga_frontPropIVariational_2008} is the ``cylinder'' analogue of the quantity denoted here by $\mu_1$.
\label{item:remark_main_corollary_Muratov_related_results}
\end{enumerate}
\end{remarks}
\subsection{Implications on the variational structure in travelling frames}
The following corollary, proved in \cref{subsec:proof_cor_full_properties_ccc_minus_infty_and_ccc_0}, is a direct consequence of \cref{thm:main,thm:characterization_existence_pushed_front}. 
\begin{corollary}[variational structure, full picture]
\label{cor:full_properties_ccc_minus_infty_and_ccc_0}
The following equalities hold:
\begin{equation}
\label{full_properties_ccc_minus_infty_and_ccc_0}
\ccc_{-\infty} = (0,\cNonLinMax)
\,,\quad\text{or equivalently}\quad
\ccc_0 = [\cNonLinMax,+\infty)
\,.
\end{equation}
In addition, if $\cNonLinMax$ is larger than $\cLinMax$ (in particular if $\cLinMax$ equals $0$), then
\begin{equation}
\label{cor_full_properties_cNonLinMax_less_than_cQuadHull}
\cNonLinMax < \cQuadHull
\,,
\end{equation}
and in this case $\cNonLinMax$ is the only speed $c$ in $(0,+\infty)$ for which the energy $\eee_c[\cdot]$ has a global minimizer in $\HoneOfTwo{c}{e}$ which is not identically equal to $e$. 
\end{corollary}
Thus there are exactly four possible configurations for the respective positions of the quantities $0$, $\cLinMax$, $\cNonLinMax$, and $\cQuadHull$ on the real line (see \cref{fig:line_of_speeds}):
\begin{enumerate}
\item $0 = \cLinMax < \cNonLinMax < \cQuadHull$,
\label{item:four_cases_marginally_stable_invaded}
\item $0 < \cLinMax = \cNonLinMax = \cQuadHull$, 
\label{item:four_cases_KPP}
\item $0 < \cLinMax = \cNonLinMax < \cQuadHull$,
\label{item:four_cases_not_enough_pushed}
\item $0 < \cLinMax < \cNonLinMax < \cQuadHull$,
\label{item:four_cases_strict_inequalities}
\end{enumerate}
and, according to \cref{thm:characterization_existence_pushed_front}, there exists a pushed front invading $e$ at a speed larger than $\cLinMax$ if and only if $\cNonLinMax$ is larger than $\cLinMax$, that is in cases \cref{item:four_cases_marginally_stable_invaded,item:four_cases_strict_inequalities}. 
\begin{figure}[!htbp]
\centering
\includegraphics[width=.8\textwidth]{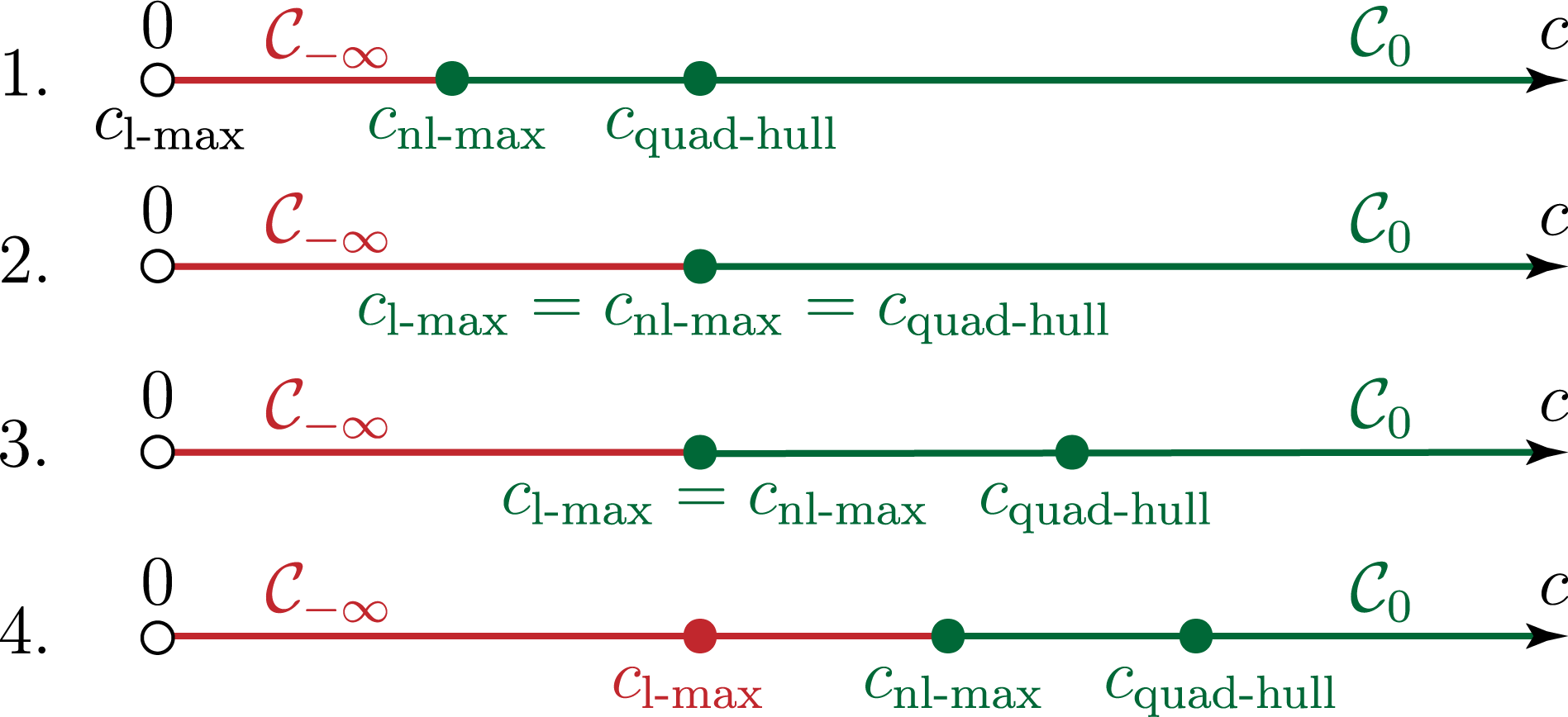}
\caption{Half-line of positive speeds, complementary intervals $\ccc_{-\infty}$ and $\ccc_0$, and values of $\cLinMax$ and $\cNonLinMax$ and $\cQuadHull$ in the four possible cases \cref{item:four_cases_marginally_stable_invaded,item:four_cases_KPP,item:four_cases_not_enough_pushed,item:four_cases_strict_inequalities} that are compatible with inequalities \cref{cLinMax_le_cNonLinMax_le_cQuadHull} and the positivity of $\cNonLinMax$ and conclusion \cref{cor_full_properties_cNonLinMax_less_than_cQuadHull} of \cref{cor:full_properties_ccc_minus_infty_and_ccc_0}. The three cases \cref{item:four_cases_KPP,item:four_cases_not_enough_pushed,item:four_cases_strict_inequalities} are encountered in the scalar example (Fisher's model) discussed in \cref{sec:Fischer_model}.}
\label{fig:line_of_speeds}
\end{figure}
\subsection{Short historical review}
\label{subsec:short_historical_review}
Global convergence towards pulled travelling fronts for scalar equations was first established in the celebrated work of Kolmogorov, Petrovskii, and Piskunov \cite{KolmogorovPetrovskii_diffusionEquation_1937}. The adjectives ``pulled'' and ``pushed'' were introduced by Stokes in \cite{Stokes_twoTypesMovingFrontQuasilinearDiffusion_1976}. Concerning global convergence towards pushed travelling fronts for scalar equations, the first results were obtained by Kanel \cite{Kanel_problemsBurningTheoryEquations_1961,Kanel_stabSolCauchyPbEquCombustion_1962} in the ``combustion'' (``ignition'') case. Fife and McLeod proved the global stability of fronts propagating into stable equilibria (``bistable fronts'', which can be seen as a particular class of pushed fronts) \cite{FifeMcLeod_approachTravFront_1977} and of stacked families of such bistable fronts \cite{FifeMcLeod_phasePlaneDisc_1981}. Still in the scalar case, global convergence towards general pushed fronts was proved by Stokes \cite{Stokes_twoTypesMovingFrontQuasilinearDiffusion_1976} and Rothe \cite{Rothe_convergenceToPushedFronts_1981}, and extended to the setting of cylinders by Roquejoffre \cite{Roquejoffre_eventualMonotonicityCVTravFrontCylinders_1997}. In all those references, proofs of global convergence rely on comparison principles, and the main result of this paper (\cref{thm:main} above) can be seen as an extension to systems of \cite[Theorem~1]{Rothe_convergenceToPushedFronts_1981}, where maximum principles are replaced with variational arguments. Still in the scalar case, a minmax expression of the minimal speed of monotone fronts invading a critical point, and therefore a characterization of the nature (pushed or pulled) of the corresponding front was provided by Hadeler and Rothe in \cite{HadelerRothe_travellingFrontsNonlinearDiffEqu_1975}. For a broader picture and a thorough review of experimental observations of invasion processes across the sciences, see \cite{VanSaarloos_frontPropagationUnstableStates_2003}, and for a deeper explanation of the difference between pushed and pulled travelling fronts, see \cite{GarnierGilettiHamelRoques_insideDynPulledPushed_2012}. 

For parabolic gradient \emph{systems} as those considered in this article (when the dimension $d$ exceeds $1$), maximum principles do not hold any more in general. However, many of the global stability results known in the scalar case can still be recovered by variational methods for such systems. The fundamental observation underlying the proofs of such extensions is the fact that a variational structure (an energy decreasing with time, at least formally) exists not only in standing frames, but also in frames travelling at any constant velocity. This fact is known for a long time, and was used for instance by Fife and McLeod in their proof of the global stability of bistable fronts in the scalar case \cite{FifeMcLeod_approachTravFront_1977} and by Roquejoffre \cite{Roquejoffre_cVTravWavesSemilinearParabEqu_1994}. However, attempts to fully embrace the implications of this rich variational structure are more recent, and originated with the works of Heinze \cite{Heinze_variationalApproachTW_2001}, and especially of Muratov and his collaborators \cite{Muratov_globVarStructPropagation_2004,LuciaMuratovNovaga_linNonlinSelectionPropSpeedInvasionUnstableEquil_2004,LuciaMuratovNovaga_existTWInvasionGLCylinders_2008,MuratovNovaga_frontPropIVariational_2008,MuratovNovaga_globExpConvTW_2012,MuratovZhong_thresholdSymSol_2013,MuratovZhong_thresholdPhenSymDecrRadialSolRDEquations_2017} (by the way, as mentioned in remark \cref{item:remark_main_corollary_Muratov_related_results} above, \cref{thm:characterization_existence_pushed_front} is essentially contained in \cite[Theorem~1.1]{LuciaMuratovNovaga_existTWInvasionGLCylinders_2008}, and conclusion \cref{item:prop_basic_properties_variational_structure_inclusion_ccc_0} of \vref{prop:basic_properties_variational_structure} is a reformulation of \cite[Theorem~3.7]{Muratov_globVarStructPropagation_2004}). Pushing further these ideas, Gallay and the second author proved global convergence results towards travelling fronts invading stable equilibria for parabolic systems of the form \cref{parabolic_system} \cite{GallayRisler_globStabBistableTW_2007,Risler_globCVTravFronts_2008}, and a rather comprehensive description of the asymptotic behaviour of solutions that are stable at both ends of $\rr$ (``bistable'' solutions) was obtained by the second author, for parabolic systems \cite{Risler_globalRelaxation_2016,Risler_globalBehaviour_2016}, for their hyperbolic analogues \cite{Risler_globalBehaviourHyperbolicGradient_2017}, and for radially symmetric solutions of parabolic systems in higher space dimension \cite{Risler_globalBehaviourRadiallySymmetric_2017,Risler_noInvasionCaseHigherSpace_2020}, under generic assumptions on the potential $V$ \cite{JolyRisler_genericTransversalityTravStandFrontsPulses_2023,Risler_genericTransvRadSymStatSol_2023}. In the meanwhile, the same variational structure has been successfully applied to a broader range of settings: harmonic heat flow \cite{BertschMuratovPrimi_tWSolHarmHeatFlow_2006}, heterogeneous environments \cite{BouhoursNadin_variationalApproachRDForcedSpeedDim1_2015}, FitzHugh--Nagumo system \cite{ChenChoi_travPulseSolutionsToFHNequs_2015,ChenChenHuang_TWFitzHughNagumoInfiniteChannel_2016}, two-dimensional heteroclinic travelling waves \cite{OliverBonafoux_TWparabAllenCahn_2021,ChenChienHuang_varApproach3PhaseTWgradSyst_2021}. 

In the setting of scalar hyperbolic equations, a set of new ideas and techniques was introduced by Gallay and Joly to derive from the same gradient structure the global stability of bistable travelling waves \cite{GallayJoly_globStabDampedWaveBistable_2009}. Their approach turns out to be especially relevant to prove global convergence towards pushed travelling fronts, as was shown by Luo \cite{Luo_globStabDampedWaveEqu_2013} still in the same setting of scalar hyperbolic equations. It is the same set of ideas and techniques, adapted to parabolic systems, that are the main building blocks of the proof of \cref{thm:main} provided here.
 
The initial motivation for this work is a recent result of the first author \cite{OliverBonafoux_heteroclinicTW1dParabSystDegenerate_2021} about the existence of travelling waves connecting degenerate minimum sets of the potential $V$ for system \cref{parabolic_system}, proved by a completely different approach, which extends the method introduced by Alikakos and Katzourakis in \cite{AlikakosKatzourakis_heteroclinicTW_2011} to curves taking values in an infinite-dimensional Hilbert space, in the spirit of earlier works by Monteil and Santambrogio \cite{MonteilSantambrogio_metricMethodsHetConnectInfDimHilbert_2020} and Smyrnelis \cite{Smyrnelis_connectOrbitsHilbertSpacesApplPDE_2020}. As a matter of fact, as already mentioned in remark \cref{item:remark_main_corollary_mu1_nonnegative} above, \cref{thm:characterization_existence_pushed_front} extends the existence part of \cite[Theorem~1]{OliverBonafoux_heteroclinicTW1dParabSystDegenerate_2021}; more refined results in the specific setting of propagation into degenerate minimum sets will be provided in the forthcoming work \cite{OliverBonafouxRisler_invasionMinSetParGradSyst_2023}. 
\section{Preliminaries}
\label{sec:preliminaries}
Let us consider a potential function $V$ in $\ccc^2(\rr^d,\rr)$ and a critical point $e$ of $V$, satisfying assumptions \cref{hyp_coerc} and \cref{hyp_crit_point}. 
\subsection{Global existence of solutions and regularization}
\label{subsec:global_existence_solutions_regularization}
Among various possible choices for the functional space where the semi-flow of system \cref{parabolic_system} can be considered, the space $\HoneulofR$ (see for instance \cite{GallaySlijepcevic_energyFlowFormallyGradient_2001,ArrietaRodriguezBernal_linearParabEquLocUnifSpaces_2004}) fits well with the purpose of this article and the variational methods involved in the proofs: it contains bounded solutions (among which travelling fronts) and the regularity of its functions allows to consider the energy of a solution from time zero (\cref{prop:decrease_energy_trav_frame}). The following proposition is standard (for a proof see for instance \cite{Risler_globCVTravFronts_2008,Risler_globalRelaxation_2016}).
\begin{proposition}[global existence and uniform bound on solutions]
\label{prop:attr_ball}
For every function $u_0$ in $\HoneulofR$, system \cref{parabolic_system} has a unique globally defined solution $t\mapsto S_t u_0$ in $\ccc^0\bigl([0,+\infty),\HoneulofR\bigr)$ with initial condition $u_0$. In addition, the quantity
\[
\limsup_{t\to+\infty} \norm{x\mapsto(S_t u_0)(x)}_{\Linfty}
\]
is bounded from above by a quantity depending only on $V$. 
\end{proposition}
In addition, the parabolic system \cref{parabolic_system} has smoothing properties (Henry \cite{Henry_geomSemilinParab_1981}). Due to these properties, since $V$ is of class $\ccc^2$ and thus the nonlinearity $\nabla V$ is of class $\ccc^1$, for every quantity $\alpha$ in the interval $(0,1)$, every solution $t\mapsto S_t u_0$ in $\ccc^0\bigl([0,+\infty),\HoneulofR\bigr)$ actually belongs to
\[
\ccc^0\left((0,+\infty),\cccb{2,\alpha}\right)\cap \ccc^1\left((0,+\infty),\cccb{0,\alpha}\right),
\]
and, for every positive quantity $\varepsilon$, the quantities
\begin{equation}
\label{bound_u_ut_ck_preliminaries}
\sup_{t\ge\varepsilon}\norm{S_t u_0}_{\cccb{2,\alpha}}
\quad\text{and}\quad
\sup_{t\ge\varepsilon}\norm{\frac{d(S_t u_0)}{dt}(t)}_{\cccb{0,\alpha}}
\end{equation}
are finite. In addition, there exists a quantity $\Ratt$ (radius of an attracting ball for the $\ccc^1$-norm), depending only on $V$, such that, for every large enough positive time $t$, 
\begin{equation}
\label{att_ball_C1_norm}
\norm{S_t u_0}_{\ccc^1(\rr,\rr^d)} \le \Ratt
\,.
\end{equation}
\subsection{Asymptotic compactness}
\label{subsec:compactness}
The next lemma follows from the bounds \cref{bound_u_ut_ck_preliminaries} above.
\begin{lemma}[asymptotic compactness]
\label{lem:compactness}
For every solution $(x,t)\mapsto u(x,t)$ of system \cref{parabolic_system}, and for every sequence $(x_n,t_n)_{n\in\nn}$ in $\rr\times[0,+\infty)$ such that $t_n\to+\infty$ as $n\to+\infty$, there exists a entire solution $u_\infty$ of system \cref{parabolic_system} in 
\[
\ccc^0\left(\rr,\cccb{2}\right)\cap \ccc^1\left(\rr,\cccb{0}\right)
\,,
\]
such that, up to replacing the sequence $(x_n,t_n)_{n\in\nn}$ by a subsequence, 
\begin{equation}
\label{compactness}
D^{2,1}u(x_n+\cdot,t_n+\cdot)\to D^{2,1}u_\infty
\quad\text{as}\quad
n\to+\infty
\,,
\end{equation}
uniformly on every compact subset of $\rr^2$, where the symbol $D^{2,1}v$ stands for $(v,v_x,v_{xx},v_t)$ (for $v$ equal to $u$ or $u_\infty$). 
\end{lemma}
\begin{proof}
See \cite[1963]{MatanoPolacik_entireSolutionBistableParabEquTwoCollidingPulses_2017} or the proof of \cite[\GlobalRelaxationLemAsymptCompactness]{Risler_globalRelaxation_2016}. 
\end{proof}
\subsection{Invaded critical point at the origin of \texorpdfstring{$\rr^d$}{Rd}}
\label{subsec:invaded_crit_point_at_origin}
For convenience, it will be assumed all along the current \cref{sec:preliminaries} (and along most of the next \cref{sec:proof}) that
\begin{equation}
\label{e_equals_origin}
e = 0_{\rr^d}
\,.
\end{equation}
This assumption amounts to replacing the initial potential function $u\mapsto V(u)$ by the ``new'' potential function $u\mapsto V(e+u)$, and it can be made without loss of generality; indeed, even if assumption \cref{hyp_coerc} is not necessarily satisfied by this new potential, this assumption will not be directly used in the proof: only its consequences (the global existence of solutions and their asymptotic compactness stated in the two \namecrefs{subsec:compactness} above) will, and these consequences still hold after a translation in the state variable $u$. 
\subsection{Time derivative of energy in a travelling frame}
\label{subsec:proof_time_derivative_energy_travelling_frame}
The aim of this \namecref{subsec:proof_time_derivative_energy_travelling_frame} is to prove \cref{prop:decrease_energy_trav_frame}. 
For every positive quantity $c$, following the notation $H^1_c(\rr,\rr^d)$ introduced in \cref{H1c}, let us introduce the following weighted Sobolev spaces:
\begin{align}
\nonumber
L^2_c(\rr,\rr^d) &= \bigl\{w\in \LtwoLoc(\rr,\rr^d): \text{ the function } \xi\mapsto e^{\frac{1}{2}c\xi} \, w(\xi) \text{ is in }  L^2(\rr,\rr^d) \}\,, \\
\label{H2c}
H^2_c(\rr,\rr^d) &= \bigl\{w\in H^1_c(\rr,\rr^d):    \text{ the function } \xi\mapsto e^{\frac{1}{2}c\xi} \, w''(\xi) \text{ is in } L^2(\rr,\rr^d)\bigr\}
\,.
\end{align}
\begin{proof}[Proof of \cref{prop:decrease_energy_trav_frame}]
The proof follows from standard results of analytic semi-group theory, see \cite{Lunardi_analyticSemigroupsOptimalRegularityParabolicPbs_1995}, and similar statements in related settings can be found in the literature, see for instance \cite[Proposition~4.1]{Muratov_globVarStructPropagation_2004}, \cite[Proposition~5.1]{MuratovNovaga_frontPropIVariational_2008}, and \cite[Proposition~3.1]{MuratovNovaga_globExpConvTW_2012}. The setting considered in these two last references (scalar equations in cylindrical domains) differs from the one considered here, however the semi-group arguments proving the result are unchanged. 

Here are some elements of the proof. The operator 
\[
\partial_{xx}:H^2_c(\rr,\rr^d)\to L^2_c(\rr,\rr^d)
\]
is a densely defined sectorial operator of $L^2_c(\rr,\rr^d)$ and since the values $u(x,t)$ taken by the solution are bounded (uniformly with respect to $x$ in $\rr$ and $t$ in $[0,+\infty)$), the nonlinearity $w\mapsto \nabla V(w)$ can be considered as a globally Lipschitz map of $L^2_c(\rr,\rr^d)$ onto itself (up to changing the values of $\nabla V$ outside of a large ball of $\rr^d$ containing all the values taken by the solution). Thus, the last conclusion of \cref{prop:decrease_energy_trav_frame} (semi-flow in $H^1_c(\rr,\rr^d)\cap\HoneulofR$ and continuity with respect to the initial condition in this space) follows from \cite[Proposition~7.1.9]{Lunardi_analyticSemigroupsOptimalRegularityParabolicPbs_1995}, and it follows from \cite[Proposition~7.1.10]{Lunardi_analyticSemigroupsOptimalRegularityParabolicPbs_1995} that, for every $\alpha$ in $(0,1)$, the solution $t\mapsto u(\cdot,t)$ belongs to the space
\begin{equation}
\label{space_solution}
\ccc^0\bigl([0,+\infty),H^1_c(\rr,\rr^d)\bigr) \,\cap\, 
\ccc^\alpha\bigl((0,+\infty),H^2_c(\rr,\rr^d)\bigr) \,\cap\, 
\ccc^{1,\alpha}\bigl((0,+\infty),L^2_c(\rr,\rr^d)\bigr)
\,;
\end{equation}
in particular, its time derivative $t\mapsto u_t(\cdot,t)$ belongs to the space
\[
\ccc^\alpha\bigl((0,+\infty),L^2_c(\rr,\rr^d)\bigr)
\,,
\]
which means that $D_c$ is uniformly continuous on $(0,+\infty)$. It follows from the first among these two conclusions that the function $t\mapsto \eee_c[v(\cdot,t)]$ is continuous on $[0,+\infty)$. Regarding its differentiability, a formal derivation under the integral sign yields:
\[
\frac{d}{dt}\eee_c\bigl[v(\cdot,t)\bigr] = \int_{\rr} e^{c\xi}\bigl(v_\xi\cdot v_{t\xi} + \nabla V(v)\cdot v_t\bigr) \, d\xi
\,,
\]
which provides the intended conclusion after integrating by parts the term $e^{c\xi}v_\xi\cdot v_{t\xi}$. However, this computation is not rigorously justified since the function $v_{t\xi}$ may not belong to $L^2_c(\rr,\rr^d)$, or even exist. A way to circumvent this issue is to work with the discrete time derivative of $v_\xi$. Notice that, due to local parabolic estimates and since $\nabla V$ is of class $\ccc^1$, the function $v$ is of class $\ccc^1$ in time and $\ccc^2$ in space on $\rr\times(0,+\infty)$. As a consequence, for all positive quantities $h$ and $L$, introducing the functions $v^h$ and $v^h_{\xi}$ defined as
\[
v^h(\xi,t) = \frac{v(\xi,t+h/2) - v(\xi,t-h/2)}{h}
\quad\text{and}\quad
v^h_{\xi}(\xi,t) = \frac{v_\xi(\xi,t+h/2) - v_\xi(\xi,t-h/2)}{h}
\,,
\]
the following integration by parts holds:
\[
\begin{aligned}
\int_{-L}^L e^{c\xi} v^h_{\xi}(\xi,t)\cdot v_{\xi}(\xi,t) \, d\xi &= e^{cL}v^h(L,t)\cdot v_\xi(L,t)-e^{-cL}v^h(-L,t)\cdot v_{\xi}(-L,t) \\
&\quad -\int_{-L}^L  e^{c\xi} v^h(\xi,t) \cdot \bigl(v_{\xi\xi}(\xi,t)+cv_{\xi}(\xi,t)\bigr) \, d\xi
\,.
\end{aligned}
\]
It follows from the regularity of $v$ that
\[
\begin{aligned}
\int_{-L}^L e^{c\xi} v^h_{\xi}(\xi,t)\cdot v_{\xi}(\xi,t) \, d\xi &= e^{cL}v_t(L,t)\cdot v_\xi(L,t)-e^{-cL}v_t(-L,t)\cdot v_{\xi}(-L,t) \\
&\quad -\int_{-L}^L  e^{c\xi} v_t(\xi,t) \cdot \bigl(v_{\xi\xi}(\xi,t)+cv_{\xi}(\xi,t)\bigr) \, d\xi + O_{h\to 0}(h)
\,.
\end{aligned}
\]
For all positive quantities $T_1$ and $T_2$ satisfying the inequalities $h/2<T_1<T_2$, integrating this equality on the interval $[T_1,T_2]$ and applying Fubini's Theorem yields:
\[
\begin{aligned}
&\frac{1}{h}\int_{T_2-h/2}^{T_2} \int_{-L}^L e^{c\xi} v_\xi(\xi,t+h/2)\cdot v_\xi(\xi,t)\, d\xi\, dt \\
- &\frac{1}{h}\int_{T_1-h/2}^{T_1} \int_{-L}^L e^{c\xi} v_\xi(\xi,t+h/2)\cdot v_\xi(\xi,t)\, d\xi\, dt \\
&\quad =\int_{T_1}^{T_2} \bigl( e^{cL} v_t(L,t)\cdot v_\xi(L,t) - e^{-cL} v_t(-L,t)\cdot v_{\xi}(-L,t)\bigr) \, dt \\
&\qquad -\int_{-L}^L \int_{T_1}^{T_2} e^{c\xi} v_t(\xi,t) \cdot \bigl(v_{\xi\xi}(\xi,t)+cv_{\xi}(\xi,t)\bigr) \, dt \, d\xi + O_{h\to 0}(h)
\,;
\end{aligned}
\]
according to the continuity of $v_\xi$, passing to the limit as $h$ goes to $0$ yields:
\[
\begin{aligned}
\frac{1}{2}\int_{-L}^L e^{c\xi} v_{\xi}^2(\xi,T_2)\, d\xi & -\frac{1}{2}\int_{-L}^L e^{c\xi}  v_{\xi}^2(\xi,T_1)\, d\xi \\
&=\int_{T_1}^{T_2} \bigl( e^{cL} v_t(L,t)\cdot v_\xi(L,t) - e^{-cL} v_t(-L,t)\cdot v_{\xi}(-L,t)\bigr) \, dt \\
&\quad -\int_{-L}^L \int_{T_1}^{T_2} e^{c\xi} v_t(\xi,t) \cdot \bigl(v_{\xi\xi}(\xi,t)+cv_{\xi}(\xi,t)\bigr) \, dt \, d\xi
\,;
\end{aligned}
\]
and, since $t \mapsto u(\cdot,t)$ belongs to the space \cref{space_solution}, passing to the limit as $L$ goes to $+\infty$ along a suitable subsequence yields, after another application of Fubini's Theorem,
\begin{equation}
\label{T2_T1_derivatives}
\begin{aligned}
\frac{1}{2}\int_{\rr} e^{c\xi}  v_{\xi}^2(\xi,T_2) \, d\xi & - \frac{1}{2}\int_{\rr} e^{c\xi} v_{\xi}^2(\xi,T_1)\, d\xi \\
&= -\int_{T_1}^{T_2} \int_{\rr} e^{c\xi} v_t(\xi,t) \cdot \bigl(v_{\xi\xi}(\xi,t) + cv_{\xi}(\xi,t)\bigr) \, d\xi \, dt
\,.
\end{aligned}
\end{equation}
Another consequence of Fubini's Theorem is the identity
\begin{equation}
\label{T2_T1_potentials}
\int_{\rr} e^{c\xi} V\bigl(v(\xi,T_2)\bigr) \, d\xi - \int_{\rr} e^{c\xi} V\bigl(v(\xi,T_1)\bigr)  \, d\xi = \int_{T_1}^{T_2} \int_{\rr} e^{c\xi} \nabla V\bigl(v(\xi,t)\bigr) \cdot v_t(\xi,t)\, d\xi \, dt
\,.
\end{equation}
It follows from \cref{T2_T1_derivatives,T2_T1_potentials} that
\[
\eee_c\bigl[v(\cdot,T_2)\bigr]-\eee_c\bigl[v(\cdot,T_1)\bigr] = -\int_{T_1}^{T_2} D_c(t) \, dt
\,,
\]
which, by the Fundamental Theorem of Calculus, implies that $t\mapsto \eee_c\bigl[v(\cdot,t)\bigr]$ is differentiable on $(0,+\infty)$ and that its derivative at every positive time $t$ is equal to $-D_c(t)$, which completes the proof.
\end{proof}
\subsection{Energy of a pushed front in a travelling frame}
Let $\phi$ denote the profile of a pushed front travelling at some (positive) speed $c$ and invading $0_{\rr^d}$. According to the notation \cref{def_sigma_c_2d_of_e} for the eigenvalues of the linearized differential system \cref{syst_trav_front_order_1_2_linearized} at $0_{\rr^d}$ and to the \cref{def:pushed_travelling_wave_front} of a pushed travelling front, there exists an integer $j$ in $\{1,\dots,d\}$ such that, for $k$ in $\{0,1,2\}$ and if $D^k\phi$ denotes the $k$-th derivative of $\phi$,
\[
D^k \phi(\xi) = O\bigl(e^{\lambda_{c,-}(\mu_j)\xi}\bigr)
\quad\text{as}\quad
\xi\to+\infty
\,.
\] 
According to the expression \cref{def_lambda_c_pm_of_mu} of $\lambda_{c,-}(\mu)$ and to the \cref{def:pushed_travelling_wave_front} of a pushed travelling front,
\[
2\lambda_{c,-}(\mu_j) < -c 
\,.
\]
The following result (see \cref{fig:graph_eee_cprime_of_phi}) was first established by Muratov \cite[Proposition~3.10]{Muratov_globVarStructPropagation_2004} in the setting of gradient parabolic systems in cylinders. For sake of completeness, a proof in the present setting is provided below. 
\begin{proposition}[energy of a pushed front in a travelling frame]
\label{prop:energy_pushed_front_travelling_frame}
For every speed $c'$ in the interval $\bigl(0,2\abs{\lambda_{c,-}(\mu_j)}\bigr)$, the following equality holds:
\begin{equation}
\label{energy_pushed_front_travelling_frame}
\eee_{c'}[\phi] = \left(1-\frac{c}{c'}\right) \int_{\rr} e^{c'\xi} \phi'(\xi)^2 \, d\xi
\,;
\end{equation}
in particular, the energy of a pushed front in the frame travelling at its own speed vanishes:
\begin{equation}
\label{energy_of_pushed_front_vanishes}
\eee_{c}[\phi] = 0 
\,.
\end{equation}
\end{proposition}
\begin{figure}[!htbp]
\centering
\includegraphics[width=.5\textwidth]{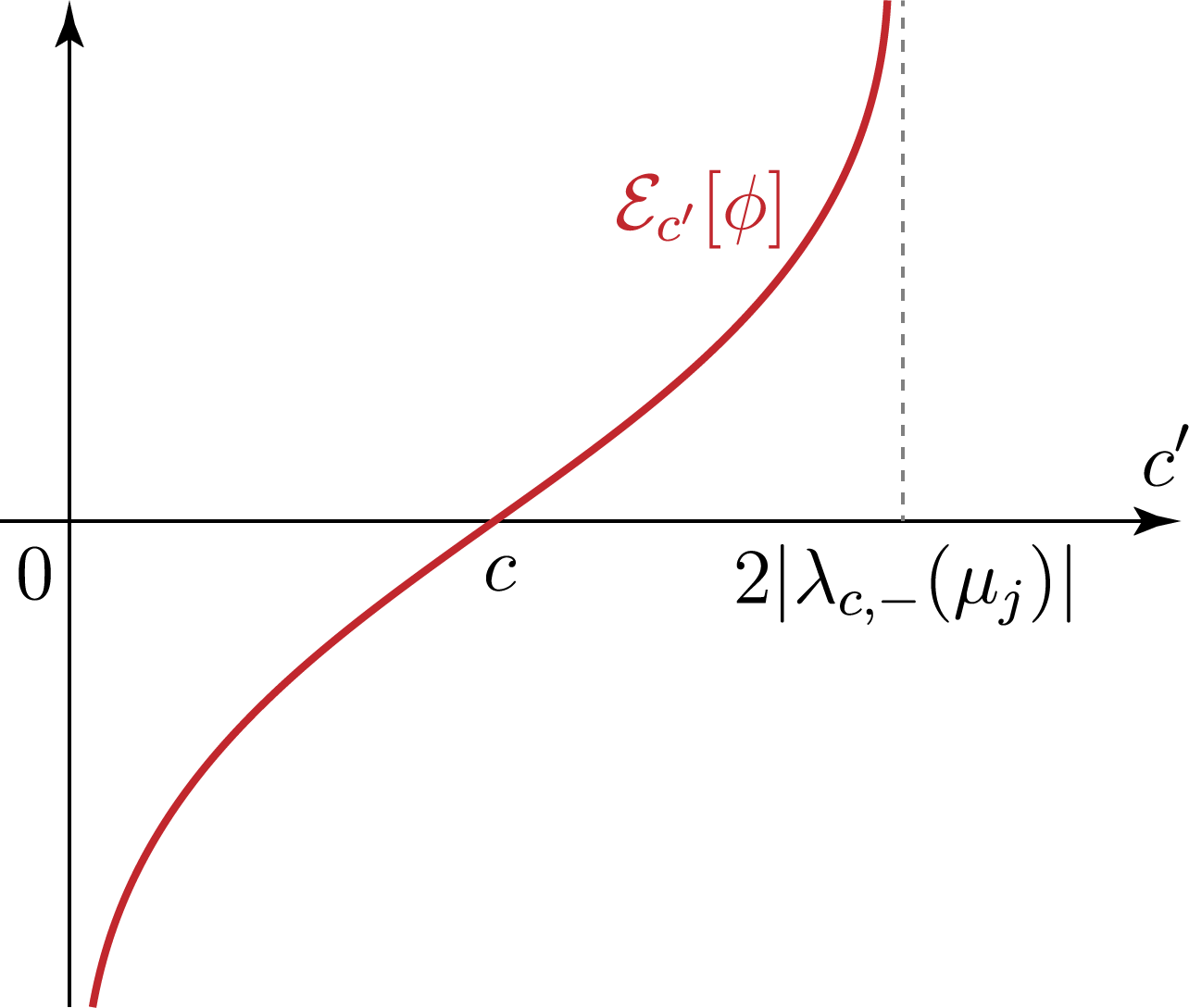}
\caption{Graph of the function $c'\mapsto \eee_{c'}[\phi]$ (\cref{prop:energy_pushed_front_travelling_frame}). The function vanishes and changes sign at $c'$ equals $c$. It diverges as $c'$ goes to $0$ due to the ratio $c/c'$ in the expression \cref{energy_pushed_front_travelling_frame}, and it also diverges as $c'$ goes to $2\abs{\lambda_{c,-}(\mu_j)}$ due to the asymptotics of $\phi'(\xi)$ as $\xi$ goes to $+\infty$.}
\label{fig:graph_eee_cprime_of_phi}
\end{figure}
\begin{proof}
Multiplying the differential system \cref{syst_trav_front_order_2} governing the profile of $\phi$ by $e^{c'\xi} \phi'(\xi)$ and integrating over $\rr$ leads to (omitting the argument $\xi$ of $\phi$ and its derivatives in the integrand):
\[
\int_{\rr} e^{c'\xi}\left(\phi'\cdot \phi'' + c (\phi')^2 - \nabla V(\phi)\cdot\phi'\right)\, d\xi = 0
\,,
\]
or equivalently, 
\[
\begin{aligned}
(c'-c) \int_{\rr} e^{c'\xi} (\phi')^2 \, d\xi &= \int_{\rr} e^{c'\xi}\left(\phi'\cdot \phi'' + c' (\phi')^2 - \nabla V(\phi)\cdot\phi'\right)\, d\xi \\
&= \int_{\rr} e^{c'\xi}\left(-\frac{1}{2}c'(\phi')^2 + c' (\phi')^2 +c' V(\phi)\right)\, d\xi \\
&= c'\int_{\rr} e^{c'\xi}\left(\frac{1}{2}(\phi')^2 + V(\phi)\right)\, d\xi
\,,
\end{aligned}
\]
which is the intended equality \cref{energy_pushed_front_travelling_frame}. Choosing $c'$ equal to $c$ in this equality \cref{energy_pushed_front_travelling_frame} yields the second equality \cref{energy_of_pushed_front_vanishes}. 
\end{proof}
\subsection{Poincaré inequalities in weighted Sobolev spaces}
As was already observed by Muratov \cite{Muratov_globVarStructPropagation_2004}, Poincaré inequalities in the weighted Sobolev spaces $H^1_c(\rr,\rr^d)$ are a key ingredient for exploiting the variational structure in travelling frames, in that they provide lower bounds on the energy $\eee_c[\cdot]$. The following lemma is a variant of \cite[Lemma~2.7]{Muratov_globVarStructPropagation_2004}, \cite[inequalities~(2.8) and (2.11)]{GallayRisler_globStabBistableTW_2007}, and \cite[Proposition~4.3]{GallayJoly_globStabDampedWaveBistable_2009}. It will be used in the proof of \cref{lem:lower_bound_energy_trav_frame} in the next \namecref{subsec:lower_bound_energy_trav_frame}, and furthermore all along the proof of \cref{thm:main}.
\begin{lemma}[Poincaré inequalities]
\label{lem:Poincare_inequality}
For every positive quantity $c$ and every function $v$ in $H^1_c(\rr,\rr^d)$, the following conclusions hold.
\begin{enumerate}
\item The following limits hold: $\quad v(\xi) = o\bigl(e^{-\frac{1}{2}c\xi}\bigr) \text{ as }\xi\to\pm\infty$.
\label{item:lem_Poincare_inequality_convergence_at_both_ends}
\item For every real quantity $\xi_0$, every real quantity $\xi_1$ greater than $\xi_0$, and every positive quantity $\lambda$, the following inequalities hold:
\begin{align}
\label{Poincare_inequality_any_gamma_bounded_interval}
\int_{\xi_0}^{\xi_1} e^{c\xi}v'(\xi)^2 \, d\xi &\ge \lambda e^{c\xi_0}v(\xi_0)^2 - \lambda e^{c\xi_1}v(\xi_1)^2 + \lambda(c-\lambda)\int_{\xi_0}^{\xi_1} e^{c\xi}v(\xi)^2\, d\xi
\,, \\
\label{Poincare_inequality_any_gamma}
\int_{\xi_0}^{+\infty} e^{c\xi}v'(\xi)^2 \, d\xi &\ge \lambda e^{c\xi_0}v(\xi_0)^2 + \lambda(c-\lambda)\int_{\xi_0}^{+\infty} e^{c\xi}v(\xi)^2\, d\xi
\,, \\
\label{Poincare_inequality_gamma_equals_c_over_2}
\int_{\xi_0}^{+\infty} e^{c\xi}v'(\xi)^2\, d\xi &\ge \frac{c}{2}e^{c\xi_0}v(\xi_0)^2 + \frac{c^2}{4}\int_{\xi_0}^{+\infty}e^{c\xi}v(\xi)^2\, d\xi
\,, \\
\label{Poincare_inequality_on_R}
\int_{-\infty}^{+\infty} e^{c\xi}v'(\xi)^2\, d\xi &\ge \frac{c^2}{4}\int_{-\infty}^{+\infty}e^{c\xi}v(\xi)^2\, d\xi
\,.
\end{align}
In addition, if $v$ is not identically equal to $0_{\rr^d}$ on $\rr$ then inequality \cref{Poincare_inequality_on_R} is actually strict, and so is inequality \cref{Poincare_inequality_gamma_equals_c_over_2} if $v$ is not identically equal to $0_{\rr^d}$ on $[\xi_0,+\infty)$. 
\end{enumerate}
\end{lemma}
\begin{remark}
Inequality \cref{Poincare_inequality_any_gamma} is the limit of inequality \cref{Poincare_inequality_any_gamma_bounded_interval} as $\xi_1$ goes to $+\infty$, and inequality \cref{Poincare_inequality_gamma_equals_c_over_2} is nothing but inequality \cref{Poincare_inequality_any_gamma} for $\lambda$ equal to $c/2$. This choice of $\lambda$ is optimal to maximize the term involving the integral of $v(\xi)^2$ to the right-hand side of these inequalities; in particular, it is the best possible choice if the integration domain is the whole real line (inequality \cref{Poincare_inequality_on_R}). In inequality \cref{Poincare_inequality_any_gamma}, choosing a quantity $\lambda$ which is larger than $c/2$ (say between $c/2$ and $c$) increases the size of the term involving $v(\xi_0)^2$ at the expense of the integral (see statement \cref{item:lem_lower_bound_energy_trav_frame_without_bar_xi} of \cref{lem:lower_bound_energy_trav_frame}) --- and choosing $\lambda$ smaller than $c/2$ does not make sense. In inequality \cref{Poincare_inequality_any_gamma_bounded_interval} by contrast, choosing $\lambda$ smaller than $c/2$ can make sense since this decreases the size of the negative term on the right-hand side (see the proof of \cref{lem:gap_between_invasion_points_is_bounded}).
\end{remark}
\begin{proof}
For every quantity $\xi_1$ greater than $\xi_0$, 
\[
e^{c\xi_1}v(\xi_1)^2 - e^{c\xi_0}v(\xi_0)^2 = \int_{\xi_0}^{\xi_1} e^{c\xi} \bigl(c v(\xi)^2 + 2 v(\xi)\cdot v'(\xi) \bigr)\, d\xi
\,.
\]
Since $v$ is in $H^1_c(\rr,\rr^d)$, the right-hand side of this inequality converges to a finite limit as $\xi_1$ goes to $+\infty$; thus the same is true for the quantity $e^{c\xi_1}v(\xi_1)^2$, and since the function $\xi\mapsto e^{c\xi}v(\xi)^2$ is in $L^1(\rr,\rr^d)$, this limit is necessarily $0$. The same argument shows that the quantity $e^{c\xi_0}v(\xi_0)^2$ must also go to $0$ as $\xi_0$ goes to $-\infty$. This proves conclusion \cref{item:lem_Poincare_inequality_convergence_at_both_ends}. 

For every positive quantity $\lambda$, using the polar identity
\[
2 v(\xi)\cdot v'(\xi) = - \lambda^{-1} v'(\xi)^2 - \lambda v(\xi)^2 + \bigl(\lambda^{-1/2}v'(\xi)+\lambda^{1/2}v(\xi)\bigr)^2
\]
and multiplying the previous equality by $\lambda$, it follows that
\begin{equation}
\label{Poincare_equality_any_gamma}
\begin{aligned}
\int_{\xi_0}^{\xi_1} e^{c\xi}v'(\xi)^2\, d\xi = & \lambda e^{c\xi_0}v(\xi_0)^2 - \lambda e^{c\xi_1}v(\xi_1)^2 + \lambda(c-\lambda) \int_{\xi_0}^{\xi_1} e^{c\xi}v(\xi)^2\, d\xi \\
& + \lambda\int_{\xi_0}^{\xi_1}e^{c\xi}\bigl(\lambda^{-1/2}v'(\xi)+\lambda^{1/2}v(\xi)\bigr)^2\, d\xi
\,,
\end{aligned}
\end{equation}
and dropping the last (nonnegative) integral gives inequality \cref{Poincare_inequality_any_gamma_bounded_interval}. According to conclusion \cref{item:lem_Poincare_inequality_convergence_at_both_ends}, passing to the limit as $\xi_1$ goes to $+\infty$ in equality \cref{Poincare_equality_any_gamma} gives
\[
\begin{aligned}
\int_{\xi_0}^{+\infty} e^{c\xi}v'(\xi)^2\, d\xi = & \lambda e^{c\xi_0}v(\xi_0)^2 + \lambda(c-\lambda) \int_{\xi_0}^{+\infty} e^{c\xi}v(\xi)^2\, d\xi \\
& + \lambda\int_{\xi_0}^{+\infty}e^{c\xi}\bigl(\lambda^{-1/2}v'(\xi)+\lambda^{1/2}v(\xi)\bigr)^2\, d\xi
\,,
\end{aligned}
\]
and dropping the last (nonnegative) integral gives inequality \cref{Poincare_inequality_any_gamma} and inequality \cref{Poincare_inequality_gamma_equals_c_over_2} for $\lambda$ equal to $c/2$, and finally inequality \cref{Poincare_inequality_on_R} by passing to the limit as $\xi_0$ goes to $+\infty$. 

To prove the ``strict'' version of inequalities \cref{Poincare_inequality_gamma_equals_c_over_2,Poincare_inequality_on_R}, observe that, if the quantity 
\begin{equation}
\label{integral_of_polar_square_in_Poincare_equality}
\int_{\xi_0}^{+\infty}e^{c\xi}\bigl(\lambda^{-1/2}v'(\xi)+\lambda^{1/2}v(\xi)\bigr)^2\, d\xi
\end{equation}
vanishes, then there must exists some vector $w$ of $\rr^d$ such that, for every $\xi$ in $[\xi_0,+\infty)$, 
\[
v(\xi) = e^{-\lambda\xi} w
\,,
\]
and if in addition $\lambda$ is equal to $c/2$, then according to conclusion \cref{item:lem_Poincare_inequality_convergence_at_both_ends} this vector $w$ must be equal to $0_{\rr^d}$. In other words, if $v$ is not identically equal to $0_{\rr^d}$ on $[\xi_0,+\infty)$ then the integral \cref{integral_of_polar_square_in_Poincare_equality} is positive, and the same is true for the same integral over $\rr$ if $v$ is not identically equal to $0_{\rr^d}$ on $\rr$.
\end{proof}
\begin{remark}
As can be seen on equality \cref{Poincare_equality_any_gamma}, the fact that the quantity $e^{c\xi_1}v(\xi_1)^2$ goes to $0$ as $\xi_1$ goes to $+\infty$ is crucial to obtain a meaningful lower bounds on the integrals of $e^{c\xi}v'(\xi)^2$ at the left-hand side of inequalities \cref{Poincare_inequality_any_gamma,Poincare_inequality_gamma_equals_c_over_2,Poincare_inequality_on_R}. 
\end{remark}
\subsection{Lower bound on energy in a travelling frame}
\label{subsec:lower_bound_energy_trav_frame}
Let us consider a positive quantity $c_0$ and a negative quantity $\mu_0$ such that
\begin{align}
\label{condition_mu_0_c_0_lower_bound_energy}
\mu_0 = -\frac{c_0^2}{4}\,,&
\quad\text{or equivalently}\quad
c_0 = 2\sqrt{-\mu_0}
\,,\\
\text{and}\quad
\mu_0 \le\mu_1\,,&
\quad\text{or equivalently}\quad
\cLinMax \le c_0 
\,,
\nonumber
\end{align}
see \cref{fig:correspondence_mu_c}. Let us consider a quantity $c$ in $[c_0,+\infty)$ and a function $v$ in $H^1_c(\rr,\rr^d)$, and, in accordance with the notation $\lambda_{c,\pm}(\cdot)$ introduced in \cref{def_lambda_c_pm_of_mu}, let us consider the quantities
\begin{equation}
\label{def_lambda_c_pm_of_mu0}
\lambda_{c,\pm}(\mu_0) = -\frac{c}{2}\pm\sqrt{\frac{c^2}{4}+\mu_0}
\,,
\end{equation}
see \cref{fig:eigenvalues}. The next lemma will rely on the assumption that the inequality 
\begin{equation}
\label{V_greater_than_one_half_mu0_v_of_xi_square}
V\bigl(v(\xi)\bigr)\ge\frac{1}{2}\mu_0 v(\xi)^2
\end{equation}
holds for $\xi$ in $\rr$ or for $\xi$ in some interval $[\widebar{\xi},+\infty)$ of $\rr$ (see \cref{fig:graph_of_V_for_appendix}). Conclusion \cref{item:lem_lower_bound_energy_trav_frame_without_bar_xi} of this lemma is similar to \cite[Lemma~3.6]{Muratov_globVarStructPropagation_2004} and conclusion \cref{item:lem_lower_bound_energy_trav_frame_with_bar_xi} is similar to \cite[inequality~(2.9)]{GallayRisler_globStabBistableTW_2007}, \cite[inequality~(4.18)]{GallayJoly_globStabDampedWaveBistable_2009}, and \cite[inequalities~(2.8) and (6.15)]{Luo_globStabDampedWaveEqu_2013}.
\begin{lemma}[lower bound on energy in a travelling frame]
\label{lem:lower_bound_energy_trav_frame}
The following two statements hold. 
\begin{enumerate}
\item If $v$ is not identically equal to $0_{\rr^d}$ and inequality \cref{V_greater_than_one_half_mu0_v_of_xi_square} holds for every $\xi$ in $\rr$, then 
\begin{equation}
\label{lower_bound_energy_trav_frame_without_bar_xi}
\eee_c[v]>0
\,.
\end{equation}
\label{item:lem_lower_bound_energy_trav_frame_without_bar_xi}
\item If there exists $\widebar{\xi}$ in $\rr$ such that inequality \cref{V_greater_than_one_half_mu0_v_of_xi_square} holds for every $\xi$ in $[\widebar{\xi},+\infty)$, then
\begin{equation}
\label{lower_bound_energy_trav_frame_with_bar_xi}
\eee_c[v] \ge e^{c\widebar{\xi}}\left(-\frac{\abs{\Vmin}}{c}+\frac{1}{2}\abs{\lambda_{c,-}(\mu_0)}v(\widebar{\xi})^2\right)
\,,
\end{equation}
and if in addition $c$ is greater than $c_0$, then there exists a positive quantity $\alpha$, depending on $c$ and $c_0$ (only) such that
\begin{equation}
\label{lower_bound_energy_trav_frame_H1_norm_from_bar_xi}
\eee_c[v] \ge -e^{c\widebar{\xi}}\frac{\abs{\Vmin}}{c}+\alpha\int_{\widebar{\xi}}^{+\infty} e^{c\xi}\bigl(v'(\xi)^2+v(\xi)^2\bigr)\, d\xi
\,.
\end{equation}
\label{item:lem_lower_bound_energy_trav_frame_with_bar_xi}
\end{enumerate}
\end{lemma}
\begin{proof}
It follows from the expression \cref{def_eee_c} of $\eee_c[v]$ that
\[
\begin{aligned}
\eee_c[v] &= \int_{\rr}e^{c\xi} \left(\frac{1}{2} v'(\xi)^2 + V\bigl(v(\xi)\bigr)\right)\, d\xi \\
&= \int_{\rr}e^{c\xi} \left( \frac{1}{2}\Bigl(v'(\xi)^2 -\frac{c^2}{4}v(\xi)^2\Bigr) + \frac{c^2-c_0^2}{8}v(\xi)^2 + \Bigl(V\bigl(v(\xi)\bigr)-\frac{\mu_0}{2}v(\xi)^2\Bigr) \right) \, d\xi
\,,
\end{aligned}
\]
so that, since $c$ is greater than or equal to $c_0$, if inequality \cref{V_greater_than_one_half_mu0_v_of_xi_square} holds for every $\xi$ in $\rr$, then
\[
\eee_c[v] \ge \frac{1}{2} \int_{\rr}e^{c\xi}\Bigl(v'(\xi)^2 -\frac{c^2}{4}v(\xi)^2\Bigr)\, d\xi
\,,
\]
and if in addition $v\not\equiv0_{\rr^d}$, inequality \cref{lower_bound_energy_trav_frame_without_bar_xi} follows from inequality \cref{Poincare_inequality_on_R} of \cref{lem:Poincare_inequality}. Statement \cref{item:lem_lower_bound_energy_trav_frame_without_bar_xi} is proved. 

Now, let us assume that there exists $\widebar{\xi}$ in $\rr$ such that inequality \cref{V_greater_than_one_half_mu0_v_of_xi_square} holds for every $\xi$ in $[\widebar{\xi},+\infty)$. It again follows from the expression \cref{def_eee_c} of $\eee_c[v]$ that 
\begin{align}
\nonumber
\eee_c[v] &= \int_{-\infty}^{\widebar{\xi}} e^{c\xi}\left(\frac{1}{2}v'(\xi)^2 + V\bigl(v(\xi)\bigr)\right)\, d\xi + \int_{\widebar{\xi}}^{+\infty} e^{c\xi}\left(\frac{1}{2}v'(\xi)^2 + V\bigl(v(\xi)\bigr)\right)\, d\xi \\
\nonumber
&\ge\int_{-\infty}^{\widebar{\xi}} e^{c\xi} V_{\min}\, d\xi + \frac{1}{2}\int_{\widebar{\xi}}^{+\infty} e^{c\xi}\left(v'(\xi)^2 + \mu_0 v(\xi)^2\right)\, d\xi \\
\label{lower_bound_Ec_proof}
&= -e^{c\widebar{\xi}}\frac{\abs{\Vmin}}{c} + \frac{1}{2}\int_{\widebar{\xi}}^{+\infty} e^{c\xi}\left(v'(\xi)^2 + \mu_0 v(\xi)^2\right)\, d\xi
\,.
\end{align}
Thus, if we consider a quantity $\lambda$ satisfying
\[
\lambda(c-\lambda) = -\mu_0 \iff \lambda^2 - c\lambda - \mu_0 = 0 \iff \lambda = \frac{c}{2}\pm\sqrt{\frac{c^2}{4}+\mu_0} = \abs{\lambda_{c,\pm}(\mu_0)}
\,,
\]
then it follows from the lower bound \cref{lower_bound_Ec_proof} on $\eee_c[v]$ and from inequality \cref{Poincare_inequality_any_gamma} of \cref{lem:Poincare_inequality} that
\[
\eee_c[v] \ge 
-e^{c\widebar{\xi}}\frac{\abs{\Vmin}}{c} + \frac{1}{2}\lambda e^{c\widebar{\xi}} v(\widebar{\xi})^2
\,,
\]
so that, if $\lambda$ is chosen equal to $\abs{\lambda_{c,-}(\mu_0)}$ (which provides a better lower bound than if it is chosen equal to $\abs{\lambda_{c,+}(\mu_0)}$), then inequality \cref{lower_bound_energy_trav_frame_with_bar_xi} follows. Inequality \cref{lower_bound_energy_trav_frame_with_bar_xi} of statement \cref{item:lem_lower_bound_energy_trav_frame_with_bar_xi} is proved. 

Let $\alpha$ denote a positive quantity to be chosen below, and let us introduce the quantity $Q$ defined as
\[
Q = \eee_c[v] + e^{c\widebar{\xi}}\frac{\abs{\Vmin}}{c} - \alpha\int_{\widebar{\xi}}^{+\infty} e^{c\xi}\bigl(v'(\xi)^2+v(\xi)^2\bigr)\, d\xi
\,;
\]
proving the second inequality \cref{lower_bound_energy_trav_frame_H1_norm_from_bar_xi} of statement \cref{item:lem_lower_bound_energy_trav_frame_with_bar_xi} amounts to prove that $Q$ is nonnegative. It follows from the lower bound \cref{lower_bound_Ec_proof} on $\eee_c[v]$ that
\[
Q \ge \int_{\widebar{\xi}}^{+\infty} e^{c\xi}\left(\Bigl(\frac{1}{2}-\alpha\Bigr)v'(\xi)^2) - \Bigl(\frac{c_0^2}{8}+\alpha\Bigr)v(\xi)^2\right)\, d\xi
\,.
\]
Thus, if $\alpha$ is smaller than or equal to $1/2$, it follows from inequality \cref{Poincare_inequality_gamma_equals_c_over_2} of \cref{lem:Poincare_inequality} that
\[
\begin{aligned}
Q &\ge \left(\Bigl(\frac{1}{2}-\alpha\Bigr)\frac{c^2}{4}  - \Bigl(\frac{c_0^2}{4}+\alpha\Bigr)\right) \int_{\widebar{\xi}}^{+\infty} e^{c\xi}v(\xi)^2\, d\xi \\
&= \frac{1}{8}\big(c^2 - c_0^2 - \alpha(8+2c^2)\big) \int_{\widebar{\xi}}^{+\infty} e^{c\xi}v(\xi)^2\, d\xi 
\,,
\end{aligned}
\]
so that, if $\alpha$ is chosen as
\[
\alpha = \min\left(\frac{c^2-c_0^2}{8+2c^2},\frac{1}{2}\right)
\,,
\]
then $\alpha$ is positive and the quantity $Q$ is nonnegative. This proves inequality \cref{lower_bound_energy_trav_frame_H1_norm_from_bar_xi}, and therefore completes the proof of statement \cref{item:lem_lower_bound_energy_trav_frame_with_bar_xi}. 
\end{proof}
\begin{remark}
In the proof of statement \cref{item:lem_lower_bound_energy_trav_frame_with_bar_xi}, using Poincaré inequality \cref{Poincare_inequality_gamma_equals_c_over_2} (that is, choosing $\lambda$ equal to $c/2$ instead of $\abs{\lambda_{c,-}(\mu_0)}$) would have led to the (slightly weaker) inequality
\[
\eee_c[v] \ge e^{c\widebar{\xi}}\left(-\frac{\abs{\Vmin}}{c}+\frac{c}{4}v(\widebar{\xi})^2\right)
\,,
\]
which would actually have fulfilled the same needs as the stronger inequality \cref{lower_bound_energy_trav_frame_with_bar_xi}, in the remaining of the paper. 
\end{remark}
\subsection{Basic properties of the variational structure}
\label{subsec:basic_properties_variational_structure}
\subsubsection{Basic properties of the sets \texorpdfstring{$\ccc_{-\infty}$ and $\ccc_0$}{C-infty and C0}}
\label{subsubsec:basic_properties_sets_ccc_minus_infty_ccc_0}
Let us recall the notation $\ccc_{-\infty}$ and $\ccc_0$ introduced in \cref{subsubsec:infimum_energy_trav_frame}. 
\begin{proposition}[basic properties of the sets $\ccc_{-\infty}$ and $\ccc_0$]
\label{prop:basic_properties_variational_structure} 
The following conclusions hold.
\begin{enumerate}
\item The set $\ccc_0$ contains the interval $[\cQuadHull,+\infty)$; in addition, if $c$ is the speed of a pushed travelling front invading $0_{\rr^d}$, then 
\begin{equation}
\label{speed_of_a_pushd_front_is_less_than_cQuadHull}
c<\cQuadHull
\,.
\end{equation}
\label{item:prop_basic_properties_variational_structure_inclusion_ccc_0}
\item The set $\ccc_{-\infty}$ contains the interval $(0,\cLinMax)$.
\label{item:prop_basic_properties_variational_structure_inclusion_ccc_minus_infty}
\item The set $\ccc_{-\infty}$ is open; equivalently, the set $\ccc_0$ is closed. 
\label{item:prop_basic_properties_variational_structure_topology}
\end{enumerate}
\end{proposition}
\begin{remark} 
Conclusion \cref{item:prop_basic_properties_variational_structure_inclusion_ccc_0} is close to \cite[Lemma~3.6 and Theorem~3.7]{Muratov_globVarStructPropagation_2004}; the quantities $\muQuadHull$ and $\cQuadHull$ are denoted by $\mu_-$ and $c_{\max}$ in this reference, see \cite[notation~(2.17) and notation~(3.10)]{Muratov_globVarStructPropagation_2004}.
\end{remark}
\begin{proof}
Let us us consider a speed $c$ in the interval $[\cQuadHull,+\infty)$ and a function $v$ in $H^1_c(\rr,\rr^d)$ which is not identically equal to $0_{\rr^d}$. According to the last inequality of \cref{V_greater_than_lower_quadratic_hull}, the assumptions of statement \cref{item:lem_lower_bound_energy_trav_frame_without_bar_xi} of \cref{lem:lower_bound_energy_trav_frame} hold when the parameter $\mu_0$ involved in this lemma is replaced with $\muQuadHull$. According to this statement, the quantity $\eee_c[v]$ must be positive. This shows that the interval $[\cQuadHull,+\infty)$ is included in the set $\ccc_0$. In addition, since the energy of a pushed travelling front in the frame travelling at its own speed vanishes (equality \cref{energy_of_pushed_front_vanishes} of \cref{prop:energy_pushed_front_travelling_frame}), the quantity $c$ cannot be the speed of a pushed travelling front invading $0_{\rr^d}$. Conclusion \cref{item:prop_basic_properties_variational_structure_inclusion_ccc_0} is proved. 

Let us prove conclusion \cref{item:prop_basic_properties_variational_structure_inclusion_ccc_minus_infty}. If the maximal linear invasion speed $\cLinMax$ is zero (that is, if the least eigenvalue $\mu_1$ of $D^2V(0_{\rr^d})$ is nonnegative) there is nothing to prove. Let us assume that $\cLinMax$ is positive, or equivalently that $\mu_1$ is negative, and let $c$ denote a quantity (speed) in the interval $(0,\cLinMax)$. Let $u_1$ denote a normalized eigenvector of $D^2V(0_{\rr^d})$ for the eigenvalue $\mu_1$ and let $\chi:\rr\to\rr$ denote a smooth cutoff function satisfying
\begin{equation}
\label{cut_off}
\chi(x) = \left\{
\begin{aligned}
1 \quad\text{if}\quad x\le 0\,,\\
0 \quad\text{if}\quad 1\le x\,,
\end{aligned}
\right.
\quad\text{and, for all $x$ in $\rr$,}\quad
0\le \chi(x)\le 1 
\,.
\end{equation}
Let $\varepsilon$ denote a positive quantity, small enough so that
\[
c+\varepsilon < \cLinMax
\,,
\]
and let us consider the function $w:\rr\to\rr^d$, defined as:
\[
w(\xi) = \chi(1-\xi) e^{-\frac{c+\varepsilon}{2}\xi} u_1 
\,,
\]
and which belongs to $H^1_c(\rr,\rr^d)$. It follows from this expression that, for all $\xi$ in $[1,+\infty)$, 
\[
\frac{1}{2} w'(\xi)^2 = \frac{1}{8} (c+\varepsilon)^2 e^{-(c+\varepsilon)\xi}
\,,
\]
and that
\[
V\bigl(w(\xi)\bigr) \sim \frac{1}{2} \mu_1 e^{-(c+\varepsilon)\xi} = - \frac{1}{8} \cLinMax^2 e^{-(c+\varepsilon)\xi}
\quad\text{as}\quad
\xi\to+\infty
\,,
\]
so that
\[
e^{c\xi}\left(\frac{1}{2} w'(\xi)^2 + V\bigl(w(\xi)\bigr)\right) \sim - \frac{1}{8} \bigl(\cLinMax^2-(c+\varepsilon)^2\bigr) e^{-\varepsilon\xi}
\quad\text{as}\quad
\xi\to+\infty
\,.
\]
It follows that 
\[
\eee_c[w]\to -\infty 
\quad\text{as}\quad
\varepsilon\to0\,, \quad \varepsilon >0
\,.
\]
This shows that $c$ belongs to $\ccc_{-\infty}$, and therefore proves conclusion \cref{item:prop_basic_properties_variational_structure_inclusion_ccc_minus_infty}. 

To prove conclusion \cref{item:prop_basic_properties_variational_structure_topology}, let us consider a quantity $c$ in the set $\ccc_{-\infty}$ (this quantity $c$ is therefore positive). According to the definition \cref{def_ccc_minus_infty_ccc_0} of $\ccc_{-\infty}$, there exists a function $w$ in $H^1_c(\rr,\rr^d)$ such that the energy $\eee_c[w]$ is negative. Let us consider again the smooth cutoff function $\chi$ satisfying the conditions \cref{cut_off}. 
Let $\xLarge$ denote a (large) positive quantity to be chosen below and let us consider the function $\tilde{w}$ defined as
\begin{equation}
\label{expression_cutoffed_initial_condition}
\tilde{w}(x) = \chi(x-\xLarge) w(x)
\,.
\end{equation}
Since $w$ is in $H^1_c(\rr,\rr^d)$, the quantity $\eee_c[\tilde{w}]$ goes to $\eee_c[w]$ as $\xLarge$ goes to $+\infty$; thus, if $\xLarge$ is large enough positive, the quantity $\eee_c[\tilde{w}]$ is (also) negative; it follows that, for $c'$ close enough to $c$, the quantity $\eee_{c'}[\tilde{w}]$ is again negative, which shows that $c'$ belongs to $\ccc_{-\infty}$ and yields the intended conclusion. 
\end{proof}
\subsubsection{A sufficient condition for invasion to occur}
\label{subsubsec:sufficient_condition_invasion}
It follows from conclusion \cref{item:prop_basic_properties_variational_structure_inclusion_ccc_minus_infty} of \cref{prop:basic_properties_variational_structure} above that, if the maximal nonlinear invasion speed $\cLinMax$ is positive (that is, if the least eigenvalue $\mu_1$ of $D^2V(0_{\rr^d})$ is negative), then the set $\ccc_{-\infty}$ (which according to conclusion \cref{item:prop_basic_properties_variational_structure_inclusion_ccc_minus_infty} of \cref{prop:basic_properties_variational_structure} contains the interval $(0,\cLinMax)$) is nonempty. The next proposition (which extends \cite[Lemma~7]{Risler_globCVTravFronts_2008}) sets the ground for the upcoming \cref{cor:ccc_minus_infty_nonempty} which states that the the set $\ccc_{-\infty}$ is actually \emph{always} nonempty.
\begin{proposition}[a sufficient condition for invasion]
\label{prop:sufficient_condition_invasion}
For every positive quantity $\cDecay$ and every $w$ in $H^1_{\cDecay}(\rr,\rr^d)$, if 
\begin{equation}
\label{sufficient_condition_invasion}
\limsup_{L\to+\infty} \int_{-L}^{+\infty}\left(\frac{1}{2}w'(x)^2 + V\bigl(w(x)\bigr)\right)\, dx < 0
\,,
\end{equation}
then, for every sufficiently small speed $c$ (in the interval $(0,\cDecay]$), the energy $\eee_{c}[w]$ is negative. 
\end{proposition}
\begin{proof}
Let $\cDecay$ denote a positive quantity and $w$ denote a function in $H^1_{\cDecay}(\rr,\rr^d)$. Let us consider the function $\mathfrak{e}:\rr\to\rr$ defined as
\[
\mathfrak{e}(x) = \frac{1}{2}w'(x)^2 + V\bigl(w(x)\bigr)
\,,
\]
and let us assume that assumption \cref{sufficient_condition_invasion} above holds, that is
\[
\limsup_{x\to-\infty}\int_{x}^{+\infty}\mathfrak{e}(y)\, dy < 0
\,.
\]
It follows from this assumption that there exists a (small) positive quantity $\varepsilon$ and a (large) negative quantity $\xLeftNegEn$ such that
\begin{equation}
\label{rewording_assumption_prop_sufficient_condition_negative_energy_at_small_positive_speed}
x\le\xLeftNegEn 
\implies
\int_x^{+\infty}\mathfrak{e}(y)\, dy \le -\varepsilon
\,,
\end{equation}
see \cref{fig:tauberian_theorem}.
\begin{figure}[!htbp]
\centering
\includegraphics[width=\textwidth]{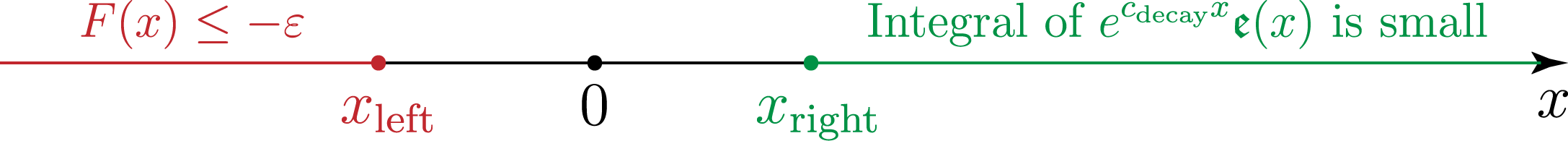}
\caption{Illustration of the proof of \cref{prop:sufficient_condition_invasion}.}
\label{fig:tauberian_theorem}
\end{figure}
Since $w$ belongs to $H^1_{\cDecay}(\rr,\rr^d)$, there exists a (large, positive) quantity $\xRightTrunc$ such that the following conclusions hold:
\begin{align}
\label{V_nonnegative_for_x_not_smaller_than_xRightTrunc}
& x\ge \xRightTrunc \implies V\bigl(w(x)\bigr)\ge 0 \,,
\quad\text{and thus}\quad
\mathfrak{e}(x)\ge 0\,, \\
\text{and}\quad
&\int_{\xRightTrunc}^{+\infty} e^{\cDecay x} \mathfrak{e}(x) \, dx \le \frac{\varepsilon}{2}
\label{integral_to_the_right_of_xRightTrunc_is_negligible}
\,,
\end{align}
see again \cref{fig:tauberian_theorem}. Let $c$ denote a quantity in $(0,\cDecay]$. It follows from the implication \cref{V_nonnegative_for_x_not_smaller_than_xRightTrunc} that inequality \cref{integral_to_the_right_of_xRightTrunc_is_negligible} still holds if $\cDecay$ is replaced with $c$, and it follows that
\begin{equation}
\label{inequality_introduction_eee_trunc}
\eee_c[w] \le \eeecTrunc[w] + \frac{\varepsilon}{2}\,,
\quad\text{where}\quad
\eeecTrunc[w] = \int_{-\infty}^{\xRightTrunc} e^{cx} \mathfrak{e}(x)\, dx
\,.
\end{equation}
Let us consider the function $F:\rr\to\rr$ defined as:
\[
F(x) = \int_x^{\xRightTrunc} \mathfrak{e}(y)\, dy
\,,\quad\text{so that}\quad
F(\xRightTrunc) = 0 
\quad\text{and}\quad
F'(x) = - \mathfrak{e}(x)
\,.
\]
Integrating by parts the expression of $\eeecTrunc[w]$ yields:
\[
\begin{aligned}
\eeecTrunc[w] &= \Bigl[-e^{cx}F(x)\Bigr]_{-\infty}^{\xRightTrunc} + \int_{-\infty}^{\xRightTrunc} c e^{cx}F(x) \, dx \\
&= \int_{-\infty}^{\xRightTrunc} c e^{cx}F(x) \, dx \\
&= \int_{-\infty}^{\xLeftNegEn} c e^{cx}F(x) \, dx + \int_{\xLeftNegEn}^{\xRightTrunc} c e^{cx}F(x) \, dx 
\,.
\end{aligned}
\]
Since $\mathfrak{e}(x)$ is nonnegative for $x$ not smaller than $\xRightTrunc$, it follows from inequality \cref{rewording_assumption_prop_sufficient_condition_negative_energy_at_small_positive_speed} that
\[
x\le\xLeftNegEn \implies F(x)\le-\varepsilon
\,.
\]
Thus it follows from the expression of $\eeecTrunc[w]$ above that
\[
\eeecTrunc[w] \le -\varepsilon e^{c\xLeftNegEn} + \bigl(e^{c\xRightTrunc}-e^{c\xLeftNegEn}\bigr)\max_{x\in[\xLeftNegEn,\xRightTrunc]}F(x)
\,.
\]
As a consequence, if the positive quantity $c$ is small enough, the following inequality holds:
\[
\eeecTrunc[w] <-\frac{\varepsilon}{2}
\,,
\]
and the intended conclusion follows from inequality \cref{inequality_introduction_eee_trunc}. \Cref{prop:sufficient_condition_invasion} is proved. 
\end{proof}
\begin{corollary}[non-emptiness of the set $\ccc_{-\infty}$]
\label{cor:ccc_minus_infty_nonempty}
There exists a positive quantity $\varepsilon$ such that 
\begin{equation}
\label{ccc_minus_infty_nonempty}
(0,\varepsilon)\subset\ccc_{-\infty}
\,.
\end{equation}
\end{corollary}
\begin{proof}
Let $u_-$ denote a point of $\rr^d$ such that $V(u_-)$ is negative (the existence of such a point $u_-$ follows from the negativity of $\Vmin$ stated in \cref{hyp_crit_point}). Let us consider the cutoff function $\chi$ introduced in \cref{cut_off}, and let us consider the function $w$ defined as
\[
w(x) = \chi(x)u_-
\,.
\]
This function $w$ fulfils the assumptions of \cref{prop:sufficient_condition_invasion} so that, according to its conclusion, the intended conclusion \cref{ccc_minus_infty_nonempty} follows. 
\end{proof}
\subsubsection{Lower semi-continuity of the variational invasion speed}
\label{subsubsec:lower_semi_continuity_variational_invasion_speed}
\begin{proof}[Proof of \cref{prop:lower_semi_continuity_variational_invasion_speed}]
Let $c$ denote a quantity (speed) in $(\cNonLinMax,+\infty)$, let $w$ denote a function in the space $\HoneulofR\cap H^1_c(\rr,\rr^d)$ (recall that the critical point $e$ is assumed to be equal to $0_{\rr^d}$ in this section), and let $u$ denote the solution of the parabolic system \cref{parabolic_system} for the initial condition $w$ at time $0$. According to the definition of the variational speed \cref{def:decay_variational_invasion_speed}, $\cVar[w]$ is equal to $\cVar[u]$, and, for every positive quantity $\varepsilon$, there exists a nonnegative time $t_0$ such that 
\[
\eee_{\cVar[w]-\varepsilon}\bigl[u(\cdot,t_0)\bigr] < 0
\,.
\]
According to the continuity of the semi-flow of system \cref{parabolic_system} (restricted to $\HoneulofR\cap H^1_c(\rr,\rr^d)$) with respect to initial conditions (last assertion of \cref{prop:decrease_energy_trav_frame}), for every function $\tilde{w}$ in $\HoneulofR\cap H^1_c(\rr,\rr^d)$ close enough to $w$ for the $\HoneulofR$-norm and the $H^1_c(\rr,\rr^d)$-norm, if $\tilde{u}$ denotes the solution of system \cref{parabolic_system} for the initial condition $\tilde{w}$ at time $0$, then 
\[
\eee_{\cVar[w]-\varepsilon}\bigl[\tilde{u}(\cdot,t_0)\bigr] < 0
\,;
\]
it follows that
\[
\cVar[\tilde{w}] > \cVar[w]-\varepsilon
\,,
\]
which is the intended conclusion. 
\end{proof}
\subsection{Expression of the dissipation in a travelling frame}
\label{subsec:expression_dissipation_travelling_frame}
The expression of the dissipation in \cref{def_Dc_of_t,decrease_energy_trav_frame} leads us to consider, for $v$ in the weighted Sobolev space $H^2_c(\rr,\rr^d)$ (defined in \cref{H2c}), the dissipation functional $\ddd_c[v]$ defined as
\[
\ddd_c[v] = \int_{\rr} e^{c\xi}\bigl(- \nabla V(v) + c v' + v''\bigr)^2 \, d\xi
\,.
\]
According to this expression (omitting the argument $\xi$ of $v$ in the integrand), 
\[
\begin{aligned}
\ddd_c[v] &= \int_{\rr} e^{c\xi}\Bigl(\bigl(\nabla V(v)^2 + (c v' + v'')\cdot\bigl(- 2\nabla V(v) + c v' + v'' \bigr)\Bigr)\, d\xi \\
&= \int_{\rr} e^{c\xi}\Bigl(\nabla V(v)^2 + (c v' + v'')\cdot\bigl(- 2\nabla V(v) + c v'\bigr) + cv'\cdot v'' + (v'')^2\Bigr)\, d\xi 
\,,
\end{aligned}
\]
and according to the equality,
\[
\bigl(e^{c\xi}v'\bigr)' = e^{c\xi}(cv' + v'')
\,,
\]
it follows from an integration by parts of the middle term of the integrand that
\begin{align}
\nonumber
\ddd_c[v] &= \int_{\rr} e^{c\xi}\Bigl(\nabla V(v)^2 + 2 D^2V(v)\cdot v'\cdot v' - cv'\cdot v'' + cv'\cdot v'' + (v'')^2\Bigr)\, d\xi \\
&= \int_{\rr} e^{c\xi}\Bigl(\nabla V(v)^2 + 2 D^2V(v)\cdot v'\cdot v' + (v'')^2\Bigr)\, d\xi
\,,
\label{expression_ddd_c_of_v}
\end{align}
see \cite[912]{GallayRisler_globStabBistableTW_2007} for an identical expression in the scalar case. This expression will not be used as such, but it will justify the introduction, in \cref{subsec:control_energy_to_the_right_of_invasion_point}, of another function $F_c(t)$ with the purpose of controlling the amount of energy to the right of the invasion point, in a frame travelling at a speed close to the invasion speed. 
\section{Proof of the main results}
\label{sec:proof}
\subsection{Set-up}
\label{subsec:set_up}
The proof closely follows the arguments of \cite{GallayRisler_globStabBistableTW_2007,GallayJoly_globStabDampedWaveBistable_2009,Luo_globStabDampedWaveEqu_2013}. Let us consider:
\begin{itemize}
\item a potential function $V$ in $\ccc^2(\rr^d,\rr)$ and a critical point $e$ of $V$ satisfying assumptions \cref{hyp_coerc} and \cref{hyp_crit_point}; 
\item a solution $(x,t)\mapsto u(x,t)$ of the parabolic system \cref{parabolic_system} satisfying the condition \cref{thm_main_cLinMax_smaller_than_cVar_smaller_than_cDecay} of \cref{thm:main}, that is:
\[
\cLinMax < \cVar[u] < \cDecay[u]
\,.
\]
\end{itemize}
Let $c_0$ and $\cDecay$ denote two quantities (speeds) satisfying:
\begin{equation}
\label{inequalities_between_speeds_beginning_proof}
\cLinMax < c_0 < \cVar[u] < \cDecay < \cDecay[u]
\,.
\end{equation}
According to the definition \cref{decay_variational_invasion_speed} of $\cVar[u]$ and $\cDecay[u]$ (\cref{def:decay_variational_invasion_speed}), it may be assumed that, up to changing the origin of times, 
\begin{equation}
\label{assumptions_origin_of_times}
\eee_{c_0}[u(\cdot,0)] < 0 
\quad\text{and}\quad
u(\cdot,0) \in H^1_{\cDecay}(\rr,\rr^d)
\,.
\end{equation}
Let us consider the (negative) quantity $\mu_0$ defined as
\[
\mu_0 = - \frac{c_0^2}{4}
\,;
\]
It follows from inequalities \cref{cLinMax_le_cNonLinMax_le_cQuadHull,basic_inequalities_cVar_cDecay,inequalities_between_speeds_beginning_proof} that $c_0$ is less than $\cQuadHull$; thus, 
\begin{equation}
\label{c_0_between_cLinMax_and_cQuadHull}
\cLinMax < c_0 < \cQuadHull\,,
\quad\text{or equivalently,}\quad
\muQuadHull < \mu_0 < \mu_1
\,.
\end{equation}
\subsubsection{Maximal radius of stability for pushed invasion at the speed \texorpdfstring{$c_0$}{c0}}
\label{subsubsec:max_radius_stability_pushed_invasion}
\begin{definition}[maximal radius of stability for pushed invasion at the speed $c_0$]
\label{def:radius_stability_pushed_invasion}
Let us call \emph{maximal radius of stability for pushed invasion at the speed $c_0$} the quantity $\deltaStab(c_0)$ defined as:
\begin{equation}
\deltaStab(c_0) = \inf \left\{\abs{u}:u\in\rr^d \text{ and } V(u)<\frac{1}{2}\mu_0(u-e)^2\right\} 
\,.
\end{equation}
\end{definition}
According to this definition and to inequalities \cref{c_0_between_cLinMax_and_cQuadHull}, 
\[
0 < \deltaStab(c_0) < +\infty
\,,
\]
and
\begin{equation}
\label{V_greater_than_mu0_u2_for_abs_u_smaller_than_deltaStab_of_c_0}
\text{for every $u$ in $\widebar{B}_{\rr^d}\bigl(e,\deltaStab(c_0)\bigr)$},\quad V(u)\ge \frac{1}{2}\mu_0 (u-e)^2
\,;
\end{equation}
in addition, $\deltaStab(c_0)$ is the largest positive quantity satisfying this property \cref{V_greater_than_mu0_u2_for_abs_u_smaller_than_deltaStab_of_c_0}, see \cref{fig:correspondence_mu_c,fig:graph_of_V}. 
\begin{figure}[!htbp]
\centering
\includegraphics[width=\textwidth]{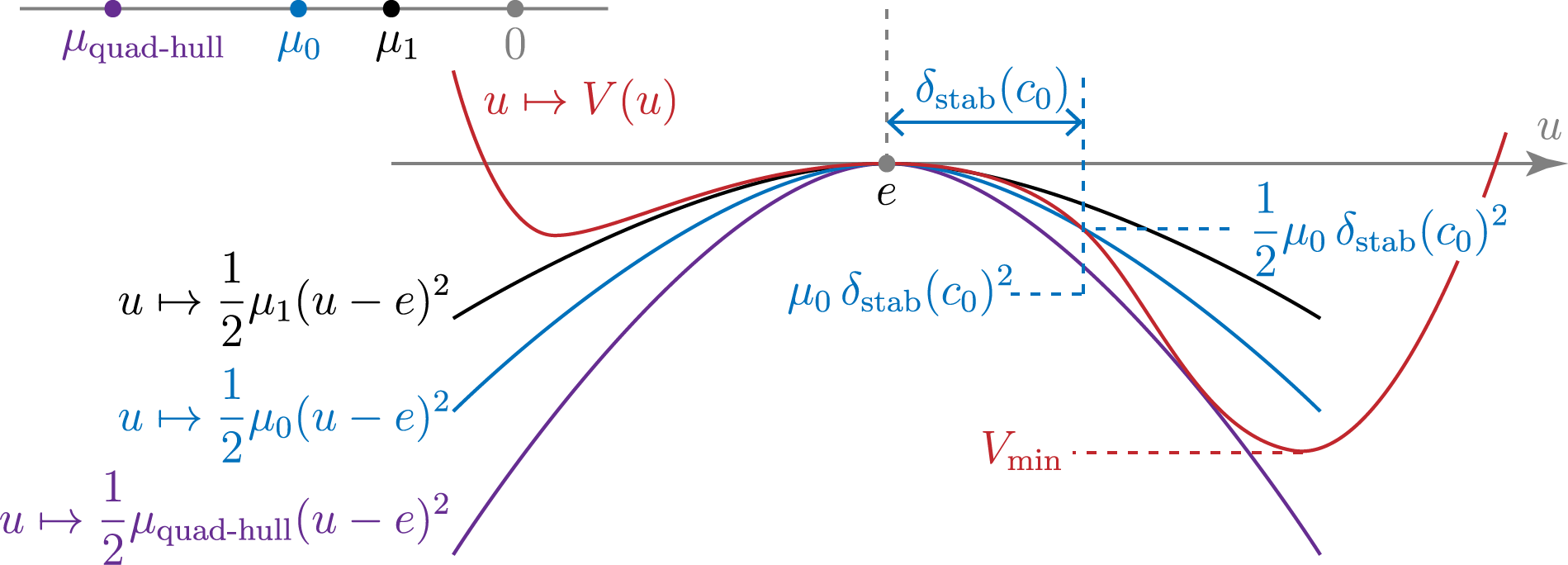}
\caption{Maximal radius of stability $\deltaStab(c_0)$.}
\label{fig:graph_of_V}
\end{figure}
\subsubsection{Invaded critical point at the origin of \texorpdfstring{$\rr^d$}{Rd} and upper bound on the solution}
For convenience, it will be assumed, until the end of \cref{sec:proof}, that $e$ is equal to the origin of $\rr^d$. Let us recall the quantity $\Ratt$, depending only on $V$, introduced in inequality \cref{att_ball_C1_norm}. According to this inequality, up to changing the origin of times (and without loss of generality), it may be assumed that, for every nonnegative time $t$, 
\begin{equation}
\label{bound_u_ux_setup}
\sup_{x\in\rr}\abs{u(x,t)} + \abs{u_x(x,t)} \le \Ratt
\,. 
\end{equation}
Likewise, it may be assumed that the conclusion \cref{decrease_energy_trav_frame} of \cref{prop:decrease_energy_trav_frame} holds for every nonnegative time $t$ (and not only for every positive time $t$).
\subsection{Invasion point}
For every nonnegative time $t$ and every quantity $\delta$ in $\bigl(0,\deltaStab(c_0)\bigr]$, let us consider the set
\begin{equation}
\label{def_SigmaFar}
\begin{aligned}
\SigmaFar{\delta}(t) &= \bigl\{x\in\rr: \abs{u(x,t)}>\delta\bigr\}
\,.
\end{aligned}
\end{equation}
It follows from the properties \cref{assumptions_origin_of_times} of $u(\cdot,0)$ and from \cref{prop:decrease_energy_trav_frame} that the quantity $\eee_{c_0}\bigl[u(\cdot,t)\bigr]$ is negative, so that the set $\SigmaFar{\delta}(t)$ is:
\begin{itemize}
\item according to inequality \cref{lower_bound_energy_trav_frame_without_bar_xi} of statement \cref{item:lem_lower_bound_energy_trav_frame_without_bar_xi} of \cref{lem:lower_bound_energy_trav_frame}, nonempty, 
\item and according to \cref{prop:decrease_energy_trav_frame}, bounded from above.
\end{itemize}
\begin{definition}[invasion point]
Let us call \emph{invasion point in the laboratory frame} (at time $t$) the quantity $\widebar{x}(t)$ defined as
\begin{equation}
\label{def_bar_x}
\widebar{x}(t) = \sup\bigl(\SigmaFar{\deltaStab(c_0)}(t)\bigr)
\end{equation}
(according to the remark above this quantity is finite), and, for every real quantity $c$, let us call \emph{invasion point in the frame travelling at the speed $c$} the quantity $\widebar{\xi}_c(t)$ defined as
\begin{equation}
\label{def_bar_xi_c_of_t}
\widebar{\xi}_c(t) = \widebar{x}(t) - ct
\,,
\end{equation}
see \cref{fig:travelling_frame_invasion_points}. 
\end{definition}
\begin{remarks}
\begin{enumerate}
\item The point labelled as ``invasion point'' in this article is often called ``leading edge'' in the literature, see for instance \cite{Muratov_globVarStructPropagation_2004,LuciaMuratovNovaga_linNonlinSelectionPropSpeedInvasionUnstableEquil_2004}.
\item In most places, $c$ will be assumed to be positive; however allowing $c$ to be nonpositive in the notation \cref{def_bar_xi_c_of_t} above is more convenient for the presentation of \cref{lem:almost_non_decrease_invasion_point} and \cref{cor:control_moves_to_left_invasion_point_trav_frame} in \cref{subsec:control_moves_to_the_left_invasion_point}. 
\end{enumerate}
\end{remarks}
\begin{figure}[!htbp]
\centering
\includegraphics[width=.9\textwidth]{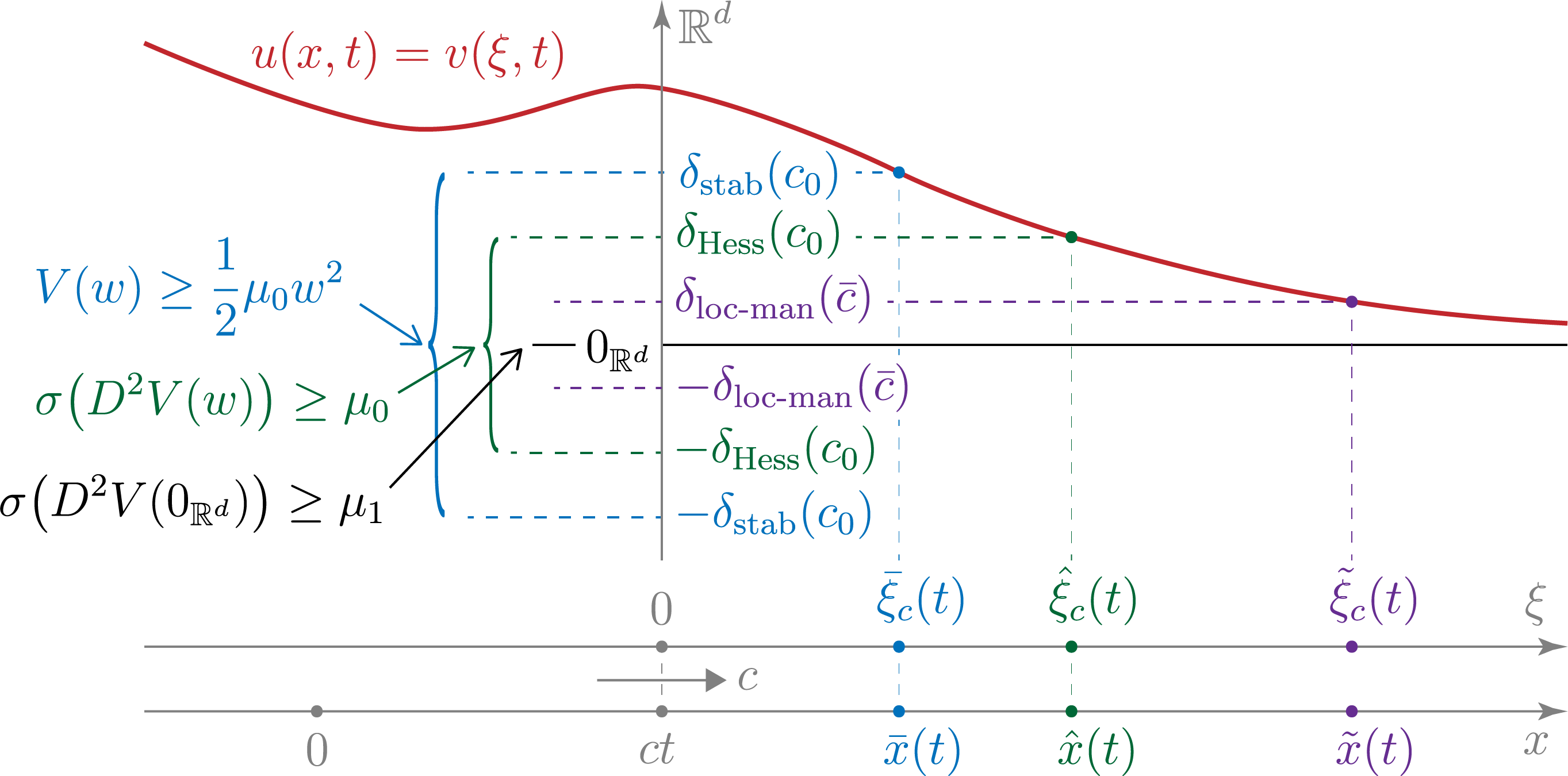}
\caption{Frame travelling at a (positive) speed $c$, maximal radius of stability $\deltaStab(c_0)$ for pushed invasion at the speed $c_0$, and invasion points $\bar{x}(t)$ and $\bar{\xi}_c(t)$ defined by this radius; and invasion points $\hat{x}(t)$ and $\hat{\xi}_c(t)$ introduced in \cref{subsec:invasion_point_defined_by_smaller_radius}, defined by the radius $\deltaHess(c_0)$, and $\tilde{x}(t)$ and $\tilde{\xi}_c(t)$ introduced in \cref{subsec:convergence}, defined by the radius $\deltaLocMan(\bar{c})$.}
\label{fig:travelling_frame_invasion_points}
\end{figure}
According to this notation, 
\begin{equation}
\label{abs_v_smaller_than_deltaStab_of_c0_to_the_right_of_invasion_point}
\abs{v\bigl(\widebar{\xi}_c(t),t\bigr)} = \deltaStab(c_0) 
\quad\text{and, for all $\xi$ in $\bigl[\widebar{\xi}_c(t),+\infty\bigr)$,}\quad
\abs{v(\xi,t)} \le \deltaStab(c_0) 
\,,
\end{equation}
and both quantities $\widebar{x}(t)$ and $\widebar{\xi}_c(t)$ are lower semi-continuous (but not necessarily continuous) with respect to $t$. This lower semi-continuity will not be used as such (more quantitative estimates on these quantities will be obtained in the next \namecref{subsec:invasion_speed}).
\subsection{Invasion speed}
\label{subsec:invasion_speed}
The content of this \namecref{subsec:invasion_speed} owe much to the arguments of \cite[Propositions~3.1 and 3.2]{GallayRisler_globStabBistableTW_2007}, \cite[section~5]{GallayJoly_globStabDampedWaveBistable_2009},  and \cite[section 3]{Luo_globStabDampedWaveEqu_2013}. 
\subsubsection{Lower bound on energy in a travelling frame}
Let $c$ denote a quantity in the interval $(c_0,\cDecay]$, and let us consider the solution $(\xi,t)\mapsto v(\xi,t)$ of system \cref{parabolic_system_trav_frame} defined in \cref{change_of_variable_stand_trav_frame} (for the speed $c$). For every nonnegative time $t$, it follows from the second assertion of \cref{assumptions_origin_of_times} and from \cref{prop:decrease_energy_trav_frame} that the quantity $E_c(t)$ defined as
\begin{equation}
\label{def_Ec_of_t}
E_c(t) = \eee_c\bigl[v(\cdot,t)\bigr]
\end{equation}
is finite. In addition, it follows from inequality \cref{lower_bound_energy_trav_frame_H1_norm_from_bar_xi} of \cref{lem:lower_bound_energy_trav_frame} that there exists a positive quantity $\alpha$, depending on $c$ and $c_0$ (only) such that, for every nonnegative time $t$, 
\begin{equation}
\label{lower_bound_Ec}
E_c(t)\ge - e^{c\widebar{\xi}_c(t)}\frac{\abs{V_{\min}}}{c} + \alpha \int_{\widebar{\xi}_c(t)}^{+\infty} e^{c\xi}\bigl(v_\xi(\xi,t)^2 + v(\xi,t)^2 \bigr)\, d\xi
\,.
\end{equation}
\subsubsection{Bounds on invasion point}
\begin{lemma}[lower bound on invasion point]
\label{lem:lower_bound_invasion_point}
For every $c$ in $\bigl(c_0,\cVar[u]\bigr)$, there exists a positive quantity $K$ such that, for every large enough positive time $t$, 
\begin{equation}
\label{lower_bound_invasion_point}
K + ct \le \widebar{x}(t)
\,.
\end{equation}
\end{lemma}
\begin{proof}
According to the definition \cref{def_bar_xi_c_of_t} of $\widebar{\xi}_c(t)$, the intended inequality \cref{lower_bound_invasion_point} is equivalent to
\begin{equation}
\label{lower_bound_invasion_point_proof}
K \le \widebar{\xi}_c(t)
\,.
\end{equation}
Since $c$ is assume to be smaller than $\cVar[u]$, it follows from the definition \cref{decay_variational_invasion_speed} of $\cVar[u]$ that there exists a nonnegative time $t'$ such that $E_c(t')$ is negative. Then, according to \cref{prop:decrease_energy_trav_frame} and inequality \cref{lower_bound_Ec}, for every time $t$ greater than or equal to $t'$, 
\[
0> E_c(t') \ge E_c(t) \ge - e^{c\widebar{\xi}_c(t)}\frac{\abs{V_{\min}}}{c}
\,,
\quad\text{so that}\quad
0 < \abs{E_c(t')} \le e^{c\widebar{\xi}_c(t)}\frac{\abs{V_{\min}}}{c}
\,,
\]
and so that, if we consider the quantity
\[
K = \frac{1}{c} \ln\left(\abs{E_c(t')}\frac{c}{\abs{V_{\min}}}\right)
\,,
\]
then inequality \cref{lower_bound_invasion_point_proof} follows. 
\end{proof}
Besides \cref{prop:invasion_speed} below, which provides more control on the asymptotic behaviour of the invasion point, the following elementary lemma will be convenient in some of the upcoming arguments. 
\begin{lemma}[upper bound over bounded time intervals on the invasion point]
\label{lem:upper_bound_bounded_time_intervals_invasion_point}
For every positive time $T$, the quantity 
\begin{equation}
\label{upper_bound_bounded_time_intervals_invasion_point}
\sup_{t\in[0,T]} \widebar{x}(t) < +\infty
\,.
\end{equation}
\end{lemma}
\begin{proof}
Let us proceed by contradiction and assume that, for some positive time $T$, the converse holds. Then there exists a sequence $(t_n)_{n\in\nn}$ of nonnegative times, converging to some limit $t_\infty$ in $[0,T]$, such that $\widebar{x}(t_n)$ goes to $+\infty$ as $n$ goes to $+\infty$. Since the function $x\mapsto u(x,t_\infty)$ is in $H^1_{\cDecay}(\rr,\rr^d)$, it converges to $0_{\rr^d}$ to the right end of space. Thus, since the solution varies continuously in $\HoneulofR$ with respect to time, $u(x,t)$ is arbitrarily close to $0_{\rr^d}$ if $x$ is large enough positive and $t$ is close enough to $t_\infty$, a contradiction with the fact that $\abs{u\bigl(\widebar{x}(t_n),t_n\bigr)}$ equals $\deltaStab(c_0)$ for all $n$ in $\nn$. 
\end{proof}
\subsubsection{Upper control on invasion point}
Let us consider the quantities
\[
\widebar{c}_- = \liminf_{t\to+\infty}\frac{\widebar{x}(t)}{t}
\quad\text{and}\quad
\widebar{c}_+ = \limsup_{t\to+\infty}\frac{\widebar{x}(t)}{t}
\,.
\]
It follows from \cref{lem:lower_bound_invasion_point} and from these expressions that
\begin{equation}
\label{cLow_smaller_than_c_minus_and_c_plus_smaller_than_cUpp}
\cVar[u] \le \widebar{c}_- \le \widebar{c}_+ \le +\infty
\,.
\end{equation}
\begin{proposition}[invasion speed]
\label{prop:invasion_speed}
The following equality holds:
\[
\cVar[u] = \widebar{c}_- = \widebar{c}_+ 
\,.
\]
\end{proposition}
\begin{proof}
Let us proceed by contradiction and assume that
\[
\cVar[u] < \widebar{c}_+
\,,
\]
and let us consider a sequence $(t_n)_{n\in\nn}$ of nonnegative times going to $+\infty$, such that
\[
\frac{\widebar{x}(t_n)}{t_n}\xrightarrow[n\to+\infty]{}\widebar{c}_+
\,.
\]
By compactness (\cref{lem:compactness}), up to replacing the sequence $(t_n)_{n\in\nn}$ by a subsequence, there exists an entire solution $u_\infty$ of system \cref{parabolic_system} such that, with the notation of \cref{subsec:compactness}, 
\begin{equation}
\label{compactness_proof_existence_invasion_speed}
D^{2,1}u\bigl(\widebar{x}(t_n)+\cdot,t_n+\cdot\bigr)\to D^{2,1}u_\infty
\quad\text{as}\quad
n\to+\infty
\,,
\end{equation}
uniformly on every compact subset of $\rr^2$. Recall that, according to the definition \cref{def_bar_x} of $\widebar{x}(\cdot)$, $\abs{u\bigl(\widebar{x}(t_n),t_n\bigr)}$ equals $\deltaStab(c_0)$ (for every nonnegative integer $n$), so that \emph{the same is true for $\abs{u_\infty(0,0)}$} (this property will be called upon at the end of the proof). 

Let us pick a quantity $c$ in the interval $\bigl(\cVar[u],\min(\widebar{c}_+,\cDecay)\bigr)$. Since $c$ is less than $\widebar{c}_+$, 
\begin{equation}
\label{lim_of_xi_of_t_subsequence}
\widebar{\xi}_c(t_n)\xrightarrow[n\to+\infty]{}+\infty
\,,
\end{equation}
and since $c$ is greater than $\cVar[u]$ but less than $\cDecay$, the quantity $E_c(t)$ is nonnegative for every nonnegative time $t$. It follows that, for every positive quantity $T$, the nonnegative quantity 
\[
E_c(t_n-T)-E_c(t_n+T)\,,
\quad\text{which, according to equality \cref{decrease_energy_trav_frame}, equals}\quad
\int_{t_n-T}^{t_n+T}D_c(t)dt
\,,
\]
goes to $0$ as $n$ goes to $+\infty$. Let us consider the function $v(\xi,t)$ defined as in \cref{change_of_variable_stand_trav_frame}. For every positive quantity $L$, the substitutions
\[
\xi = x-ct 
\quad\text{and}\quad
t = t_n + s 
\quad\text{and}\quad
x = \widebar{x}(t_n) + y 
\]
lead to:
\begin{align}
\nonumber
\int_{t_n-T}^{t_n+T}D_c(t)dt &= \int_{t_n-T}^{t_n+T}\left(\int_{\rr}e^{c\xi}v_t(\xi,t)^2 \, d\xi\right) \, dt \\
\nonumber
&= \int_{t_n-T}^{t_n+T}\left(\int_{\rr}e^{c(x-ct)}(u_t+cu_x)^2(x,t) \, dx\right) \, dt \\
\nonumber
&= \int_{-T}^{T}e^{c(\widebar{x}(t_n)-ct_n-cs)}\left(\int_{\rr}e^{cy}(u_t+cu_x)^2\bigl(\widebar{x}(t_n)+y,t_n+s\bigr) \, dy\right) \, ds \\
&\ge \int_{-T}^{T}e^{c(\widebar{x}(t_n)-ct_n-cs)}\left(\int_{-L}^{L}e^{cy}(u_t+cu_x)^2\bigl(\widebar{x}(t_n)+y,t_n+s\bigr) \, dy\right) \, ds 
\,.
\label{inequalities_proof_existence_invasion_speed}
\end{align}
Let us consider the integrals
\[
\begin{aligned}
\iii_n &= \int_{-T}^{T}\left(\int_{-L}^{L}(u_t+cu_x)^2\bigl(\widebar{x}(t_n)+y,t_n+s\bigr)\, dy\right)\, ds \\
\text{and}\quad
\iii_\infty &= \int_{-T}^{T}\left(\int_{-L}^{L}(\partial_t u_\infty+c\partial_x u_\infty)^2(y,s)\, dy\right)\, ds
\,.
\end{aligned}
\]
It follows from inequality \cref{inequalities_proof_existence_invasion_speed} that
\[
\int_{t_n-T}^{t_n+T}D_c(t)dt \ge e^{c\bigl(\widebar{x}(t_n)-ct_n-cT-L\bigr)}\iii_n = e^{c\bigl(\widebar{\xi}(t_n)-cT-L\bigr)}\iii_n
\,,
\]
and according to the limit \cref{lim_of_xi_of_t_subsequence} the exponential factor of $\iii_n$ on the right-hand side of this inequality goes to $+\infty$ as $n$ goes to $+\infty$. Thus the nonnegative quantity $\iii_n$ must go to $0$ as $n$ goes to $+\infty$. On the other hand, according to the limits \cref{compactness_proof_existence_invasion_speed}, $\iii_n$ goes to $\iii_\infty$ as $n$ goes to $+\infty$, so that $\iii_\infty$ must be equal to $0$. Since this holds for all positive quantities $T$ and $L$, the function $\partial_t u_\infty+c\partial_x u_\infty$ must be identically equal to $0_{\rr^d}$ on $\rr^2$. 

At this stage, the key observation is that, while the entire solution $u_\infty$ defined by the limits \cref{compactness_proof_existence_invasion_speed} does not depend on $c$, the previous conclusion must hold not only for one particular value of $c$, but for \emph{every} $c$ in the interval $(\widebar{c}_-,\widebar{c}_+)$. It thus follows that both functions $\partial_t u_\infty$ and $\partial_x u_\infty$ must actually be identically equal to $0_{\rr^d}$ on $\rr^2$. In other words, $u_\infty$ must be identically equal to some point of $\rr^d$, which, according to the remark made at the beginning of the proof, must be at distance $\deltaStab(c_0)$ from $0_{\rr^d}$. 

Besides, it follows from inequality \cref{lower_bound_Ec} that, for every nonnegative integer $n$, 
\begin{equation}
\label{Ec_of_zero_greater_than_alpha_H1c_integral}
E_c(0) \ge E_c(t_n) \ge -e^{c\widebar{\xi}_c(t_n)}\frac{\abs{V_{\min}}}{c} + \alpha \int_{\widebar{\xi}_c(t_n)}^{+\infty} e^{c\xi}\bigl(v_\xi(\xi,t_n)^2+v(\xi,t_n)^2\bigl)\, d\xi
\,.
\end{equation}
Let us consider the integrals
\[
\begin{aligned}
\jjj_n &= \int_{\widebar{\xi}_c(t_n)}^{+\infty} e^{c\bigl(\xi - \widebar{\xi}_c(t_n)\bigr)} \bigl(v_\xi(\xi,t_n)^2+v(\xi,t_n)^2\bigl)\, d\xi \\
\text{and}\quad
\jjj_\infty &= \int_0^{+\infty} e^{c x} \bigl(\partial_x u_\infty(x,0)^2 + u_\infty(x,0)^2\bigr) \, dx
\,.
\end{aligned}
\]
It follows from inequality \cref{Ec_of_zero_greater_than_alpha_H1c_integral} that 
\begin{equation}
\label{Ec_of_zero_greater_than_alpha_H1c_integral_bis}
\frac{1}{\alpha}\left(e^{-c\widebar{\xi}_c(t_n)} E_c(0) + \frac{\abs{V_{\min}}}{c}\right)\ge \jjj_n
\,.
\end{equation}
On the other hand, 
\[
\begin{aligned}
\jjj_n &= \int_{\widebar{x}(t_n)}^{+\infty}e^{c\bigl(x-\widebar{x}(t_n)\bigr)}\bigl(u_x(x,t_n)^2 + u(x,t_n)^2\bigr)\, dx \\
&= \int_0^{+\infty} e^{cy}(u_x^2+u^2)\bigl(\widebar{x}(t_n)+y,t_n\bigr) \, dy 
\,,
\end{aligned}
\]
so that, in view of the limits \cref{compactness_proof_existence_invasion_speed} and according to Fatou Lemma, 
\[
\jjj_\infty \le \liminf_{n\to+\infty} \jjj_n
\,,
\]
and since according to the limit \cref{lim_of_xi_of_t_subsequence} $\widebar{\xi}_c(t_n)$ goes to $+\infty$ as $n$ goes to $+\infty$, it follows from inequality \cref{Ec_of_zero_greater_than_alpha_H1c_integral_bis} that 
\[
\limsup_{n\to+\infty} \jjj_n < +\infty
\,,
\]
so that the integral $\jjj_\infty$ is finite, a contradiction with the fact that $\abs{u_\infty(\cdot,0)}$ must be identically equal to $\deltaStab(c_0)$. \Cref{prop:invasion_speed} is proved. 
\end{proof}
In the following, the positive quantity equal to $\cVar[u]$ and to $\widebar{c}_-$ and to $\widebar{c}_+$ will simply be denoted as $\widebar{c}$; with symbols,
\[
\widebar{c} = \cVar[u] = \widebar{c}_- = \widebar{c}_+
\,.
\]
\subsection{Scheme of the end of the proof}
To complete the proof, the crux is to prove that the dissipation goes to $0$, on every compact interval around the invasion point, in the frame travelling at the speed $\widebar{c}$ (\vref{prop:relaxation}); from this stage, the convergence readily follows (\vref{subsec:convergence}). 

If the converse holds (if that dissipation does not go to $0$), then there exists a sequence of times $t_n$, going to $+\infty$, at which ``some'' dissipation occurs. If, up to replacing the sequence $(t_n)$ by a subsequence, the quantities $\widebar{\xi}_{\widebar{c}}(t_n)$ are bounded from below, then reaching a contradiction is rather straightforward (see \cite[figure~2 and beginning of the proof of Proposition~4.4]{GallayRisler_globStabBistableTW_2007}): for $c$ slightly greater than $\widebar{c}$, the energy $E_c(t)$ remains bounded from below in spite of an arbitrarily large ``amount'' of dissipation, which is impossible. 

Unfortunately, if $\widebar{\xi}_{\widebar{c}}(t_n)$ goes to $-\infty$ as $n$ goes to $+\infty$, this argument fails: in a frame travelling at a constant speed, adjusted so that the invasion point is (say) at the origin $\xi=0$ at two times $t_n$ and $t_{n+p}$ (for some positive integer $p$), the energy is indeed bounded from below at time $t_{n+p}$, but, while the ``dissipation bursts'' at times $t_n$ and $t_{n+p}$ contribute to a non-negligible amount of dissipation, the other dissipation bursts occurring in between (at times $t_{n+q}$ for the integers $q$ that are positive and smaller than $p$) may be very small, since they may occur far to the left of the origin of this travelling frame, \cite[figure~4]{GallayRisler_globStabBistableTW_2007}. 

To circumvent this difficulty, the strategy proposed by Gallay and Joly in \cite{GallayJoly_globStabDampedWaveBistable_2009} is to \emph{follow} the invasion point between consecutive dissipation bursts: the benefit of this setting is that each dissipation burst contributes significantly to the decrease of the energy, and that the energy remains bounded from below, but the price to be paid is that the speed of the travelling frame must be adjusted between each dissipation burst. These speed changes induce changes in the value of the energy in the corresponding travelling frames, and these latter changes must be, asymptotically, arbitrarily small if the intended contradiction is to be reached. The crucial step is thus to obtain some control on the variation of the energy with respect to the speed (\vref{cor:Lipschitz_cont_energy}). This control will in turn follow from the key observation, made for the first time by Gallay and Joly in \cite{GallayJoly_globStabDampedWaveBistable_2009}, that, for some speed $c^*$ greater than $\widebar{c}$ but close enough to $\widebar{c}$, the \emph{energy at the invasion point} in a frame travelling at the speed $c^*$ (\cref{def:energy_at_invasion_point}) remains bounded from above (\vref{prop:upper_bound_energy_at_invasion_point}).

Finally, the proof of this \cref{prop:upper_bound_energy_at_invasion_point} follows from two arguments:
\begin{enumerate}
\item this energy (at the invasion point) has (due to Poincaré inequality \cref{Poincare_inequality_any_gamma}) the same magnitude as its ``kinetic'' part (\vref{lem:frame_Ec_with_Fc}), and, due to the parabolic system \cref{parabolic_system} satisfied by the solution, this kinetic part decreases with time, up to some ``pollution'' issued from the half space to the left of the invasion point, \vref{lem:lin_decrease_up_to_pollution_Fc};
\item the ``drift to the left'' of the invasion point (which induces an increase of this energy at the invasion point as soon as it occurs, equality \vref{def_Ec_of_xi_and_t}), is controlled, \vref{cor:control_moves_to_left_invasion_point_trav_frame}. 
\end{enumerate}
In order to state the ``linear decrease up to pollution'' \cref{lem:lin_decrease_up_to_pollution_Fc} mentioned above, another invasion point is introduced in the next \namecref{subsec:invasion_point_defined_by_smaller_radius}.
\subsection{Invasion point defined by a smaller radius}
\label{subsec:invasion_point_defined_by_smaller_radius}
Observe that the set
\begin{equation}
\label{set_defining_delta_Hess}
\left\{ \abs{w} : w\in\rr^d\text{ and } \sigma\bigl(D^2V(w)\bigr)\not\subset[\mu_0,+\infty) \right\}
\end{equation}
is nonempty; indeed, if this set was empty, then we would have, for every $w$ in $\rr^d$, 
\begin{equation}
\label{V_greater_than_one_half_hatmulow_w_square}
V(w)\ge \frac{1}{2}\mu_0 \abs{w}^2
\,,
\end{equation}
a contradiction with the fact that $\mu_0$ is greater than the quantity $\muQuadHull$ (inequalities \cref{c_0_between_cLinMax_and_cQuadHull}). Let us denote by $\deltaHess(c_0)$ the infimum of this set \cref{set_defining_delta_Hess}. Since $\sigma\bigl(D^2V(0_{\rr^d})\bigr)$ is included in $[\mu_1,+\infty)$, this quantity $\deltaHess(c_0)$ must be positive. In addition, for every $w$ in $\widebar{B}_{\rr^d}(0_{\rr^d},\deltaHess(c_0))$, 
\begin{equation}
\label{def_hat_delta}
\sigma\bigl(D^2V(w)\bigr)\subset[\mu_0,+\infty)
\,,
\end{equation}
see \cref{fig:travelling_frame_invasion_points}, and inequality \cref{V_greater_than_one_half_hatmulow_w_square} holds. According to the definition of $\deltaStab(c_0)$ (in \cref{subsec:set_up}), it follows that
\begin{equation}
\label{deltaHess_less_than_deltaStab}
0 < \deltaHess(c_0) \le \deltaStab(c_0)
\,.
\end{equation}
For every nonnegative time $t$, let us consider the set $\SigmaFar{\deltaHess(c_0)}(t)$ (see definition \cref{def_SigmaFar}). It follows from the second inequality of \cref{deltaHess_less_than_deltaStab} that this set contains the set $\SigmaFar{\deltaStab(c_0)}(t)$; it is therefore nonempty, and, for the same reason as for the set $\SigmaFar{\deltaStab(c_0)}(t)$ (namely, according to \cref{prop:decrease_energy_trav_frame}), it is bounded from above; let $\hat{x}(t)$ denote its supremum and, as in \cref{def_bar_xi_c_of_t}, let us write (for a real quantity $c$)
\[
\hat{\xi}_c(t) = \hat{x}(t) - ct
\,,
\]
see \cref{fig:travelling_frame_invasion_points}. According to these definitions, 
\[
\widebar{x}(t)\le\hat{x}(t)<+\infty
\quad\text{and}\quad
\widebar{\xi}_c(t)\le\hat{\xi}_c(t)<+\infty
\,.
\]
This new ``invasion point'' $\hat{x}(t)$ (and $\hat{\xi}_c(t)$ in a travelling frame), defined by the smaller radius $\deltaHess(c_0)$, will be used in the next two \namecrefs{subsec:control_energy_to_the_right_of_invasion_point}. The next lemma shows that this new invasion point and the previous one do not behave much differently. 
\begin{lemma}[the gap between the two invasion points is bounded]
\label{lem:gap_between_invasion_points_is_bounded}
The (nonnegative) quantity $\hat{x}(t)-\widebar{x}(t)$ is bounded, uniformly with respect to $t$ in $[0,+\infty)$. In particular, 
\[
\frac{\hat{x}(t)}{t} \to \widebar{c}
\quad\text{as}\quad
t\to+\infty
\,.
\]
\end{lemma}
\begin{proof}
Take a speed $c$ in $(c_0,\cDecay]$ and let us define the function $v(\xi,t)$ as in \cref{change_of_variable_stand_trav_frame}. For every nonnegative time $t$ (omitting the argument $(\xi,t)$ of $v$ and $v_\xi$),
\[
\begin{aligned}
E_c(t) &= \int_{-\infty}^{\widebar{\xi}_c(t)}e^{c\xi}\left(\frac{1}{2}v_\xi^2+V(v)\right)\, d\xi + \int_{\widebar{\xi}_c(t)}^{+\infty}e^{c\xi}\left(\frac{1}{2}v_\xi^2+V(v)\right)\, d\xi \\
\ge& -\frac{\abs{\Vmin}}{c}e^{c\widebar{\xi}_c(t)} + \int_{\widebar{\xi}_c(t)}^{\hat{\xi}_c(t)}e^{c\xi}\left(\frac{1}{2}v_\xi^2+\frac{1}{2}\mu_0 v^2\right)\, d\xi 
+ \int_{\hat{\xi}_c(t)}^{+\infty}e^{c\xi}\left(\frac{1}{2}v_\xi^2+\frac{1}{2}\mu_0 v^2\right)\, d\xi
\,,
\end{aligned}
\]
so that, applying Poincaré inequality \cref{Poincare_inequality_any_gamma_bounded_interval} with $\lambda$ equals $c_0/2$ on the interval $[\widebar{\xi}_c(t),\hat{\xi}_c(t)]$ and Poincaré inequality \cref{Poincare_inequality_gamma_equals_c_over_2} (that is, \cref{Poincare_inequality_any_gamma} with $\lambda$ equals $c/2$) on the interval $[\hat{\xi}_c(t),+\infty)$, it follows that
\[
\begin{aligned}
E_c(t) &\ge -\frac{\abs{\Vmin}}{c}e^{c\widebar{\xi}_c(t)} + \frac{c_0}{4}\deltaStab(c_0)^2 e^{c\widebar{\xi}_c(t)} - \frac{c_0}{4}\deltaHess(c_0)^2 e^{c\hat{\xi}_c(t)} \\
&\quad + \frac{1}{2}\left(\frac{c_0(2c-c_0)}{4}+\mu_0\right) \int_{\widebar{\xi}_c(t)}^{\hat{\xi}_c(t)}e^{c\xi} v^2\, d\xi + \frac{c}{4}\deltaHess(c_0)^2 e^{c\hat{\xi}_c(t)}
\,.
\end{aligned}
\]
Since $2c-c_0$ is greater than $c_0$ the factor of the remaining integral on the right-hand side of this inequality is nonnegative, so that
\[
E_c(t) \ge -\frac{\abs{\Vmin}}{c}e^{c\widebar{\xi}_c(t)} + \frac{c_0}{4}\deltaStab(c_0)^2 e^{c\widebar{\xi}_c(t)} + \frac{c-c_0}{4}\deltaHess(c_0)^2 e^{c\hat{\xi}_c(t)} 
\]
Let us assume that $t$ is large enough (positive) so that $\widebar{x}(t)/t$ is in $(c_0,\cDecay]$ and let us choose $c$ equal to $\widebar{x}(t)/t$. It follows that
\[
\widebar{\xi}_c(t) = 0
\quad\text{thus}\quad
\hat{\xi}_c(t) = \hat{\xi}_c(t) - \widebar{\xi}_c(t) = \hat{x}(t) - \widebar{x}(t)
\,,
\]
so that the previous inequality reads
\[
E_c(t) \ge -\frac{\abs{\Vmin}}{c} + \frac{c_0}{4}\deltaStab(c_0)^2 + \frac{c-c_0}{4}\deltaHess(c_0)^2 e^{c\bigl(\hat{x}(t) - \widebar{x}(t)\bigr)}
\,.
\]
Equivalently, 
\[
\hat{x}(t) - \widebar{x}(t) \le \frac{1}{c}\ln\Biggl(\frac{4}{(c-c_0)\deltaHess(c_0)^2}\biggl(E_c(t) + \frac{\abs{\Vmin}}{c} - \frac{c_0}{4}\deltaStab(c_0)^2 \biggr)\Biggr)
\,,
\]
and the argument of the logarithm must be positive. Since $E_c(t)$ is less than or equal to $E_c(0)$ and since $\widebar{x}(t)/t$ goes to $\widebar{c}$ as $t$ goes to $+\infty$, it follows that
\[
\limsup_{t\to+\infty} \hat{x}(t) - \widebar{x}(t) \le \frac{1}{\widebar{c}}\ln\Biggl(\frac{4}{(\widebar{c}-c_0)\deltaHess(c_0)^2}\biggl(E_{\widebar{c}}(0) + \frac{\abs{\Vmin}}{\widebar{c}} \biggr)\Biggr) < +\infty
\,.
\]
Thus there exists a positive time $T$ such that the quantity $\hat{x}(t) - \widebar{x}(t)$ is bounded, uniformly with respect to $t$ in $[T,+\infty)$. 

On the other hand, the function $t\mapsto\hat{x}(t)$ is bounded from above on the bounded interval $[0,T]$ (the reason is the same as for $t\mapsto\widebar{x}(t)$, see the proof of \cref{lem:upper_bound_bounded_time_intervals_invasion_point}). Since according to \cref{lem:lower_bound_invasion_point} the function $t\mapsto\widebar{x}(t)$ is bounded from below on $[0,T]$, the same is true for the difference $\hat{x}(t) - \widebar{x}(t)$, and the conclusion follows. 
\end{proof}
\subsection{Delayed control of the energy to the right of the invasion point}
\label{subsec:control_energy_to_the_right_of_invasion_point}
For $c$ in $(0,\cDecay]$, let us consider the function $t\mapsto F_c(t)$, defined on $[0,+\infty)$ as
\[
F_c(t) = \int_{\rr} e^{c\xi}\frac{1}{2} v_\xi(\xi,t)^2\, d\xi
\,,
\]
where the function $v(\xi,t)$ is defined as in \cref{change_of_variable_stand_trav_frame}. According to the second assumption of \cref{assumptions_origin_of_times} and to \cref{prop:decrease_energy_trav_frame}, the function $x\mapsto u(x,t)$ (and thus the function $\xi\mapsto v(\xi,t))$ is in $H^1_c(\rr,\rr^d)$ for every nonnegative time $t$, so that $F_c(t)$ is well defined (and finite). The reason for introducing this function $F_c(\cdot)$ is that it satisfies the ``linear decrease up to pollution'' property \cref{lin_decrease_up_to_pollution_Fc} stated by the next \cref{lem:lin_decrease_up_to_pollution_Fc}, which as a consequence will provide some control over the energy function $E_c(t)$, as stated in \cref{cor:control_on_Ec_induced_by_Fc} below. It would be more straightforward if this ``linear decrease up to pollution'' held directly for $E_c(t)$, as it happens when the invaded equilibrium is stable, see \cite[inequality~(2.14)]{GallayRisler_globStabBistableTW_2007}. Unfortunately this does not seem to hold in the present context where the invaded equilibrium is not stable, see the remark in the proof of \cref{lem:lin_decrease_up_to_pollution_Fc} below. Introducing the function $F_c(t)$ is thus a way to circumvent this difficulty. 

Let us recall that, according to inequality \cref{cLow_smaller_than_c_minus_and_c_plus_smaller_than_cUpp}, the speed $c_0$ is smaller than the invasion speed $\widebar{c}$, and let us consider quantities $c_0'$ and $\mu_0'$ satisfying:
\[
c_0' \in(c_0,\widebar{c})
\quad\text{and}\quad
\mu_0' = -\frac{(c_0')^2}{4}
\,,
\quad\text{so that}\quad
\mu_0' < \mu_0 
\quad\text{and}\quad
\abs{\mu_0'} > \abs{\mu_0}
\,,
\]
see \vref{fig:speeds}. The value of $c_0'$ does not matter much, provided that it is in the interval $(c_0,\widebar{c})$; for instance $c_0'$ can be chosen as the mean of $c_0$ and $\widebar{c}$. The following lemma is the ``parabolic'' analogue of \cite[inequality~(6.7)]{GallayJoly_globStabDampedWaveBistable_2009}, and of \cite[inequality~(4.9) and the inequality at the bottom of p.~276]{Luo_globStabDampedWaveEqu_2013}, and the ``unstable invaded equilibrium'' version of \cite[inequality~(2.14)]{GallayRisler_globStabBistableTW_2007}. 
\begin{lemma}[linear decrease up to pollution for $F_c(t)$]
\label{lem:lin_decrease_up_to_pollution_Fc}
There exist positive quantities $\nu$ and $K_F$ (depending only on $V$ and $c_0$ and $c_0'$) such that, for every $c$ in $[c_0',\cDecay]$ and for every nonnegative time $t$, the following inequality holds:
\begin{equation}
\label{lin_decrease_up_to_pollution_Fc}
F_c'(t) \le -\nu F_c(t) + K_F e^{c\hat{\xi}_c(t)}
\,.
\end{equation}
\end{lemma}
\begin{proof}
It follows from the expression of $F_c(t)$ that, for every nonnegative time $t$ (omitting the arguments $(\xi,t)$ of $v$ and its partial derivatives and proceeding as in \cref{subsec:expression_dissipation_travelling_frame}, and according to the assumptions made in \cref{subsec:set_up} on the origin of times), 
\[
\begin{aligned}
F_c'(t) &= \int_{\rr} e^{c\xi} v_\xi\cdot v_{\xi t}\, d\xi \\
&= -\int_{\rr} e^{c\xi} (cv_\xi + v_{\xi\xi}) \cdot v_t \, d\xi \\
&= -\int_{\rr} e^{c\xi} (cv_\xi + v_{\xi\xi}) \cdot \bigl(cv_\xi + v_{\xi\xi} - \nabla V(v)\bigr) \, d\xi \\
&= -\int_{\rr} e^{c\xi} \Bigl((cv_\xi + v_{\xi\xi})\cdot \bigl(cv_\xi - \nabla V(v) \bigr) + (c v_\xi + v_{\xi\xi})\cdot v_{\xi\xi}\Bigr) \, d\xi \\
&= -\int_{\rr} e^{c\xi} \Bigl(v_\xi\cdot\bigl(-cv_{\xi\xi}+D^2V(v)\cdot v_\xi\bigr) + (c v_\xi + v_{\xi\xi})\cdot v_{\xi\xi}\Bigr) \, d\xi \\
&= -\int_{\rr} e^{c\xi} \bigl(v_{\xi\xi}^2 + D^2V(v)\cdot v_\xi\cdot v_\xi \bigr) \, d\xi
\,.
\end{aligned}
\]
According to \cref{prop:decrease_energy_trav_frame}, the function $v(\cdot,t)$ is in $H_c^2(\rr,\rr^d)$, so that Poincaré inequality \cref{Poincare_inequality_on_R} applies to the function $v_\xi(\cdot,t)$, leading to
\[
\begin{aligned}
F_c'(t) &\le -\int_{\rr} e^{c\xi} \left(\frac{c^2}{4}v_{\xi}^2 + D^2V(v)\cdot v_\xi\cdot v_\xi \right) \, d\xi \\
&\le -\int_{\rr} e^{c\xi} \left(\abs{\mu_0'}v_{\xi}^2 + D^2V(v)\cdot v_\xi\cdot v_\xi \right) \, d\xi
\,.
\end{aligned}
\]
\begin{remark}
Observe that this last expression is easier to handle than expression \cref{expression_ddd_c_of_v} of $\ddd_c[\cdot]$ obtained in \cref{subsec:expression_dissipation_travelling_frame}: here the ``bad'' term $D^2V(v)\cdot v_\xi\cdot v_\xi$ (``bad'' since it may be as negative as roughly $\mu_1 v_\xi^2$ when $v$ is small) is, fortunately, not strengthened by a ``$2$'' factor as in the expression of $\ddd_c[\cdot]$. 
\end{remark}
According to the property \cref{def_hat_delta} defining $\deltaHess(c_0)$, it follows from the previous inequality that
\[
F_c'(t) \le - \int_{-\infty}^{\hat{\xi}_c(t)} e^{c\xi} \left(\abs{\mu_0'}v_{\xi}^2 + D^2V(v)\cdot v_\xi\cdot v_\xi \right) \, d\xi - \int_{\hat{\xi}_c(t)}^{+\infty}e^{c\xi} \bigl(\abs{\mu_0'} -\abs{\mu_0}\bigr)v_{\xi}^2 \, d\xi
\,,
\]
so that, for every positive quantity $\nu$, 
\begin{equation}
\label{F_prime_c_plus_nu_Fc_less_than}
\begin{aligned}
F_c'(t) + \nu F_c(t) &\le - \int_{-\infty}^{\hat{\xi}_c(t)}e^{c\xi} \Biggl(\left(\abs{\mu_0'} - \frac{\nu}{2}\right)v_\xi^2 + D^2V(v)\cdot v_\xi\cdot v_\xi\Biggr) \, d\xi  \\
&\quad - \int_{\hat{\xi}_c(t)}^{+\infty}e^{c\xi} \left(\abs{\mu_0'} -\abs{\mu_0} - \frac{\nu}{2}\right)v_{\xi}^2 \, d\xi
\,.
\end{aligned}
\end{equation}
Thus, if the (positive) quantity $\nu$ is chosen as
\[
\nu = 2\bigl(\abs{\mu_0'} -\abs{\mu_0}\bigr)
\,,
\quad\text{so that}\quad
\frac{\nu}{2} = \abs{\mu_0'} -\abs{\mu_0} < \abs{\mu_0'}
\,,
\]
then it follows from inequality \cref{F_prime_c_plus_nu_Fc_less_than} that
that
\[
F_c'(t) + \nu F_c(t) \le - \int_{-\infty}^{\hat{\xi}_c(t)}e^{c\xi} D^2V(v)\cdot v_\xi\cdot v_\xi \, d\xi
\,,
\]
and inequality \cref{lin_decrease_up_to_pollution_Fc} follows from the bound \cref{bound_u_ux_setup} on the solution. 
\end{proof}
\begin{lemma}[framing of $E_c(t)$ with $F_c(t)$]
\label{lem:frame_Ec_with_Fc}
There exist positive quantities $C_1$ and $C_2$ (depending only on $V$ and $c_0$ and $c_0'$) such that, for every $c$ in $[c_0',\cDecay]$ and for every nonnegative time $t$, the following inequality holds:
\begin{equation}
\label{frame_Ec_with_Fc}
\frac{1}{C_1}F_c(t) - C_2 e^{c\hat{\xi}_c(t)} \le E_c(t) \le C_1 F_c(t) + C_2 e^{c\hat{\xi}_c(t)}
\,.
\end{equation}
\end{lemma}
\begin{proof}
Let $\varepsilon$ denote quantity in $(0,1]$. For every nonnegative time $t$ (omitting the argument $(\xi,t)$ of $v$ and $v_\xi$ in the integrands and using the Poincaré inequality \cref{Poincare_inequality_gamma_equals_c_over_2}),
\[
\begin{aligned}
E_c(t) - \varepsilon F_c(t) &= \int_{-\infty}^{\hat{\xi}_c(t)}\left(\frac{1}{2}(1-\varepsilon)v_\xi^2 + V(v)\right)\, d\xi + \int_{\hat{\xi}_c(t)}^{+\infty}\left(\frac{1}{2}(1-\varepsilon)v_\xi^2 + V(v)\right)\, d\xi \\
&\ge - \frac{\abs{\Vmin}}{c} e^{c\hat{\xi}_c(t)} + \frac{1}{2}\int_{\hat{\xi}_c(t)}^{+\infty}\bigl((1-\varepsilon)v_\xi^2 -\abs{\mu_0} v^2 \bigr)\, d\xi \\
&\ge - \frac{\abs{\Vmin}}{c} e^{c\hat{\xi}_c(t)} + \frac{1}{2}\bigl((1-\varepsilon)\abs{\mu_0'}-\abs{\mu_0}\bigr)\int_{\hat{\xi}_c(t)}^{+\infty}v^2 \, d\xi
\,,
\end{aligned}
\]
so that if $\varepsilon$ is chosen as
\[
\varepsilon = 1 - \frac{\abs{\mu_0}}{\abs{\mu_0'}}
\,,
\]
then the factor of the integral of this last expression vanishes. On the other hand, introducing the quantities $\mu_{0,\max}$ and $\Vmax$ defined as
\[
\mu_{0,\max} = \max_{w\in\widebar{B}_{\rr^d}(0_{\rr^d},\deltaHess(c_0))}\max\sigma\bigl(D^2V(w)\bigr)
\quad\text{and}\quad
\Vmax = \max_{w\in\rr^d,\, \abs{w}\le \Ratt} V(w)
\,,
\]
it follows from \cref{bound_u_ux_setup} that
\[
\begin{aligned}
E_c(t) &\le \frac{\Vmax}{c}  e^{c\hat{\xi}_c(t)} + \int_{-\infty}^{\hat{\xi}_c(t)}\frac{1}{2}v_\xi^2\, d\xi + \int_{\hat{\xi}_c(t)}^{+\infty}\left(\frac{1}{2}v_\xi^2 + \frac{1}{2}\mu_{0,\max} v^2 \right)\, d\xi \\
&\le \frac{\Vmax}{c}  e^{c\hat{\xi}_c(t)} + \int_{-\infty}^{\hat{\xi}_c(t)}\frac{1}{2}v_\xi^2\, d\xi + \frac{1}{2}\left(1 + \max(\mu_{0,\max},0)\frac{4}{c^2}\right)\int_{\hat{\xi}_c(t)}^{+\infty}v_\xi^2\, d\xi \\
&\le \frac{\Vmax}{c}  e^{c\hat{\xi}_c(t)} + \frac{1}{2} \left( 1+ \frac{\max(\mu_{0,\max},0)}{\abs{\mu_0'}}\right)F_c(t)
\,.
\end{aligned}
\]
Finally, if $C_1$ and $C_2$ are chosen as
\[
\begin{aligned}
C_1 &= \max\Biggl(\frac{\abs{\mu_0'}}{\abs{\mu_0'}-\abs{\mu_0}},\frac{1}{2} \left( 1+ \frac{\max(\mu_{0,\max},0)}{\abs{\mu_0'}}\right)\Biggr)\,, \\
\text{and}\quad
C_2 &= \max\left(\frac{\abs{\Vmin}}{c_0'},\frac{\max(\Vmax,0)}{c_0'}\right)
\end{aligned}
\]
(these quantities depend only on $V$ and $c_0$ and $c_0'$), then inequality \cref{frame_Ec_with_Fc} holds. 
\end{proof}
The following corollary of \cref{lem:lin_decrease_up_to_pollution_Fc,lem:frame_Ec_with_Fc} is the analogue of \cite[Lemma~6.4]{GallayJoly_globStabDampedWaveBistable_2009} and \cite[Lemma~4.2]{Luo_globStabDampedWaveEqu_2013}.
\begin{corollary}[delayed control of the energy to the right of the invasion point]
\label{cor:control_on_Ec_induced_by_Fc}
There exist positive quantities $K$ and $K'$ (depending only on $V$ and $c_0$ and $c_0'$) such that, for every $c$ in $[c_0',\cDecay]$ and for all nonnegative times $t$ and $T$, the following inequality holds:
\begin{equation}
\label{control_on_Ec_induced_by_Fc}
E_c(t+T) \le K e^{-\nu T}E_c(t) + K' e^{c\hat{\xi}_{c,\sup}(t,T)}\,,
\quad\text{where}\quad
\hat{\xi}_{c,\sup}(t,T) = \sup_{t\le s\le t+T}\hat{\xi}_c(s)
\,.
\end{equation}
\end{corollary}
\begin{proof}
For all nonnegative times $t$ and $T$, with the notation $\hat{\xi}_{c,\sup}(t,T)$ of \cref{control_on_Ec_induced_by_Fc}, it follows from \cref{lem:lin_decrease_up_to_pollution_Fc} that
\[
F_c(t+T) \le e^{-\nu T}F_c(t) + \frac{K_F}{\nu} e^{c\hat{\xi}_{c,\sup}(t,T)}
\,,
\]
so that, according to inequality \cref{frame_Ec_with_Fc} of \cref{lem:frame_Ec_with_Fc},
\[
\begin{aligned}
E_c(t+T) &\le C_1 F_c(t+T) + C_2 e^{c\hat{\xi}_{c,\sup}(t,T)} \\
&\le C_1 e^{-\nu T}F_c(t) + \left(\frac{C_1 K_F}{\nu}+ C_2\right)e^{c\hat{\xi}_{c,\sup}(t,T)} \\
&\le C_1^2 e^{-\nu T}E_c(t) + \left(C_1^2C_2 + \frac{C_1 K_F}{\nu} + C_2\right)e^{c\hat{\xi}_{c,\sup}(t,T)}
\,,
\end{aligned}
\]
and so that, choosing the quantities $K$ and $K'$ as
\[
K = C_1^2
\quad\text{and}\quad
K' = C_1^2C_2 + \frac{C_1 K_F}{\nu} + C_2
\]
(these quantities depend only on $V$ and $c_0$ and $c_0'$), inequality \cref{control_on_Ec_induced_by_Fc} follows. 
\end{proof}
\subsection{Control on the left drift of the invasion point}
\label{subsec:control_moves_to_the_left_invasion_point}
The following notation is the ``parabolic'' analogue of the one introduced in \cite[(4.7)]{GallayJoly_globStabDampedWaveBistable_2009} and \cite[(1.14)]{Luo_globStabDampedWaveEqu_2013}.
\begin{notation}
For every speed $c$ in $(0,\cDecay]$, every nonnegative time $t$ and every real quantity $\xi$, if $v(\cdot,\cdot)$ denotes the function introduced in \cref{change_of_variable_stand_trav_frame}, let us consider the quantities
\[
\begin{aligned}
E_c(\xi,t) &= \eee_c[v(\xi+\cdot,t)] = \int_{\rr}e^{c\zeta}\left(\frac{1}{2}v_\xi(\xi+\zeta,t)^2 + V\bigl(v(\xi+\zeta,t)\bigr)\right)\, d\zeta \\
\text{and}\quad
D_c(\xi,t) &= \int_{\rr}e^{c\zeta} v_t(\xi+\zeta,t)^2\, d\zeta
\,.
\end{aligned}
\] 
\end{notation}
These quantities differ from the energy $E_c(t)$ and the dissipation $D_c(t)$ defined in \cref{def_Dc_of_t,def_Ec_of_t} only by the fact that their exponential weight is normalized so that it takes the value $1$ at the value $\xi$ (rather than $0$) of the travelling abscissa, and they are related to these initial energy and dissipation by:
\begin{align}
\label{def_Ec_of_xi_and_t}
E_c(\xi,t) &= e^{-c\xi} E_c(0,t) = e^{-c\xi} E_c(t) \\
\text{and}\quad
D_c(\xi,t) &= e^{-c\xi} D_c(0,t) = e^{-c\xi} D_c(t)
\,,
\label{def_Dc_of_xi_and_t}
\end{align}
so that, according to \cref{prop:decrease_energy_trav_frame}, 
\[
\partial_t E_c(\xi,t) = - D_c(\xi,t)
\,.
\]
In addition, the two-variable function $(\xi,t)\mapsto E_c(\xi,t)$ is:
\begin{itemize}
\item non-increasing with respect to the time variable $t$,
\item and exponentially decreasing, in size, with respect to the space variable $\xi$,
\end{itemize}
see \cref{fig:behaviour_Ec_of_xi_t}.
\begin{figure}[!htbp]
\centering
\includegraphics[width=.5\textwidth]{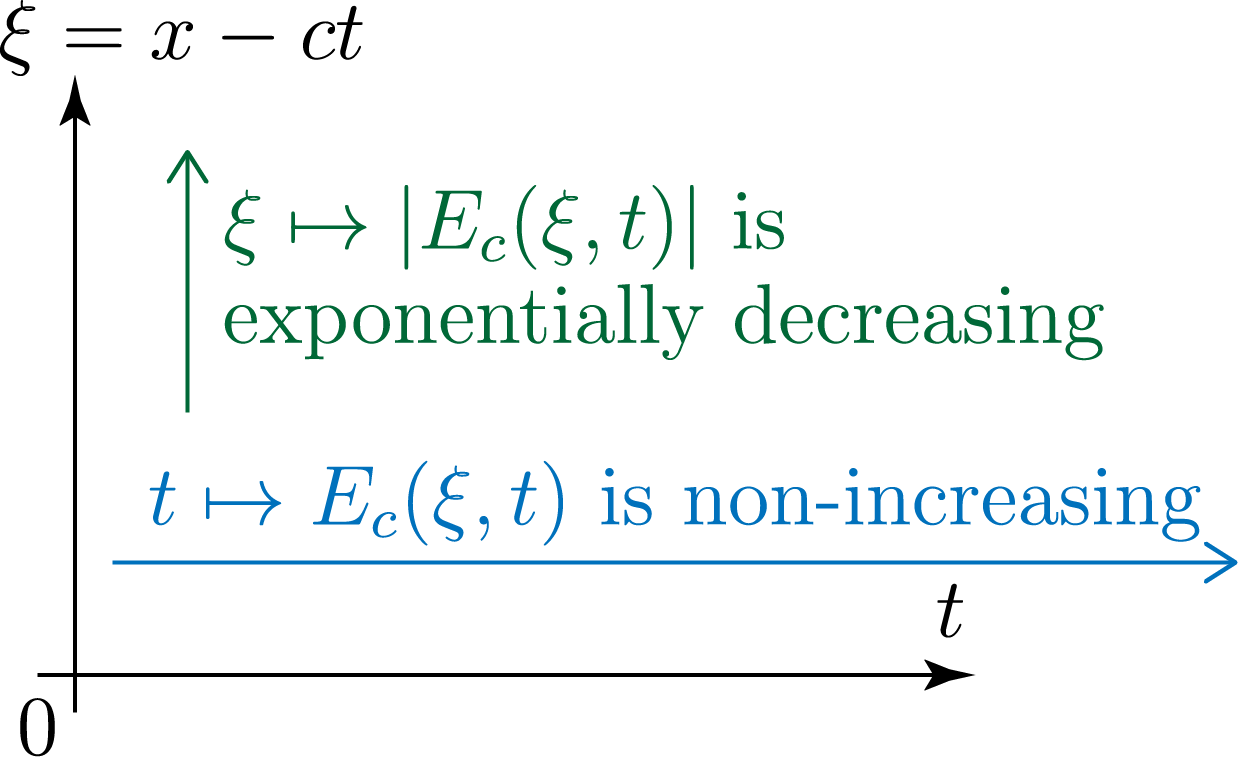}
\caption{Monotonicity of the function $(\xi,t)\mapsto E_c(\xi,t)$ with respect to its arguments.}
\label{fig:behaviour_Ec_of_xi_t}
\end{figure}
The following lemma is the analogue of \cite[Lemma~6.2]{GallayJoly_globStabDampedWaveBistable_2009}.
\begin{lemma}[control on the left drift of the invasion point in a frame travelling at a speed which is less than $\widebar{c}$] 
\label{lem:almost_non_decrease_invasion_point}
For every speed $c$ in $(-\infty,\widebar{c})$,
\begin{equation}
\label{inf_hat_xi_c_of_tprime_minus_hat_xi_c_of_t_finite}
\inf_{0\le t\le t'}\hat{\xi}_c(t') - \hat{\xi}_c(t) > -\infty
\,.
\end{equation}
\end{lemma}
\begin{proof}
If inequality \cref{inf_hat_xi_c_of_tprime_minus_hat_xi_c_of_t_finite} holds for some speed $c$, then it also holds for every speed $\cLess$ in $(-\infty,c]$; indeed, for such speeds $c$ and $\cLess$ and for all times $t$ and $t'$ satisfying $0\le t\le t'$, 
\begin{equation}
\label{hat_xi_cLess_of_tPrime_minus_hat_xi_cLess_of_t}
\begin{aligned}
\hat{\xi}_{\cLess}(t') - \hat{\xi}_{\cLess}(t) &= \hat{x}(t') - \hat{x}(t) - \cLess(t'-t) \\
&= \hat{\xi}_c(t') - \hat{\xi}_c(t) + (c-\cLess)(t'-t) \\
&\ge \hat{\xi}_c(t') - \hat{\xi}_c(t)
\,.
\end{aligned}
\end{equation}
Thus, it is sufficient to prove inequality \cref{inf_hat_xi_c_of_tprime_minus_hat_xi_c_of_t_finite} for $c$ almost equal to $\widebar{c}$ (unsurprisingly, it is for such speeds that this conclusion will called upon later), and in particular for $c$ greater than or equal to $c_0$. Let us proceed by contradiction and assume that there exists some speed $c$ in $[c_0,\widebar{c})$ and sequences of times $(t_n)_{n\in\nn}$ and $(t_n')_{n\in\nn}$ such that
\begin{align}
\label{tn_less_than_or_equal_to_tprimen}
& 0\le t_n\le t_n'\text{ for every nonnegative integer $n$}\\
\text{and}\quad
& \hat{\xi}_c(t_n')-\hat{\xi}_c(t_n)\to -\infty
\quad\text{as}\quad
n\to+\infty
\,,
\label{hat_xi_c_tn_minus_tn_prime_goes_to_infty}
\end{align}
see \cref{fig:proof_lem_almost_non_decrease_invasion_point}.
\begin{figure}[!htbp]
\centering
\includegraphics[width=\textwidth]{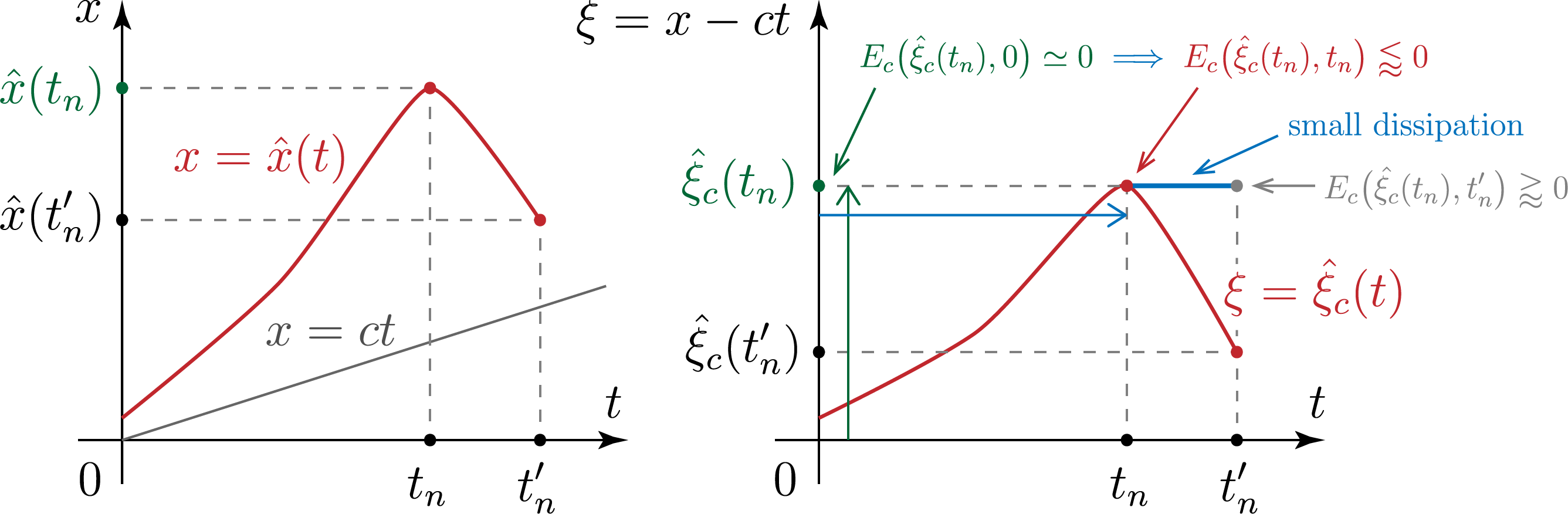}
\caption{Illustration of the proof of \cref{lem:almost_non_decrease_invasion_point}.}
\label{fig:proof_lem_almost_non_decrease_invasion_point}
\end{figure}
According to \cref{lem:gap_between_invasion_points_is_bounded}, the limit \cref{hat_xi_c_tn_minus_tn_prime_goes_to_infty} still holds if $\hat{\xi}_c(\cdot)$ is replaced with $\widebar{\xi}_c(\cdot)$, namely:
\begin{equation}
\label{bar_xi_c_tn_minus_tn_prime_goes_to_infty}
\widebar{\xi}_c(t_n)-\widebar{\xi}_c(t_n')\to +\infty
\quad\text{as}\quad
n\to+\infty
\,,
\end{equation}
Now it follows from the inequality \cref{tn_less_than_or_equal_to_tprimen}, from the limit \cref{bar_xi_c_tn_minus_tn_prime_goes_to_infty}, and from \cref{lem:lower_bound_invasion_point,lem:upper_bound_bounded_time_intervals_invasion_point} that $t_n'$ must go to $+\infty$ as $n$ goes to $+\infty$. On the other hand, since $c$ is less then $\widebar{c}$, $\widebar{\xi}_c(t)$ goes to $+\infty$ as $t$ goes to $+\infty$, and it follows that $\widebar{\xi}_c(t_n')$ goes to $+\infty$ as $n$ goes to $+\infty$ (and the same holds for $\hat{\xi}_c(t_n')$). Thus $\widebar{\xi}_c(t_n)$ also goes to $+\infty$ as $n$ goes to $+\infty$ (and the same holds for $\hat{\xi}_c(t_n)$), and again due to \cref{lem:upper_bound_bounded_time_intervals_invasion_point}, it follows that $t_n$ goes to $+\infty$ as $n$ goes to $+\infty$. In short, the three quantities
\[
t_n
\quad\text{and}\quad
\hat{\xi}_c(t_n)
\quad\text{and}\quad
\hat{\xi}_c(t_n)-\hat{\xi}_c(t_n')
\]
go to $+\infty$ as $n$ goes to $+\infty$, see \cref{fig:proof_lem_almost_non_decrease_invasion_point}. By compactness (\cref{lem:compactness}), up to replacing the sequence $(t_n)_{n\in\nn}$ by a subsequence, there exists an entire solution $u_\infty$ of system \cref{parabolic_system} such that, with the notation of \cref{subsec:compactness}, 
\begin{equation}
\label{compactness_proof_almost_non_decrease_invasion_point}
D^{2,1}u\bigl(\hat{x}(t_n)+\cdot,t_n+\cdot\bigr)\to D^{2,1}u_\infty
\quad\text{as}\quad
n\to+\infty
\,,
\end{equation}
uniformly on every compact subset of $\rr^2$. Since 
\[
E_c\bigl(\hat{\xi}_c(t_n),0\bigr) = \exp\bigl(-\hat{\xi}_c(t_n)\bigr) E_c(0,0)
\,,
\]
it follows that
\[
E_c\bigl(\hat{\xi}_c(t_n),0\bigr)\to 0
\quad\text{as}\quad
n\to+\infty
\,,
\]
see \cref{fig:proof_lem_almost_non_decrease_invasion_point}; and since, according to \cref{prop:decrease_energy_trav_frame}, the function $t\mapsto E_c(\hat{\xi}_c(t_n),t)$ is non-increasing, it follows that 
\[
\limsup_{n\to+\infty} E_c\bigl(\hat{\xi}_c(t_n),t_n\bigr) \le 0
\,,
\]
see \cref{fig:proof_lem_almost_non_decrease_invasion_point}. On the other hand, it follows from inequality \cref{lower_bound_Ec} (which holds since $c$ was assumed to be in $[c_0,\widebar{c})$) that
\[
\begin{aligned}
E_c\bigl(\hat{\xi}_c(t_n),t_n'\bigr) &= e^{-c\hat{\xi}_c(t_n)}E_c\bigl(0,t_n'\bigr) \\
&\ge -e^{c\bigl(\widebar{\xi}_c(t'_n)-\hat{\xi}_c(t_n)\bigr)}\frac{\abs{\Vmin}}{c} \\
&\ge -e^{c\bigl(\hat{\xi}_c(t'_n)-\hat{\xi}_c(t_n)\bigr)}\frac{\abs{\Vmin}}{c}
\,,
\end{aligned}
\]
so that, according to the limit \cref{hat_xi_c_tn_minus_tn_prime_goes_to_infty},
\[
\liminf_{n\to+\infty} E_c\bigl(\hat{\xi}_c(t_n),t_n'\bigr) \ge 0
\,,
\]
see \cref{fig:proof_lem_almost_non_decrease_invasion_point}. It follows that the (nonnegative) quantity
\[
E_c\bigl(\hat{\xi}_c(t_n),t_n\bigr) - E_c\bigl(\hat{\xi}_c(t_n),t_n'\bigr)
\,,
\ \text{which, according to \cref{decrease_energy_trav_frame}, equals}\ 
\int_{t_n}^{t_n'} D_c\bigl(\hat{\xi}_c(t_n),t\bigr)\, dt
\,,
\]
must go to $0$ as $n$ goes to $+\infty$, see \cref{fig:proof_lem_almost_non_decrease_invasion_point}. Proceeding as in the proof of \cref{prop:invasion_speed}, it follows that the function $\partial_t u_\infty+c\partial_x u_\infty$ must be identically equal to $0$ on $\rr\times[0,+\infty)$. 

The key observation is that, for every speed $c'$ in $(c,\widebar{c})$, the following limits ``still'' hold:
\[
\hat{\xi}_{c'}(t_n)\to +\infty
\quad\text{and}\quad
\hat{\xi}_{c'}(t_n)-\hat{\xi}_{c'}(t_n')\to +\infty
\quad\text{as}\quad
n\to+\infty
\,,
\]
so that, repeating the same argument, the function $\partial_t u_\infty+c'\partial_x u_\infty$ must (still) be identically equal to $0$ on $\rr\times[0,+\infty)$. It thus follows that both functions $\partial_t u_\infty$ and $\partial_x u_\infty$ must actually vanish on $\rr\times[0,+\infty)$, and the same contradiction as in the proof of \cref{prop:invasion_speed} follows. \Cref{lem:almost_non_decrease_invasion_point} is proved. 
\end{proof}
The following corollary is the analogue of \cite[Corollary~6.3]{GallayJoly_globStabDampedWaveBistable_2009}. 
\begin{corollary}[control on the left drift of the invasion point in a frame travelling at a speed which is slightly greater than $\widebar{c}$] 
\label{cor:control_moves_to_left_invasion_point_trav_frame}
For every positive quantity $\gamma$, there exists a positive quantity $\hatXiLeft(\gamma)$ such that, for every speed $c$ in $(-\infty,\widebar{c}+\gamma]$ and for all times $t$ and $t'$ satisfying $0\le t\le t'$, 
\begin{equation}
\label{control_moves_to_left_invasion_point_trav_frame}
\hat{\xi}_c(t')\ge \hat{\xi}_c(t) - 2\gamma(t'-t) - \hatXiLeft(\gamma)
\,.
\end{equation}
\end{corollary}
\begin{proof}
For every positive quantity $\gamma$ and for all times $t$ and $t'$ satisfying $0\le t\le t'$, arguing as in \cref{hat_xi_cLess_of_tPrime_minus_hat_xi_cLess_of_t}, 
\[
\begin{aligned}
\hat{\xi}_c(t') - \hat{\xi}_c(t) + 2\gamma(t'-t) &= \hat{\xi}_{c-2\gamma}(t') - \hat{\xi}_{c-2\gamma}(t) \\
&\ge \hat{\xi}_{\widebar{c}-\gamma}(t') - \hat{\xi}_{\widebar{c}-\gamma}(t) 
\,,
\end{aligned}
\]
and the intended conclusion follows from the conclusion of \cref{lem:almost_non_decrease_invasion_point} for the speed $\widebar{c}-\gamma$. 
\end{proof}
\subsection{Lipschitz continuity with respect to speed of the energy at the invasion point}
\label{subsec:Lipschitz_continuity}
\subsubsection{The energy at the invasion point}
\begin{definition}[energy at the invasion point]
\label{def:energy_at_invasion_point}
For every speed $c$ in $(0,\cDecay]$ and for every nonnegative time $t$, let us call \emph{energy at the invasion point} (at time $t$, in a frame travelling at the speed $c$) the quantity $\hat{E}_c(t)$ defined as
\begin{equation}
\label{def_hat_Ec_of_t}
\hat{E}_c(t) = E_c\bigl(\hat{\xi}_c(t),t\bigr) = e^{-c\hat{\xi}_c(t)} E_c(0,t) = e^{-c\hat{\xi}_c(t)} E_c(t) 
\,.
\end{equation}
Among the family $E_c(\xi,t)$ of energies, this specific energy $\hat{E}_c(t)$ is characterized by an exponential weight which is normalized so that it takes the value $1$ at the invasion point $\hat{\xi}_c(t)$. 
\end{definition}
The aim of this \namecref{subsec:Lipschitz_continuity} is to prove \cref{cor:Lipschitz_cont_energy} in \cref{subsubsec:Lipschitz_continuity} below, which states that this ``energy at the invasion point'' is Lipschitz continuous with respect to the speed $c$, uniformly in time. This continuity property is key for the relaxation scheme set up in the next \namecref{subsec:relaxation}. The main step leading to \cref{cor:Lipschitz_cont_energy} is \cref{prop:upper_bound_energy_at_invasion_point}, which states that the energy at the invasion point $\hat{E}_{c^*}(t)$ is bounded from above, uniformly with respect to $t$, for a speed $c^*$ slightly greater than $\widebar{c}$ to be chosen below.
\subsubsection{Choice of a speed slightly greater than the invasion speed}
Let $\gamma$ denote a (small) positive quantity to be chosen below, and let us consider the speed $c^*$ defined as
\[
c^* = \widebar{c}+\gamma
\,,
\]
see \cref{fig:speeds}. In order to obtain an upper bound on $\hat{E}_{c^*}(t)$, the rate $\nu$ involved in the exponential decrease of $F_{c^*}(t)$ (\cref{lem:lin_decrease_up_to_pollution_Fc}) and thus (on the long term) of $E_{c^*}(t)$ (\cref{cor:control_on_Ec_induced_by_Fc}) must balance the (possible) increase of $\hat{E}_{c^*}(t)$ due to the (possible) drift to the left of the invasion point $\hat{\xi}_{c^*}(t)$; according to \cref{cor:control_moves_to_left_invasion_point_trav_frame} this drift to the left does not occur, on the long term, at a speed larger than $2\gamma$ provided that $c$ is less than or equal to $c^*$, and a drift to the left at the speed $2\gamma$ induces an increase rate equal to $2\gamma c^*$. This leads us to introduce the positive quantity $\gamma$ defined as
\begin{equation}
\label{def_eta}
\gamma = \min\left(\frac{-\widebar{c}+\sqrt{\widebar{c}^2+\nu}}{2},\cDecay-\widebar{c}\right)
\,.
\end{equation}
This choice ensures that
\begin{equation}
\label{c_star_less_than_cUpp_and_2_eta_c_star_less_than_nu_over_2}
\gamma>0 
\quad\text{and}\quad
\widebar{c}<c^* = \widebar{c}+\gamma\le\cDecay
\quad\text{and}\quad
2\gamma c^*  \le \frac{\nu}{2}
\,,
\end{equation}
see \cref{fig:speeds}.
\begin{figure}[!htbp]
\centering
\includegraphics[width=.6\textwidth]{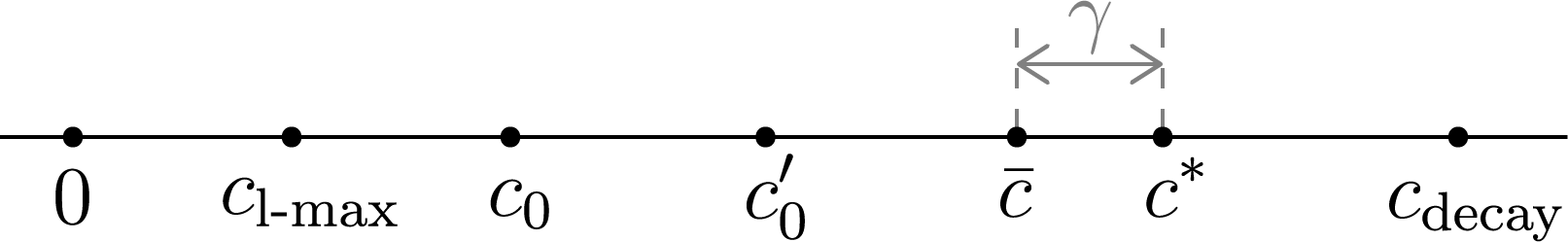}
\caption{Speeds involved in the proof.}
\label{fig:speeds}
\end{figure}
\subsubsection{Factor two shrinkage of the energy at the invasion point}
According to the last inequality of \cref{c_star_less_than_cUpp_and_2_eta_c_star_less_than_nu_over_2}, the combination of conclusion \cref{control_on_Ec_induced_by_Fc} of \cref{cor:control_on_Ec_induced_by_Fc} and conclusion \cref{control_moves_to_left_invasion_point_trav_frame} of \cref{cor:control_moves_to_left_invasion_point_trav_frame} show that, on the long run and as long as the quantity $\hat{E}_{c^*}(t)$ is large positive, this quantity decreases at an exponential rate which is at least equal to $\nu/2$, even if the possible drift to the left of the invasion point is taken into account, and provided that the invasion point $\hat{\xi}_{c^*}(t)$ is not too large positive. Let us consider the quantity $\hatXiLeft(\gamma)$, provided by \cref{cor:control_moves_to_left_invasion_point_trav_frame} for the choice \cref{def_eta} of $\gamma$, and let us recall the quantities $K$ and $K'$ involved in the conclusion \cref{control_on_Ec_induced_by_Fc} of \cref{cor:control_on_Ec_induced_by_Fc}. In order to formalize these observations in the next \cref{lem:factor_2_shrinkage_energy_at_invasion_point}, let us introduce the two parameters $\Tshrink$ and $\Elarge$ defined as
\[
\Tshrink = \frac{2}{\nu}\bigl(\ln(4K) + c^*\hatXiLeft(\gamma)\bigr)
\quad\text{and}\quad
\Elarge = 8K' e^{c^*\bigl(2\gamma \Tshrink + \hatXiLeft(\gamma)\bigr)}
\,,
\]
so that
\begin{equation}
\label{properties_Tshrink_Elarge}
K e^{-\frac{\nu}{2}\Tshrink + c^* \hatXiLeft(\gamma)} \le \frac{1}{4}
\quad\text{and}\quad
\frac{2K' e^{c^*(2\gamma \Tshrink + \hatXiLeft(\gamma))}}{\Elarge} \le \frac{1}{4}
\,.
\end{equation}
The following lemma is the analogue of the claim in the proof of \cite[Proposition~6.1]{GallayJoly_globStabDampedWaveBistable_2009}.
\begin{lemma}[factor two shrinkage of the energy at the invasion point on some time interval]
\label{lem:factor_2_shrinkage_energy_at_invasion_point}
For every nonnegative time $t$, if 
\[
\hat{E}_{c^*}(t) \ge \Elarge
\,,
\]
then there exists a time $\tShrink(t)$ in the interval $(t,t+\Tshrink]$ such that 
\begin{equation}
\label{factor_2_shrinkage_energy_at_invasion_point}
\hat{E}_{c^*}\bigl(\tShrink(t)\bigr) \le \frac{1}{2} \hat{E}_{c^*}(t)
\,.
\end{equation}
\end{lemma}
\begin{proof}
Let us distinguish two cases, see \cref{fig:illustration_proof_shrinkage}.
\begin{figure}[!htbp]
\centering
\includegraphics[width=\textwidth]{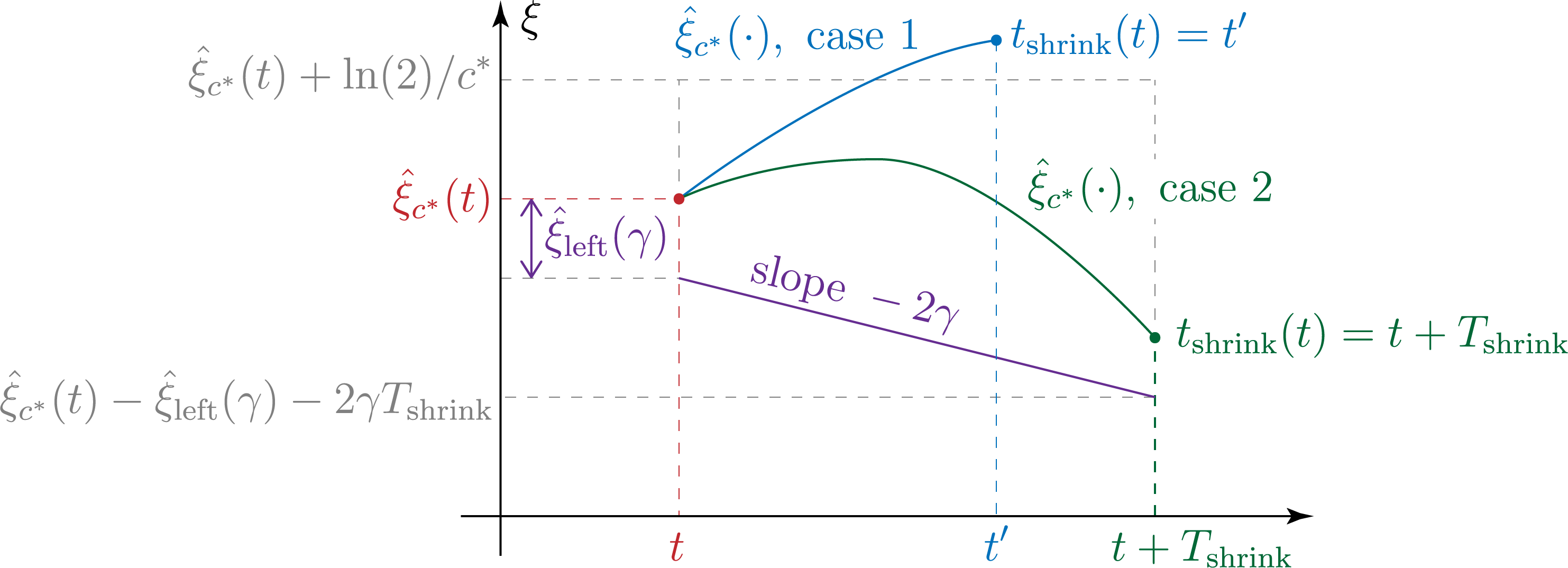}
\caption{Illustration of the proof of \cref{lem:factor_2_shrinkage_energy_at_invasion_point}.}
\label{fig:illustration_proof_shrinkage}
\end{figure}
\paragraph*{Case 1.} There exists a time $t'$ in the interval $(t,t+\Tshrink]$ such that
\begin{equation}
\label{hat_xi_of_t_plus_Trlag_minus_hat_xi_of_t_larger_than_zetaRlag}
\hat{\xi}_{c^*}(t') \ge \hat{\xi}_{c^*}(t) + \frac{\ln(2)}{c^*}
\,.
\end{equation}
In this case, let us choose $\tShrink(t) = t'$; the intended conclusion \cref{factor_2_shrinkage_energy_at_invasion_point} follows from the expression \cref{def_hat_Ec_of_t} of $\hat{E}_c(\cdot)$. 
\paragraph*{Case 2.} For every time $t''$ in the interval $(t,t+\Tshrink]$,
\[
\hat{\xi}_{c^*}(t'') \le \hat{\xi}_{c^*}(t) + \frac{\ln(2)}{c^*}
\,.
\]
In this case, let us choose $\tShrink(t) = t + \Tshrink$. It follows from the expression \cref{def_hat_Ec_of_t} of $\hat{E}_c(\cdot)$ and conclusion \cref{control_on_Ec_induced_by_Fc} of \cref{cor:control_on_Ec_induced_by_Fc} that
\[
\begin{aligned}
\hat{E}_{c^*}\bigl(\tShrink(t)\bigr) &= e^{-c^*\hat{\xi}_{c^*}\bigl(\tShrink(t)\bigr)} E_{c^*}(t+\Tshrink) \\
&\le e^{-c^*\hat{\xi}_{c^*}\bigl(\tShrink(t)\bigr)} \left( K e^{-\nu \Tshrink} E_{c^*}(t) + K' e^{c^*\bigl(\hat{\xi}_{c^*}(t) + \ln(2)/c^*\bigr)}\right) \\
&\le e^{-c^*\bigl(\hat{\xi}_{c^*}\bigl(\tShrink(t)\bigr)-\hat{\xi}_{c^*}(t)\bigr)}\left( K e^{-\nu \Tshrink} \hat{E}_{c^*}(t) + 2K' \right)
\,,
\end{aligned}
\]
so that, according to conclusion \cref{control_moves_to_left_invasion_point_trav_frame} of \cref{cor:control_moves_to_left_invasion_point_trav_frame} and the last inequality of \cref{c_star_less_than_cUpp_and_2_eta_c_star_less_than_nu_over_2}, 
\[
\begin{aligned}
\frac{\hat{E}_{c^*}\bigl(\tShrink(t)\bigr)}{\hat{E}_{c^*}(t)} &\le e^{c^*\bigl(2\gamma \Tshrink + \hatXiLeft(\gamma)\bigr)}\left( K e^{-\nu \Tshrink} + \frac{2K'}{\Elarge}\right) \\
&\le K e^{-\frac{\nu}{2}\Tshrink + c^* \hatXiLeft(\gamma)} + \frac{2K' e^{c^*(2\gamma \Tshrink + \hatXiLeft(\gamma))}}{\Elarge}
\,,
\end{aligned}
\]
and it follows from the inequalities \cref{properties_Tshrink_Elarge} satisfied by $\Tshrink$ and $\Elarge$ that the right-hand side of this last inequality is smaller than or equal to $1/2$, showing that the intended inequality \cref{factor_2_shrinkage_energy_at_invasion_point} holds.
\end{proof}
\subsubsection{Uniform upper bound on the energy at the invasion point}
The following proposition is the analogue of \cite[Proposition~6.1]{GallayJoly_globStabDampedWaveBistable_2009} and \cite[Proposition~4.1]{Luo_globStabDampedWaveEqu_2013}.
\begin{proposition}[uniform upper bound on the energy at the invasion point]
\label{prop:upper_bound_energy_at_invasion_point}
The quantity $\hat{E}_{c^*}(t)$ is bounded from above, uniformly with respect to $t$ in $[0,+\infty)$. 
\end{proposition}
\begin{proof}
Let us consider the sequence $(t_n)_{n\in\nn}$ defined as follows: $t_0=0$, and, for every nonnegative time $n$, 
\[
t_{n+1} = \left\{
\begin{aligned}
t_n + 1 \quad&\text{if}\quad \hat{E}_{c^*}(t_n) < \Elarge \,,\\
\tShrink(t_n)\quad&\text{if}\quad \hat{E}_{c^*}(t_n)\ge \Elarge \,,
\end{aligned}
\right.
\]
where $\Tshrink(\cdot)$ is the time provided by \cref{lem:factor_2_shrinkage_energy_at_invasion_point}. The sequence $(t_n)_{n\in\nn}$ is thus strictly increasing, and in the second of those cases it follows from the conclusion \cref{factor_2_shrinkage_energy_at_invasion_point} of \cref{lem:factor_2_shrinkage_energy_at_invasion_point} that $\hat{E}_{c^*}(t_{n+1})$ is less than or equal to $\hat{E}_{c^*}(t_n)/2$; as a consequence, the first case where $t_{n+1}$ equals $t_n+1$ must occur for an infinite number of nonnegative integers $n$, so that $t_n$ goes to $+\infty$ as $n$ goes to $+\infty$. 

Now, it follows from the expression \cref{def_hat_Ec_of_t} of $\hat{E}_c(\cdot)$, from the non-increase of energy with respect to time (\cref{prop:decrease_energy_trav_frame}), and from the conclusion \cref{control_moves_to_left_invasion_point_trav_frame} of \cref{cor:control_moves_to_left_invasion_point_trav_frame} that, for every nonnegative integer $n$ and for every time $t$ in the interval $[t_n,t_{n+1}]$, 
\[
\begin{aligned}
\hat{E}_{c^*}(t) &= e^{c^*\bigl(\hat{\xi}_{c^*}(t_n)-\hat{\xi}_{c^*}(t)\bigr)} \hat{E}_{c^*}(t_n) \\
&\le e^{c^*\bigl(\hat{\xi}_{c^*}(t_n)-\hat{\xi}_{c^*}(t)\bigr)} \max\bigl(\hat{E}_{c^*}(t_n),0\bigr) \\
&\le e^{c^*\bigl(2\gamma (t-t_n) + \hatXiLeft(\gamma)\bigr)} \max\bigl(\hat{E}_{c^*}(t_n),0\bigr) 
\,.
\end{aligned}
\]
It follows that, for every nonnegative integer $n$,
\[
\hat{E}_{c^*}(t_n) \le \max\Bigl(e^{c^*\bigl(2\gamma+\hatXiLeft(\gamma)\bigr)} \Elarge,\hat{E}_{c^*}(0)\Bigr)
\,,
\]
and that, for every nonnegative time $t$, 
\[
\hat{E}_{c^*}(t) \le e^{c^*\bigl(2\gamma\Tshrink+\hatXiLeft(\gamma)\bigr)} \max\Bigl(e^{c^*\bigl(2\gamma+\hatXiLeft(\gamma)\bigr)} \Elarge,\hat{E}_{c^*}(0)\Bigr)
\,,
\]
which completes the proof. 
\end{proof}
\subsubsection{Uniform bound on the \texorpdfstring{$H^1_{c^*}$}{H1c*}-norm with weight normalized at the invasion point}
\label{subsubsec:uniform_bound_H1_cStar_norm}
The following corollary is the analogue of \cite[inequality~(7.2)]{GallayJoly_globStabDampedWaveBistable_2009}. 
\begin{corollary}[uniform bound on the $H^1_{c^*}$-norm with weight normalized at the invasion point]
\label{cor:uniform_bound_H1_cStar}
The quantity
\begin{equation}
\label{uniform_bound_H1_cStar}
\int_0^{+\infty}e^{c^*y}(u^2 + u_x^2)\bigl(\hat{x}(t)+ y,t\bigr) \, dy
\end{equation}
is bounded, uniformly with respect to $t$ in $[0,+\infty)$.
\end{corollary}
\begin{proof}
For every nonnegative time $t$, it follows from the definition \cref{def_Ec_of_t} of $E_c(\cdot)$ that
\[
\hat{E}_{c^*}(t) = \int_{\rr} e^{c^*y} \left(\frac{1}{2}u_x^2 + V(u)\right)\bigl(\hat{x}(t)+y,t\bigr) \, dy
\,,
\]
so that, according to the definitions of $\deltaHess(c_0)$ and $\hat{x}(\cdot)$ (see \cref{subsec:invasion_point_defined_by_smaller_radius,fig:travelling_frame_invasion_points}),
\[
\hat{E}_{c^*}(t) +\frac{\abs{\Vmin}}{c^*} \ge \frac{1}{2}\int_0^{+\infty} e^{c^*y}\bigl(u_x^2 - \abs{\mu_0}u^2\bigr)\bigl(\hat{x}(t)+y,t\bigr) \, dy
\,.
\]
Let us consider the quantity $\alpha$ defined as
\[
\alpha = \frac{\abs{\mu_0'}+\abs{\mu_0}}{2\abs{\mu_0'}}
\,,
\]
so that both quantities
\begin{equation}
\label{1_minus_alpha_and_alpha_abs_muLow_minus_abs_hatMuLow}
1-\alpha = \frac{\abs{\mu_0'}-\abs{\mu_0}}{2\abs{\mu_0'}}
\quad\text{and}\quad
\alpha\abs{\mu_0'} - \abs{\mu_0} = \frac{\abs{\mu_0'}-\abs{\mu_0}}{2}
\end{equation}
are positive. It follows from the previous inequality that
\[
2\left(\hat{E}_{c^*}(t) + \frac{\abs{\Vmin}}{c^*}\right) \ge \int_0^{+\infty} e^{c^*y}\Bigl((1-\alpha)u_x^2 + \alpha u_x^2 - \abs{\mu_0} u^2\Bigr)\bigl(\hat{x}(t)+y,t\bigr)\, dy
\,,
\]
so that, applying the Poincaré inequality \cref{Poincare_inequality_gamma_equals_c_over_2} (with $c^*$ instead of $c$) to the term $\alpha u_x^2$ in the integrand and using the inequality $(c^*)^2/4>\abs{\mu_0}$ it follows that, denoting by $\beta$ the minimum of the two (positive) quantities \cref{1_minus_alpha_and_alpha_abs_muLow_minus_abs_hatMuLow}, 
\[
2\left(\hat{E}_{c^*}(t) + \frac{\abs{\Vmin}}{c^*}\right) \ge \beta\int_0^{+\infty} e^{c^*y}(u_x^2 + u^2)\bigl(\hat{x}(t)+y,t\bigr)\, dy
\,,
\]
and, in view of \cref{prop:upper_bound_energy_at_invasion_point}, the intended conclusion \cref{uniform_bound_H1_cStar} follows. 
\end{proof}
\subsubsection{Lipschitz continuity with respect to speed of the energy at the invasion point}
\label{subsubsec:Lipschitz_continuity}
The next \cref{cor:Lipschitz_cont_energy} is the analogue of \cite[Lemma~7.2]{GallayJoly_globStabDampedWaveBistable_2009}.
\begin{corollary}[Lipschitz continuity with respect to speed of the energy at the invasion point]
\label{cor:Lipschitz_cont_energy}
There exist a (finite, positive) quantity $\Klipsch$ such that, for all speeds $c_1$ and $c_2$ in $[c_0,c^*]$ and every time $t$ in $[0,+\infty)$, 
\begin{equation}
\label{Lipschitz_cont_energy}
\abs{\hat{E}_{c_1}(t) - \hat{E}_{c_2}(t)}\le \Klipsch \abs{c_1 - c_2} 
\,.
\end{equation}
\end{corollary}
\begin{proof}
For all speeds $c_1$ and $c_2$ in $[c_0,c^*]$ and every time $t$ in $[0,+\infty)$, 
\[
\hat{E}_{c_1}(t) - \hat{E}_{c_2}(t) = \int_{\rr} \bigl(e^{c_1 y}-e^{c_2 y}\bigr)\left(\frac{1}{2}u_x\bigl(\hat{x}(t)+y,t\bigr)^2 + V\Bigl(u\bigl(\hat{x}(t)+y,t\bigr)\Bigr)\right)\, dy 
\,.
\]
Besides, for every $y$ in $\rr$, 
\[
\abs{e^{c_1 y}-e^{c_2 y}} \le \left\{
\begin{aligned}
\abs{c_1-c_2} e^{\max(c_1,c_2)y} \le \abs{c_1-c_2} e^{c^*y} \quad\text{if}\quad y\ge 0\,,\\
\abs{c_1-c_2} e^{\min(c_1,c_2)y} \le \abs{c_1-c_2} e^{c_0 y} \quad\text{if}\quad y\le 0
\,. 
\end{aligned}
\right.
\]
It follows that, if $c_1$ and $c_2$ differ, 
\[
\begin{aligned}
\frac{\abs{\hat{E}_{c_1}(t) - \hat{E}_{c_2}(t)}}{\abs{c_1-c_2}} &\le \int_{-\infty}^0 e^{c_0 y}\abs{\frac{1}{2}u_x\bigl(\hat{x}(t)+y,t\bigr)^2 + V\Bigl(u\bigl(\hat{x}(t)+y,t\bigr)\Bigr)} \, dy \\
&\quad+ \int_0^{+\infty} e^{c^*y}\abs{\frac{1}{2}u_x\bigl(\hat{x}(t)+y,t\bigr)^2 + V\Bigl(u\bigl(\hat{x}(t)+y,t\bigr)\Bigr)} \, dy
\,.
\end{aligned}
\]
It follows from the bound \cref{bound_u_ux_setup} on the solution that the first among the two integrals of the right-hand side of this inequality is bounded (uniformly with respect to $t$ in $[0,+\infty)$), and it follows from inequality \cref{uniform_bound_H1_cStar} of \cref{cor:uniform_bound_H1_cStar} that the same is true for the second integral. Inequality \cref{Lipschitz_cont_energy} (for a large enough positive quantity $\Klipsch$) is proved. 
\end{proof}
\subsection{Relaxation}
\label{subsec:relaxation}
The aim of this \namecref{subsec:relaxation} is to prove the following proposition, which is the analogue of \cite[Proposition~7.1]{GallayJoly_globStabDampedWaveBistable_2009} and \cite[limit~(4.2)]{GallayJoly_globStabDampedWaveBistable_2009}. 
\begin{proposition}[relaxation]
\label{prop:relaxation}
For every positive quantity $T$, the following limit holds:
\begin{equation}
\label{relaxation}
\int_{t-T}^{t} D_{\widebar{c}}\bigl(\hat{\xi}_{\widebar{c}}(t),s\bigr)\, ds \to 0
\quad\text{as}\quad
t\to+\infty
\,.
\end{equation}
\end{proposition}
\begin{proof}
Let us proceed by contradiction and assume that the converse holds. In this case, there exists a positive quantity $\epsDissip$ and a sequence $(t_n)_{n\in\nn}$ of times, going to $+\infty$ such that, for every $n$ in $\nn$, $t_n-T$ is nonnegative and
\begin{equation}
\label{hyp_contradiction_finite_dissip}
\int_{t_n-T}^{t_n} D_{\widebar{c}}\bigl(\hat{\xi}_{\widebar{c}}(t_n),s\bigr)\, ds \ge \epsDissip
\,.
\end{equation}
Up to replacing the sequence $(t_n)_{n\in\nn}$ by a subsequence, let us assume that, for every $n$ in $\nn$, $t_{n+1}$ is greater than $t_n+T$. For every $n$ in $\nn$, let us consider the speed $\widebar{c}_n$ defined as
\begin{equation}
\label{def_cn}
\widebar{c}_n = \frac{\bigl(\hat{x}(t_{n+1}) - \widebar{c}T\bigr) - \hat{x}(t_n)} {(t_{n+1}-T) - t_n} 
= 
\widebar{c} + \frac{\hat{\xi}_{\widebar{c}}(t_{n+1}) - \hat{\xi}_{\widebar{c}}(t_n)} {(t_{n+1}-T) - t_n}
\,,
\end{equation}
see \cref{fig:broken_line}. 
\begin{figure}[!htbp]
\centering
\includegraphics[width=\textwidth]{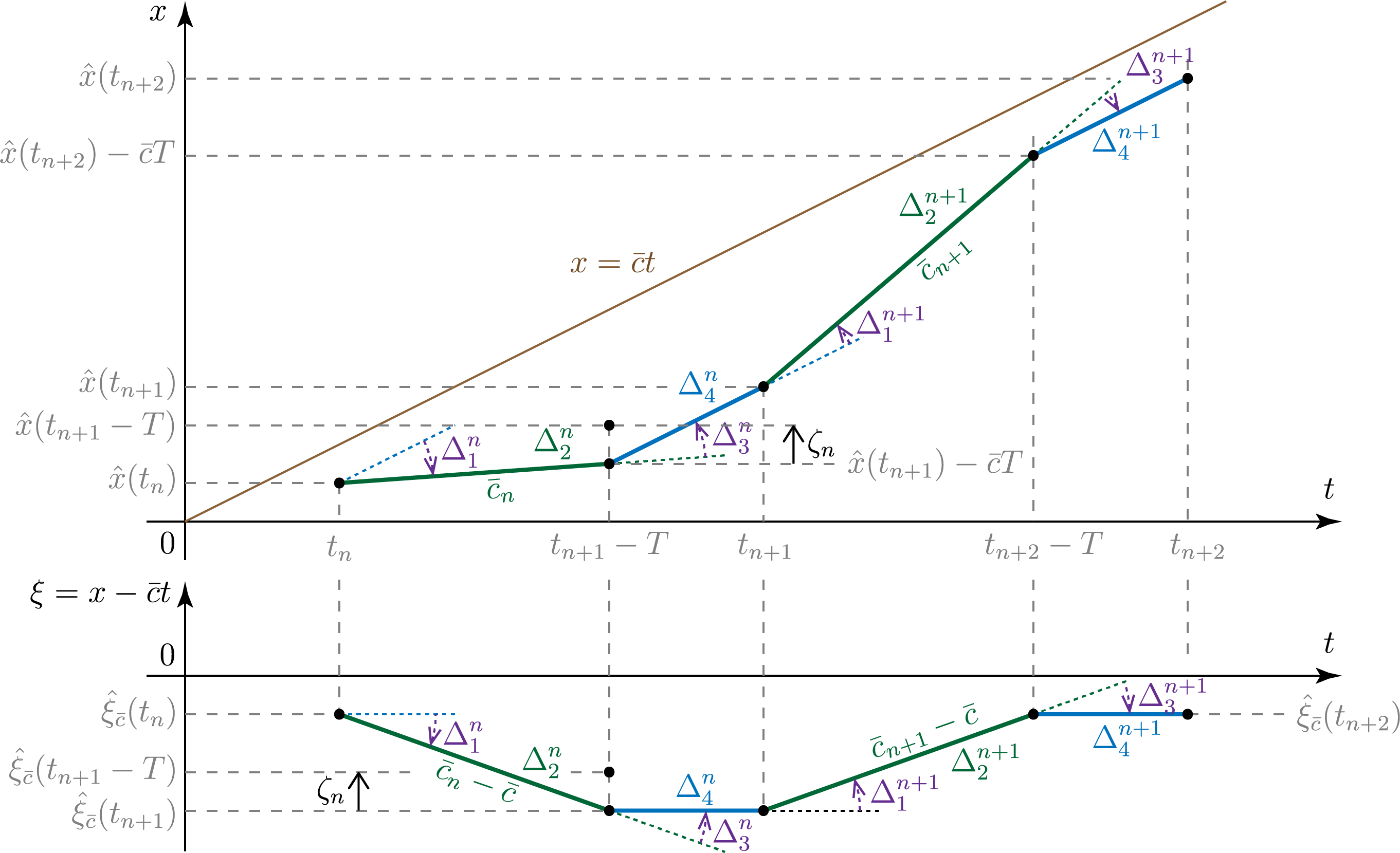}
\caption{Broken line in the coordinates $(t,x)$ and $(t,\xi = x-\bar{c}t)$ used to decompose the differences $\Delta_n$ between $\hat{E}_{\bar{c}}(t_n)$ and $\hat{E}_{\bar{c}}(t_{n+1})$ (and $\Delta_{n+1}$ between $\hat{E}_{\bar{c}}(t_{n+1})$ and $\hat{E}_{\bar{c}}(t_{n+2})$). Along the green lines, the contribution (dissipation) is nonnegative, and it is significant (not smaller than the positive quantity $\epsDissip$) along the blue lines (which have slope $\bar{c}$ in the coordinates $(t,x)$ and slope $0$ in the coordinates $(t,\xi)$). The contributions due to the changes of speed at the vertices (in purple) are small.}
\label{fig:broken_line}
\end{figure}
Again up to replacing the sequence $(t_n)_{n\in\nn}$ by a subsequence, it may be assumed, according to \cref{prop:invasion_speed}, that 
\[
\widebar{c}_n\xrightarrow[n\to+\infty]{}\widebar{c}
\quad\text{and, for every $n$ in $\nn$,}\quad
c_0\le \widebar{c}_n \le c^*
\,.
\]
Following the notation of \cite[Proposition~7.1]{GallayJoly_globStabDampedWaveBistable_2009}, let us introduce, for every $n$ in $\nn$, the quantities
\begin{align}
\label{def_Delta_n}
\Delta_n &= \hat{E}_{\widebar{c}}(t_n) - \hat{E}_{\widebar{c}}(t_{n+1}) \,, \\
\nonumber
\text{and}\quad 
\Delta^1_n &= \hat{E}_{\widebar{c}}(t_n) - \hat{E}_{\widebar{c}_n}(t_n)  \,, \\
\nonumber
\text{and}\quad 
\Delta^2_n &= \hat{E}_{\widebar{c}_n}(t_n) -  E_{\widebar{c}_n}\bigl(\hat{\xi}_{\widebar{c}_n}(t_n),t_{n+1}-T\bigr) \\
\nonumber
&= E_{\widebar{c}_n}\bigl(\hat{\xi}_{\widebar{c}_n}(t_n),t_n\bigr) - E_{\widebar{c}_n}\bigl(\hat{\xi}_{\widebar{c}_n}(t_n),t_{n+1}-T\bigr) \,, \\
\nonumber
\text{and}\quad 
\Delta^3_n &= E_{\widebar{c}_n}\bigl(\hat{\xi}_{\widebar{c}_n}(t_n),t_{n+1}-T\bigr) - E_{\widebar{c}}\bigl(\hat{\xi}_{\widebar{c}}(t_{n+1}),t_{n+1}-T\bigr) \,, \\
\nonumber
\text{and}\quad 
\Delta^4_n &= E_{\widebar{c}}\bigl(\hat{\xi}_{\widebar{c}}(t_{n+1}),t_{n+1}-T\bigr) - \hat{E}_{\widebar{c}}(t_{n+1}) \\
\nonumber
&= E_{\widebar{c}}\bigl(\hat{\xi}_{\widebar{c}}(t_{n+1}),t_{n+1}-T\bigr) - E_{\widebar{c}}\bigl(\hat{\xi}_{\widebar{c}}(t_{n+1}),t_{n+1}\bigr) \,, \\
\label{Delta_n_equals_Delta_1_n_plus_plus_Delta_4_n}
\text{so that}\quad 
\Delta_n &= \Delta^1_n + \Delta^2_n + \Delta^3_n + \Delta^4_n
\,.
\end{align}
According to inequality \cref{Lipschitz_cont_energy} of \cref{cor:Lipschitz_cont_energy}, 
\begin{equation}
\label{upper_bound_abs_Delta1n}
\abs{\Delta^1_n}\le \Klipsch\abs{\widebar{c}-\widebar{c}_n}
\,.
\end{equation}
According to equality \cref{decrease_energy_trav_frame} of \cref{prop:decrease_energy_trav_frame}, 
\begin{equation}
\label{Delta2n_nonnegative}
\Delta^2_n = \int_{t_n}^{t_{n+1}-T} D_{\widebar{c}_n}\bigl(\hat{\xi}_{\widebar{c}_n}(t_n),s\bigr)\, ds \ge 0 
\,,
\end{equation}
and according to assumption \cref{hyp_contradiction_finite_dissip}, 
\begin{equation}
\label{Delta4n_larger_than_deltaDissip}
\Delta^4_n = \int_{t_{n+1}-T}^{t_{n+1}} D_{\widebar{c}}\bigl(\hat{\xi}_{\widebar{c}}(t_{n+1}),s\bigr) \, ds \ge \epsDissip
\,.
\end{equation}
Let us consider the quantity $\zeta_n$ defined as
\begin{equation}
\label{def_alpha_n_cn}
\zeta_n = \hat{\xi}_{\widebar{c}_n}(t_{n+1}-T) - \hat{\xi}_{\widebar{c}_n}(t_n)
\,.
\end{equation}
According to the first inequality of \cref{def_cn}, 
\begin{align}
\nonumber
\zeta_n &= \hat{x}(t_{n+1}-T) - \hat{x}(t_n) - \widebar{c}_n (t_{n+1}-T-t_n) \\
\nonumber
&= \hat{x}(t_{n+1}-T) - \hat{x}(t_{n+1}) + \widebar{c} T \\
\label{def_alpha_n_bar_c}
&= \hat{\xi}_{\widebar{c}}(t_{n+1}-T) - \hat{\xi}_{\widebar{c}}(t_{n+1})
\,,
\end{align}
and according to inequality \cref{control_moves_to_left_invasion_point_trav_frame} of \cref{cor:control_moves_to_left_invasion_point_trav_frame},
\begin{equation}
\label{upper_bound_zeta_n}
\zeta_n \le 2\gamma T + \hatXiLeft(\gamma)
\,.
\end{equation}
According to the two expressions \cref{def_alpha_n_cn,def_alpha_n_bar_c} of $\zeta_n$,
\[
\begin{aligned}
\Delta^3_n &= e^{\widebar{c}_n\zeta_n} \hat{E}_{\widebar{c}_n}(t_{n+1}-T) - e^{\widebar{c}\zeta_n} \hat{E}_{\widebar{c}}(t_{n+1}-T) \\
&= \left(e^{\widebar{c}_n\zeta_n} - e^{\widebar{c}\zeta_n}\right) \hat{E}_{\widebar{c}_n}(t_{n+1}-T) + e^{\widebar{c}\zeta_n} \bigl(\hat{E}_{\widebar{c}_n}(t_{n+1}-T) - \hat{E}_{\widebar{c}}(t_{n+1}-T) \bigr) 
\,,
\end{aligned}
\]
so that, according to inequality \cref{Lipschitz_cont_energy} of \cref{cor:Lipschitz_cont_energy} and inequality \cref{upper_bound_zeta_n},
\begin{equation}
\label{upper_bound_Delta3n}
\abs{\Delta^3_n} \le \abs{e^{\widebar{c}_n\zeta_n} - e^{\widebar{c}\zeta_n}} \abs{\hat{E}_{\widebar{c}_n}(t_{n+1}-T)} \\ + e^{\widebar{c}\bigl(2\gamma T + \hatXiLeft(\gamma)\bigr)} \Klipsch\abs{\widebar{c}_n - \widebar{c}}
\,.
\end{equation}
Observe that, for every positive quantities $z_{\max}$ and $c$
\[
\sup_{z\in(-\infty,z_{\max}]} z e^{cz} = \max\left(\frac{1}{ce},z_{\max} e^{c z_{\max}}\right)
\,.
\]
Thus it follows from inequality \cref{upper_bound_zeta_n} that
\begin{align}
\nonumber
\abs{e^{\widebar{c}_n\zeta_n} - e^{\widebar{c}\zeta_n}} &\le \abs{\widebar{c}_n - \widebar{c}} \abs{\zeta_n} \max(e^{\widebar{c}_n\zeta_n},e^{\widebar{c}\zeta_n}) \\
&\le \abs{\widebar{c}_n - \widebar{c}} \max\left(\frac{1}{c_0 e}, \bigl(2\gamma T + \hatXiLeft(\gamma)\bigr)e^{c^*\bigl(2\gamma T + \hatXiLeft(\gamma)\bigr)}\right)
\,.
\label{upper_bound_difference_of_exp}
\end{align}
Since $\widebar{c}_n$ goes to $\widebar{c}$ as $n$ goes to $+\infty$, it follows from inequalities \cref{upper_bound_abs_Delta1n,Delta2n_nonnegative,Delta4n_larger_than_deltaDissip,upper_bound_Delta3n,upper_bound_difference_of_exp} and from the expression \cref{Delta_n_equals_Delta_1_n_plus_plus_Delta_4_n} of $\Delta_n$ that
\begin{equation}
\label{liminf_Delta_n}
\liminf_{n\to+\infty}\Delta_n \ge \epsDissip
\,.
\end{equation}
According to the definition \cref{def_Delta_n} of $\Delta_n$, for every positive integer $n$, 
\[
\hat{E}_{c_0}(t_0) - \hat{E}_{\widebar{c}_n}(t_n) = \sum_{k=0}^{n-1}\Delta_k
\,,
\]
so that, according to inequality \cref{liminf_Delta_n}, 
\begin{equation}
\label{lim_E_cn_of_tn}
\hat{E}_{\widebar{c}_n}(t_n) \to -\infty 
\quad\text{as}\quad 
n\to+\infty
\,.
\end{equation}
On the other hand, for every speed $c$ in $[c_0,\cDecay]$ and for every nonnegative time $t$, it follows from inequality \cref{lower_bound_Ec} that
\[
\begin{aligned}
\hat{E}_c(t) &\ge -\frac{\abs{V_{\min}}}{c}e^{-c\bigl(\hat{\xi}_c(t)-\widebar{\xi}_c(t)\bigr)} \\
&\ge -\frac{\abs{V_{\min}}}{c}
\,,
\end{aligned}
\]
a contradiction with the limit \cref{lim_E_cn_of_tn}. \Cref{prop:relaxation} is proved. 
\end{proof}
\begin{remark}
In the proof above, inequality \cref{upper_bound_zeta_n} is intimately related to the choice of the ``dissipation interval'' (namely $[t_{n+1}-T,t_{n+1}]$, rather than, say $[t_n,t_n+T]$), and is the reason for the choice of the integration interval $[t-T,t]$ (rather than, say, $[t,t+T]$) in inequality \cref{relaxation} of \cref{prop:relaxation}. 
\end{remark}
\subsection{Convergence and proof of \texorpdfstring{\cref{thm:main}}{Theorem \ref{thm:main}}}
\label{subsec:convergence}
Let $\deltaLocMan(\widebar{c})$ denote a positive quantity, small enough so that the ``local steep stable manifold'' \cref{prop:local_steep_stable_manifold} holds for $e$ equal to $0_{\rr^d}$ and $c$ equal to $\widebar{c}$ and $\delta$ equal to $\deltaLocMan(\widebar{c})$. Let us assume in addition that $\deltaLocMan(\widebar{c})$ is smaller than or equal to the quantity $\deltaHess(c_0)$ defined in \cref{subsec:invasion_point_defined_by_smaller_radius}. For every nonnegative time $t$, let us consider the set $\SigmaFar{\deltaLocMan(\widebar{c})}(t)$ (see definition \cref{def_SigmaFar}). For the same reasons as for the set $\SigmaFar{\deltaHess(\widebar{c})}(t)$ (see the beginning of \cref{subsec:invasion_point_defined_by_smaller_radius}), this set $\SigmaFar{\deltaLocMan(\widebar{c})}(t)$ is altogether nonempty and bounded from above; let $\tilde{x}(t)$ denote its supremum and let us write (for a positive quantity $c$)
\[
\tilde{\xi}_c(t) = \tilde{x}(t) - ct
\,,
\]
see \cref{fig:travelling_frame_invasion_points}. According to these definitions, 
\begin{equation}
\label{hat_x_less_than_tilde_x}
\widebar{x}(t)\le\hat{x}(t)\le\tilde{x}(t)<+\infty 
\quad\text{and}\quad
\widebar{\xi}_c(t)\le\hat{\xi}_c(t)\le\tilde{\xi}_c(t)<+\infty
\,, 
\end{equation}
and
\begin{equation}
\label{abs_u_of_tilde_x_of_t_t_equals_tilde_delta}
\abs{u\bigl(\tilde{x}(t),t\bigr)} = \deltaLocMan(\widebar{c})
\,.
\end{equation}
\begin{remark}
According to \cref{prop:extension_local_steep_stable_manifold}, it turns out that the quantity $\deltaLocMan(\widebar{c})$ could actually be chosen equal to the quantity $\deltaHess(c_0)$ introduced in \cref{subsec:invasion_point_defined_by_smaller_radius}; with such a choice, the invasion points $\tilde{x}(t)$ and $\tilde{\xi}_c(t)$ would not differ from $\hat{x}(t)$ and from $\hat{\xi}_c(t)$, respectively. However, this would not significantly simplify the remaining part of the proof; for that reason, this remaining part will be presented for a quantity $\deltaLocMan(\widebar{c})$ not necessarily equal to $\deltaHess(c_0)$, in other words without calling upon the conclusions of \cref{prop:extension_local_steep_stable_manifold}. 
\end{remark}
Let us recall the notation $\uPushedFront{0_{\rr^d}}{\deltaLocMan(\widebar{c})}{\widebar{c}}$ introduced in \cref{def_uPushedFront}, and, for every $w$ in $\partial B_{\rr^d}\bigl(0_{\rr^d},\deltaLocMan(\widebar{c})\bigr)$, the notation $\phi_{\widebar{c},w}$ introduced in \cref{initial_condition_phi_c_u}.
\begin{lemma}[invasion through profiled of pushed fronts, 1]
\label{lem:invasion_through_profiles_pushed_fronts}
The following conclusions hold:
\begin{enumerate}
\item $\limsup_{t\to+\infty} u\bigl(\tilde{x}(t),t\bigr)\cdot u_x\bigl(\tilde{x}(t),t\bigr)<0$. 
\label{item:lem_invasion_through_profiles_pushed_fronts_u_dot_ux_negative}
\item the set $\uPushedFront{0_{\rr^d}}{\deltaLocMan(\widebar{c})}{\widebar{c}}$ is nonempty;
\item the following limits hold as $t$ goes to $+\infty$:
\begin{enumerate}
\item $\dist\Bigl(u\bigl(\tilde{x}(t),t\bigr),\uPushedFront{0_{\rr^d}}{\deltaLocMan(\widebar{c})}{\widebar{c}}\Bigr)\to 0$;
\item $u_t\bigl(\tilde{x}(t),t\bigr)+\widebar{c}u_x\bigl(\tilde{x}(t),t\bigr)\to0$;
\label{item:lem_invasion_through_profiles_pushed_fronts_u_t_plus_bar_c_u_x_goes_to_0}
\item for every positive quantity $L$, 
\[
\sup_{y\in\bigl[-L,+\infty\bigr)}\abs{u\bigl(\tilde{x}(t)+y,t\bigr)-\phi_{\widebar{c},u(\tilde{x}(t),t)}(y)}\to 0
\,.
\]
\end{enumerate}
\end{enumerate}
\end{lemma}
The proof calls upon some properties of the profiles of pushed travelling waves invading a critical point, namely \cref{lem:asymptotics_at_the_two_ends_of_space,lem:uniform_convergence_towards_e_wrt_u} stated in the next \cref{sec:properties_profiles_pushed_trav_waves}.
\begin{proof}
Let us proceed by contradiction and assume that at least one of the conclusions of this lemma does not hold. Then, there exists a sequence $(t_n)_{n\in\nn}$ of times going to $+\infty$ such that one of the following properties hold:
\begin{enumerate}[label=\Alph*]
\item the quantity $u_t\bigl(\tilde{x}(t_n),t_n\bigr)+\widebar{c}u_x\bigl(\tilde{x}(t_n),t_n\bigr)$ does not go to $0$ as $n$ goes to $+\infty$;
\label{item:proof_lem_invasion_through_profiles_pushed_fronts_ut_plus_bar_c_ux_goes_to_0_at_invasion point}
\item or, either the set $\uPushedFront{0_{\rr^d}}{\deltaLocMan(\widebar{c})}{\widebar{c}}$ is empty or, if it is nonempty, the distance $\dist\Bigl(u\bigl(\tilde{x}(t_n),t_n\bigr),\uPushedFront{0_{\rr^d}}{\deltaLocMan(\widebar{c})}{\widebar{c}}\Bigr)$ does not go to $0$ as $n$ goes to $+\infty$;
\label{item:proof_lem_invasion_through_profiles_pushed_fronts_dist_to_uPushedFront_goes_to_0}
\item or the quantity $\limsup_{n\to+\infty} u\bigl(\tilde{x}(t_n),t_n\bigr)\cdot u_x\bigl(\tilde{x}(t_n),t_n\bigr)$ is nonnegative;
\label{item:proof_lem_invasion_through_profiles_pushed_fronts_scalar_product_ux_ut_negative}
\item or there exists a positive quantity $L_0$ such that the quantity
\[
\sup_{y\in\bigl[-L_0,+\infty\bigr)}\abs{u\bigl(\tilde{x}(t_n)+y,t_n\bigr)-\phi_{\widebar{c},u(\tilde{x}(t_n),t_n)}(y)}
\]
does not go to $0$ as $n$ goes to $+\infty$. 
\label{item:proof_lem_invasion_through_profiles_pushed_fronts_approach_profile}
\end{enumerate}
By compactness (\cref{lem:compactness}), up to replacing the sequence $(t_n)_{n\in\nn}$ by a subsequence, there exists an entire solution $u_\infty$ of system \cref{parabolic_system} such that, with the notation of \cref{subsec:compactness}, 
\begin{equation}
\label{converge_for_subsequence_proof_lemma_invasion_through_profiles_pushed_fronts}
D^{2,1}u\bigl(\tilde{x}(t_n)+\cdot,t_n+\cdot\bigr) \to D^{2,1} u_\infty
\quad\text{as}\quad
n\to+\infty
\,,
\end{equation}
uniformly on compact subsets of $\rr^2$. Let $T$ denote a positive quantity, and let us consider the quantity $\iii_n$ defined as
\[
\iii_n = \int_{t_n-T}^{t_n} D_{\widebar{c}}\bigl(\hat{\xi}_{\widebar{c}}(t_n),t\bigr) \, dt
\,.
\]
According to \cref{prop:relaxation}, this quantity $\iii_n$ goes to $0$ as $n$ goes to $+\infty$. Observe that
\[
\begin{aligned}
\iii_n &= \int_{t_n-T}^{t_n} e^{-\widebar{c}\hat{\xi}_{\widebar{c}}(t_n)} D_{\widebar{c}}(t) \, dt \\
&= \int_{t_n-T}^{t_n} e^{-\widebar{c}\hat{\xi}_{\widebar{c}}(t_n)} \left(\int_{\rr} e^{\widebar{c}(x-\widebar{c}t)} (u_t + \widebar{c} u_x)^2(x,t) \, dx \right) \, dt
\,,
\end{aligned}
\]
so that, substituting $t$ with $t_n+s$ and $x$ with $\tilde{x}(t_n)+y$, 
\[
\iii_n = \int_{-T}^0 e^{\widebar{c}\bigl(\tilde{x}(t_n)-\widebar{c}t_n-\hat{\xi}_{\widebar{c}}(t_n)-\widebar{c}s\bigr)}\left(\int_{\rr} e^{\widebar{c}y} (u_t + \widebar{c}u_x)^2\bigl(\tilde{x}(t_n)+y,t_n+s\bigr) \, dy \right)\, ds 
\,.
\]
According to inequality \cref{hat_x_less_than_tilde_x}, 
\[
\tilde{x}(t_n)-\widebar{c}t_n-\hat{\xi}_{\widebar{c}}(t_n) = \tilde{\xi}_{\widebar{c}}(t_n)-\hat{\xi}_{\widebar{c}}(t_n) \ge 0
\,,
\]
and the term $-\widebar{c}^2 s$ in the argument of the exponential factor is nonnegative for $s$ in $[-T,0]$; it follows that
\[
\iii_n \ge \int_{-T}^0 \left(\int_{\rr} e^{\widebar{c}y} (u_t + \widebar{c}u_x)^2\bigl(\tilde{x}(t_n)+y,t_n+s\bigr) \, dy \right)\, ds
\,.
\]
Let $L$ denote a positive quantity, and let us consider the integrals
\[
\begin{aligned}
\jjj_n &= \int_{-T}^0 \left(\int_{-L}^L (u_t + \widebar{c}u_x)^2\bigl(\tilde{x}(t_n)+y,t_n+s\bigr) \, dy \right)\, ds \,, \\
\text{and}\quad
\jjj_\infty &= \int_{-T}^0 \left(\int_{-L}^L \bigl(\partial_t u_\infty(y,s) + \widebar{c}\partial_x u_\infty(y,s)\bigr)^2 \, dy \right)\, ds
\,.
\end{aligned}
\]
It follows from the previous inequality that
\[
\iii_n\ge e^{-\widebar{c}L} \jjj_n 
\,,
\]
and since $\iii_n$ goes to $0$ as $n$ goes to $+\infty$, the same must therefore be true for the nonnegative quantity $\jjj_n$. On the other hand, it follows from the convergence \cref{converge_for_subsequence_proof_lemma_invasion_through_profiles_pushed_fronts} that $\jjj_n$ goes to $\jjj_\infty$ as $n$ goes to $+\infty$. It follows that $\jjj_\infty$ must be equal to $0$, and since the positive quantity $L$ was any, it follows that the function $\partial_t u_\infty + \widebar{c}\partial_x u_\infty$ is identically equal to $0_{\rr^d}$ on $\rr\times[0,+\infty)$, thus in particular on $\rr\times\{0\}$. In view of the convergence \cref{converge_for_subsequence_proof_lemma_invasion_through_profiles_pushed_fronts}, it follows that the property \cref{item:proof_lem_invasion_through_profiles_pushed_fronts_ut_plus_bar_c_ux_goes_to_0_at_invasion point} above cannot hold. 

Let us consider the function $\phi_\infty:\rr\to\rr^d$ defined as
\[
\phi_\infty(\xi) = u_\infty(\xi,0)\,, \quad\text{for all}\quad \xi \text{ in }\rr
\,.
\]
Then, since $u_\infty$ is a solution of the parabolic system \cref{parabolic_system}, it follows that $\phi_\infty$ is a solution of the differential systems \cref{syst_trav_front_order_2,syst_trav_front_order_1} (for $c$ equal to $\widebar{c}$) governing the profiles of waves travelling at the speed $\widebar{c}$. In addition, since according to inequalities \cref{hat_x_less_than_tilde_x} $\tilde{x}(t)$ is greater than or equal to $\hat{x}(t)$, it follows from conclusion \cref{uniform_bound_H1_cStar} of \cref{cor:uniform_bound_H1_cStar} that the quantity 
\begin{equation}
\label{H1cstar_norm_of_u_centred_at_tidle_x_of_tn}
\int_0^{+\infty} e^{c^*y} (u^2 + u_x^2)\bigl(\tilde{x}(t_n)+y,t_n\bigr) \, dy 
\end{equation}
is bounded, uniformly with respect to $n$; thus, according to Fatou Lemma, it follows from the convergence \cref{converge_for_subsequence_proof_lemma_invasion_through_profiles_pushed_fronts} that $\phi_\infty$ must belong to the space $H^1_{c^*}(\rr,\rr^d)$. Thus, according to conclusion \cref{item:lem_Poincare_inequality_convergence_at_both_ends} of \cref{lem:Poincare_inequality}, the following limit holds:
\begin{equation}
\phi_\infty(\xi) = o\bigl(e^{-\frac{1}{2}c^*\xi}\bigr)
\quad\text{as}\quad
\xi\to+\infty
\,.
\end{equation}
This shows that $\phi_\infty$ is the profile of a \emph{pushed} travelling wave invading $0_{\rr^d}$ (\cref{def:pushed_travelling_wave_front}). And, since according to equality \cref{abs_u_of_tilde_x_of_t_t_equals_tilde_delta} $\abs{\phi_\infty(0)}$ is equal to $\deltaLocMan(\widebar{c})$, it follows that 
\[
\phi_\infty = \phi_{\widebar{c},\phi_\infty(0)} = \phi_{\widebar{c},u_\infty(0,0)}
\quad\text{(notation \cref{initial_condition_phi_c_u}).}
\]
Finally, since according to the bound \cref{bound_u_ux_setup} $\phi_\infty(\xi)$ is bounded, uniformly with respect to $\xi$, it follows from conclusion \cref{item:lem_asymptotics_at_the_two_ends_of_space_left_end} of \cref{lem:asymptotics_at_the_two_ends_of_space} that $\phi_\infty$ must be the profile of a pushed \emph{front} invading $0_{\rr^d}$ at the speed $\widebar{c}$. And since according to equality \cref{abs_u_of_tilde_x_of_t_t_equals_tilde_delta} the quantity $\abs{\phi_\infty(0)}$ is equal to $\deltaLocMan(\widebar{c})$, the vector $\phi_\infty(0)$ must belong to the set $\uPushedFront{0_{\rr^d}}{\deltaLocMan(\widebar{c})}{\widebar{c}}$, which is therefore nonempty. In view of the convergence \cref{converge_for_subsequence_proof_lemma_invasion_through_profiles_pushed_fronts}, this shows that the property \cref{item:proof_lem_invasion_through_profiles_pushed_fronts_dist_to_uPushedFront_goes_to_0} above cannot hold. In addition, since according to conclusion \cref{item:lem_asymptotics_at_the_two_ends_of_space_negative_scalar_product} of \cref{lem:asymptotics_at_the_two_ends_of_space} (applied with $\delta$ equal to $\deltaHess(c_0)$ and $c$ equal to $\widebar{c}$) the scalar product $\phi_\infty(0)\cdot\phi_\infty'(0)$ is negative, it follows from the convergence \cref{converge_for_subsequence_proof_lemma_invasion_through_profiles_pushed_fronts} that the property \cref{item:proof_lem_invasion_through_profiles_pushed_fronts_scalar_product_ux_ut_negative} above cannot hold either. 

It remains to derive a contradiction from property \cref{item:proof_lem_invasion_through_profiles_pushed_fronts_approach_profile}. Observe that, since $u\bigl(\tilde{x}(t_n),t_n\bigr)$ goes to $u_\infty(0,0)$ as $n$ goes to $+\infty$, according to the continuity of the solutions of the differential systems \cref{syst_trav_front_order_2,syst_trav_front_order_1} with respect to initial conditions, 
\[
\phi_{\widebar{c},u(\tilde{x}(t_n),t_n)}(\cdot)\to \phi_{\widebar{c},u_\infty(0,0)} (\cdot) = \phi_\infty(\cdot) = u_\infty(\cdot,0) 
\,,
\]
uniformly on every compact subset of $\rr$. Thus, it follows from the convergence \cref{converge_for_subsequence_proof_lemma_invasion_through_profiles_pushed_fronts} that 
\[
u\bigl(\tilde{x}(t_n)+\cdot,t_n\bigr) - \phi_{\widebar{c},u(\tilde{x}(t_n),t_n)}(\cdot) \to 0_{\rr^d}
\,,
\]
uniformly on every compact subset of $\rr$. Therefore, it follows from property \cref{item:proof_lem_invasion_through_profiles_pushed_fronts_approach_profile} that there must exist a sequence $(y_n)_{n\in\nn}$, going to $+\infty$ as $n$ goes to $+\infty$, such that 
\begin{equation}
\label{liminf_u_minus_phi_at_tilde_x_of_tn_plus_yn}
\liminf_{n\to+\infty} \abs{u\bigl(\tilde{x}(t_n)+y_n,t_n\bigr) - \phi_{\widebar{c},u(\tilde{x}(t_n),t_n)}(y_n)} >0
\,.
\end{equation}
According to the uniform convergence stated in \cref{lem:uniform_convergence_towards_e_wrt_u} (for the same parameters and $\delta$ and $c$ as the ones chosen above to apply \cref{lem:asymptotics_at_the_two_ends_of_space}), 
\[
\phi_{\widebar{c},u(\tilde{x}(t_n),t_n)}(y_n)\to 0_{\rr^d}
\quad\text{as}\quad
n\to+\infty
\,;
\]
thus, it follows from inequality \cref{liminf_u_minus_phi_at_tilde_x_of_tn_plus_yn} that
\[
\liminf_{n\to+\infty} \abs{u\bigl(\tilde{x}(t_n)+y_n,t_n\bigr)} >0
\,,
\]
a contradiction with the uniform bound on the quantity \cref{H1cstar_norm_of_u_centred_at_tidle_x_of_tn}. \Cref{lem:invasion_through_profiles_pushed_fronts} is proved. 
\end{proof}
\begin{proof}[End of the proof of \cref{thm:main}]
Let us consider the function $f:\rr\times[0,+\infty)\to\rr$ defined as: $f(x,t) = \frac{1}{2} u(x,t)^2$. Then, in view of the property \cref{abs_u_of_tilde_x_of_t_t_equals_tilde_delta}, for every $(x,t)$ in $\rr\times[0,+\infty)$,
\[
\begin{aligned}
f\bigl(\tilde{x}(t),t\bigr) = \frac{1}{2} \deltaLocMan(\widebar{c})^2
\quad\text{and}\quad
\partial_x f(x,t) &= u_x(x,t)\cdot u(x,t) \,, \\
\text{and}\quad
\partial_t f(x,t) &= u_t(x,t)\cdot u(x,t)
\,.
\end{aligned}
\]
It thus follows from conclusion \cref{item:lem_invasion_through_profiles_pushed_fronts_u_dot_ux_negative} of \cref{lem:invasion_through_profiles_pushed_fronts} and from the Implicit Function Theorem that, for $t$ large enough positive, the function $t\mapsto \tilde{x}(t)$ is of class $C^1$ and satisfies: 
\[
\begin{aligned}
\tilde{x}'(t) &= - \frac{u_t\bigl(\tilde{x}(t),t\bigr)\cdot u\bigl(\tilde{x}(t),t\bigr)}{u_x\bigl(\tilde{x}(t),t\bigr)\cdot u\bigl(\tilde{x}(t),t\bigr)} \\
&= \widebar{c} - \frac{\Bigl(u_t\bigl(\tilde{x}(t),t\bigr) + \widebar{c}u_x\bigl(\tilde{x}(t),t\bigr)\Bigr)\cdot u\bigl(\tilde{x}(t),t\bigr)}{u_x\bigl(\tilde{x}(t),t\bigr)\cdot u\bigl(\tilde{x}(t),t\bigr)}
\,,
\end{aligned}
\]
and it follows from this expression, from conclusions \cref{item:lem_invasion_through_profiles_pushed_fronts_u_dot_ux_negative,item:lem_invasion_through_profiles_pushed_fronts_u_t_plus_bar_c_u_x_goes_to_0} of \cref{lem:invasion_through_profiles_pushed_fronts}, and from the bound \cref{bound_u_ux_setup} on $\abs{u(x,t)}$ that
\[
\tilde{x}'(t)\to\widebar{c}
\quad\text{as}\quad
t\to+\infty
\,.
\]
In view of the property \cref{abs_u_of_tilde_x_of_t_t_equals_tilde_delta} and of the conclusions of \cref{lem:invasion_through_profiles_pushed_fronts}, all the conclusions of \cref{thm:main} are proved. 
\end{proof}
\subsection{Proof of \texorpdfstring{\cref{thm:characterization_existence_pushed_front}}{Theorem \ref{thm:characterization_existence_pushed_front}}}
\label{subsec:proof_cor_main}
As in the previous \namecrefs{subsec:proof_cor_main}, let us assume that the critical point $e$ is equal to $0_{\rr^d}$. 
\begin{proof}[Proof of the equivalence between conditions \cref{item:thm_characterization_existence_pushed_front_cLinMax_le_cNonLinMax,item:thm_characterization_existence_pushed_front_existence} of \cref{thm:characterization_existence_pushed_front}]
Let us assume that condition \cref{item:thm_characterization_existence_pushed_front_cLinMax_le_cNonLinMax} of \cref{thm:characterization_existence_pushed_front} holds. Since $e$ is assumed to be equal to $0_{\rr^d}$, this means that there exists a quantity $c_0$, greater than $\cLinMax$, and a function $w$ in $H^1_{c_0}(\rr,\rr^d)$ such that the energy $\eee_{c_0}[w]$ is negative. Let $\chi:\rr\to\rr$ denote a smooth cutoff function satisfying the conditions \cref{cut_off}, let $\xLarge$ denote a (large) positive quantity to be chosen below, and let us consider the function $\tilde{w}$ defined as in \cref{expression_cutoffed_initial_condition}:
\[
\tilde{w}(x) = \chi(x-\xLarge) w(x)
\,.
\]
Let us consider the solution $(x,t)\mapsto u(x,t)$ of the parabolic system \cref{parabolic_system} for the initial condition $u(\cdot,0) = \tilde{w}(\cdot)$. According to the definition of $\tilde{w}$, the quantity $\cDecay[u]$ (defined in \cref{decay_variational_invasion_speed}) is equal to $+\infty$. In addition, since $w$ is in $H^1_{c_0}(\rr,\rr^d)$, the quantity $\eee_{c_0}[\tilde{w}]$ goes to $\eee_{c_0}[w]$ as $\xLarge$ goes to $+\infty$; thus, if $\xLarge$ is large enough positive, the quantity $\eee_{c_0}[\tilde{w}]$ is (also) negative, so that, in this case the variational speed $\cVar[u]$ (also defined in \cref{decay_variational_invasion_speed}) is greater than $c_0$. It follows that
\[
\cLinMax < \cVar[u] < \cDecay[u]
\,,
\]
or in other words that the condition \cref{thm_main_cLinMax_smaller_than_cVar_smaller_than_cDecay} of \cref{thm:main} holds. According to the conclusion of this theorem, the solution $u$ invades the critical point $0_{\rr^d}$ through the profiles of pushed fronts, which ensures the existence of (at least) one pushed front invading $0_{\rr^d}$ at the speed $\cVar[u]$ (see \cref{def:invasion_through_profiles_pushed_fronts}); in other words, conclusion \cref{item:thm_characterization_existence_pushed_front_existence} of \cref{thm:characterization_existence_pushed_front} holds. 

Conversely, if conclusion \cref{item:thm_characterization_existence_pushed_front_existence} of \cref{thm:characterization_existence_pushed_front} holds (with $e$ equal to $0_{\rr^d}$), then there exists a pushed front invading $0_{\rr^d}$ at a speed $c$ greater than $\cLinMax$; let us denote by $\phi$ the profile of this pushed front. According to equality \cref{energy_pushed_front_travelling_frame} of \cref{prop:energy_pushed_front_travelling_frame}, for every speed $c'$ in the interval $(\cLinMax,c)$, the energy $\eee_{c'}[\phi]$ is negative. Since $\phi$ belongs to $H^1_c(\rr,\rr^d)$, it also belongs to $H^1_{c'}(\rr,\rr^d)$, thus $c'$ belongs to the set $\ccc_{-\infty}$; this ensures that the quantity $\cNonLinMax$ is greater than or equal to $c'$, and thus greater than $\cLinMax$; in other words, condition \cref{item:thm_characterization_existence_pushed_front_cLinMax_le_cNonLinMax} of \cref{thm:characterization_existence_pushed_front} holds.
\end{proof}
\begin{proof}[Proof of conclusion \cref{item:thm_characterization_existence_pushed_front_existence_speed_cNonLinMax} of \cref{thm:characterization_existence_pushed_front}]
Let us proceed by contradiction and assume that no pushed front invades $0_{\rr^d}$ at the speed $\cNonLinMax$. The proof (written above) that condition \cref{item:thm_characterization_existence_pushed_front_cLinMax_le_cNonLinMax} (of \cref{thm:characterization_existence_pushed_front}) implies condition \cref{item:thm_characterization_existence_pushed_front_existence} that there exists a pushed front invading $e$ at a speed arbitrarily close to $\cNonLinMax$ in the interval $(\cLinMax,\cNonLinMax)$. In other words, there exists an increasing sequence $(c_n)_{n\in\nn}$ of speeds in the interval $(\cLinMax,\cNonLinMax)$, going to $\cNonLinMax$ as $n$ goes to $+\infty$, such that, for every nonnegative integer $n$, there exists a pushed front invading $e$ at the speed $c_n$. Let $\phi_n$ denote the profile of this pushed front, normalized (with respect to space translations) by the condition $\abs{\phi_n(0)} = \deltaHess(c_0)$. Up to replacing the sequence $(\phi_n)_{n\in\nn}$ by a subsequence, it may be assumed that the vectors $\phi_n(0)$ converge, as $n$ goes to $+\infty$, towards a vector $u_\infty$ of $\partial B_{\rr^d}\bigl(0_{\rr^d},\deltaHess(c_0)\bigr)$. As shown by \cref{prop:extension_local_steep_stable_manifold}, the map $\wssloc{e}{\deltaHess(c_0)}{\cNonLinMax}$ defining the stable manifold of $(e,0_{\rr^d})$ for the differential system \cref{syst_trav_front_order_1} governing the profiles of fronts travelling at the speed $\cNonLinMax$ is defined on the closed ball $\widebar{B}_{\rr^d}\bigl(e,\deltaHess(c_0)\bigr)$. Let $\phi_\infty$ denote the solution of the corresponding second order differential \cref{syst_trav_front_order_2} (still for the speed $\cNonLinMax$) for the initial condition: 
\[
\bigl(\phi_\infty(0),\phi_\infty'(0)\bigr) = \bigl(u_\infty,\wssloc{e}{\deltaHess(c_0)}{\cNonLinMax}(u_\infty)\bigr)
\,.
\]
According to this definition, $\phi_\infty$ is the profile of a pushed travelling wave invading $e$ at the speed $\cNonLinMax$ (\cref{def:pushed_travelling_wave_front}). In addition, since according to conclusion \cref{item:lem_asymptotics_at_the_two_ends_of_space_left_end} of \cref{lem:asymptotics_at_the_two_ends_of_space} the quantities $\sup_{\xi\in\rr}\abs{\phi_n(\xi)}$ are bounded from above by a quantity depending only on $V$, the solution $\phi_\infty$ is globally defined, and $\sup_{\xi\in\rr}\abs{\phi_\infty(\xi)}$ is bounded from above by the same quantity. In other words, $\phi_\infty$ is the profile of a pushed travelling front (and not only a pushed travelling wave).
\end{proof}
\begin{remark}
For a generic potential $V$, the set of profiles of pushed travelling fronts and the set of speeds of pushed travelling fronts are discrete, \cite{JolyOliverBRisler_genericTransvPulledPushedTFParabGradSyst_2023}, so that in this case, the speeds $c_n$ introduced in the proof above must be equal to $\cNonLinMax$ for $n$ large enough, and the last compactness argument is unnecessary. 
\end{remark}
\begin{proof}[Proof of conclusion \cref{item:thm_characterization_existence_pushed_front_global_minimizer} of \cref{thm:characterization_existence_pushed_front}]
Let $\phi$ denote the profile of a pushed front invading $e$ at the speed $\cNonLinMax$. It follows from equality \cref{energy_of_pushed_front_vanishes} of \cref{prop:energy_pushed_front_travelling_frame} that its energy $\eee_{\cNonLinMax}[\phi]$ vanishes. On the other hand, it follows from the definition of $\cNonLinMax$ (\cref{def:max_nonlin_invasion_speed}) and from the fact that $\ccc_0$ is closed (conclusion \cref{item:prop_basic_properties_variational_structure_topology} of \cref{prop:basic_properties_variational_structure}) that $\cNonLinMax$ belongs to $\ccc_0$; according to the definition \cref{def_ccc_minus_infty_ccc_0} of $\ccc_0$, $\phi$ is therefore a global minimizer of the energy $\eee_{\cLinMax}[\cdot]$ in $H^1_{\cNonLinMax}(\rr,\rr^d)$, which is the intended conclusion. 
\end{proof}
\subsection{Proof of \texorpdfstring{\cref{cor:full_properties_ccc_minus_infty_and_ccc_0}}{Corollary \ref{cor:full_properties_ccc_minus_infty_and_ccc_0}}}
\label{subsec:proof_cor_full_properties_ccc_minus_infty_and_ccc_0}
\begin{proof}[Proof of \cref{cor:full_properties_ccc_minus_infty_and_ccc_0}]
According to conclusions \cref{item:prop_basic_properties_variational_structure_inclusion_ccc_minus_infty,item:prop_basic_properties_variational_structure_topology} of \cref{prop:basic_properties_variational_structure} and to the definition \cref{def_maximal_nonlinear_invasion_speed} of the quantity $\cNonLinMax$, 
\begin{equation}
\label{framing_ccc_minus_infty}
(0,\cLinMax) \subset \ccc_{-\infty} \subset (0,\cNonLinMax)
\,,
\end{equation}
and these inclusions show that, if $\cNonLinMax$ is equal to (that is, not larger than) $\cLinMax$, then conclusion \cref{full_properties_ccc_minus_infty_and_ccc_0} of \cref{cor:full_properties_ccc_minus_infty_and_ccc_0} holds. 

Let us assume that the converse holds, or in other words let us assume that $\cNonLinMax$ is larger than $\cLinMax$. In this case, conclusion \cref{item:thm_characterization_existence_pushed_front_existence_speed_cNonLinMax} of \cref{thm:characterization_existence_pushed_front} states that there exists a pushed front invading $e$ at the speed $\cNonLinMax$. Let us denote by $\phi$ the profile of this pushed front. According to equality \cref{energy_pushed_front_travelling_frame} of \cref{prop:energy_pushed_front_travelling_frame} (see \cref{fig:graph_eee_cprime_of_phi}), for every $c$ in $(0,\cNonLinMax)$, the energy $\eee_c[\phi]$ is negative; this shows that the whole interval $(0,\cNonLinMax)$ is included in $\ccc_{-\infty}$, and in view of the inclusions \cref{framing_ccc_minus_infty}, conclusion \cref{full_properties_ccc_minus_infty_and_ccc_0} of \cref{cor:full_properties_ccc_minus_infty_and_ccc_0} again holds. 

Still in this case where $\cNonLinMax$ is larger than $\cLinMax$, it follows from equality \cref{energy_of_pushed_front_vanishes} of \cref{prop:energy_pushed_front_travelling_frame} that the energy $\eee_{\cNonLinMax}[\phi]$ vanishes. This shows that $\cNonLinMax$ cannot be equal to $\cQuadHull$, or else conclusion \cref{item:lem_lower_bound_energy_trav_frame_without_bar_xi} of \cref{lem:lower_bound_energy_trav_frame} would ensure that the same energy is positive, a contradiction. In view of the second inequality of \cref{cLinMax_le_cNonLinMax_le_cQuadHull}, conclusion \cref{cor_full_properties_cNonLinMax_less_than_cQuadHull} of \cref{cor:full_properties_ccc_minus_infty_and_ccc_0} is proved. 

According to conclusion \cref{item:thm_characterization_existence_pushed_front_global_minimizer} of \cref{thm:characterization_existence_pushed_front}, for $c$ equals $\cNonLinMax$ there exists at least one global minimizer of $\eee_c[\cdot]$ in $H^1_c(\rr,\rr^d)$ which is not identically equal to $0_{\rr^d}$. According to the definition of $\ccc_{-\infty}$, such a global minimizer does not exist for $c$ in $(-\infty,\cNonLinMax)$. To complete the proof, let us proceed by contradiction and assume that a global minimizer of $\eee_c[\cdot]$ in $H^1_c(\rr,\rr^d)$, not identically equal to $0_{\rr^d}$, exists for some $c$ in $(\cNonLinMax,+\infty)$. It follows from \cref{prop:decrease_energy_trav_frame} that this minimizer must be the profile of a travelling wave invading $e$ at the speed $c$ (and, since this profile is in $H^1_c(\rr,\rr^d)$, a \emph{pushed} travelling wave). In addition, it follows from the coercivity assumption \cref{hyp_coerc} that this profile must be bounded, it is therefore the profile of a pushed travelling \emph{front}; and again according to equality \cref{energy_pushed_front_travelling_frame} of \cref{prop:energy_pushed_front_travelling_frame}, it follows that the whole interval $(0,c)$ must be in $\ccc_{-\infty}$, a  contradiction with the definition of $\cNonLinMax$. \Cref{cor:full_properties_ccc_minus_infty_and_ccc_0} is proved. 
\end{proof}
\appendix
\section{Some properties of the profiles of pushed travelling waves invading a critical point}
\label{sec:properties_profiles_pushed_trav_waves}
As in the previous sections, let us consider a potential $V$ in $\ccc^2(\rr^d,\rr)$ and a critical point $e$ of $V$, and let us assume that assumptions \cref{hyp_coerc} and \cref{hyp_crit_point} (stated in \cref{subsec:system_semi_flow,subsec:invaded_critical_point}) hold (in this \namecref{sec:properties_profiles_pushed_trav_waves} it will not be assumed that $e$ equals $0_{\rr^d}$). As in the notation \cref{def_mu_1_mu_d} let us denote by $\mu_1$ the least eigenvalue of $D^2V(e)$, and as in \cref{def_max_linear_invasion_speed} let us denote by $\cLinMax$ the maximal linear invasion speed of $e$. Let us recall the notation $\eee_c[\cdot]$ introduced in \cref{def_eee_c}, the notation $H^1_c(\rr,\rr^d)$ introduced in \cref{H1c}, and the notation $\muQuadHull$ and $\cQuadHull$ introduced in \cref{subsubsec:lower_quadratic_hull}. As in \cref{subsec:set_up}, let us consider a positive quantity $c_0$ and a negative quantity $\mu_0$ related by
\begin{equation}
\label{def_c0_appendix}
c_0 = 2\sqrt{-\mu_0} \iff \mu_0 = -\frac{c_0^2}{4}
\,,
\end{equation}
and let us assume that
\begin{equation}
\label{cLinMax_less_than_c0_less_than_cQuadHull_appendix}
\cLinMax< c_0 < \cQuadHull\,,
\quad\text{or equivalently}\quad
\muQuadHull < \mu_0 < \mu_1
\,, 
\end{equation}
see \vref{fig:correspondence_mu_c}. According to this assumption the quantity $\deltaHess(c_0)$ can be defined exactly as in \cref{subsec:invasion_point_defined_by_smaller_radius}, so that, for every $w$ in $\widebar{B}_{\rr^d}\bigl(e,\deltaHess(c_0)\bigr)$, 
\begin{equation}
\label{spetrum_of_D2V_of_u_larger_than_mu_0_appendix}
\sigma\bigl(D^2V(w)\bigr)\subset [\mu_0,+\infty)
\,,
\end{equation}
see \cref{fig:travelling_frame_invasion_points}. Finally, let us consider a speed $c$ greater than $c_0$, and let us consider the differential systems \cref{syst_trav_front_order_2,syst_trav_front_order_1} governing the profiles $\xi\mapsto\phi(\xi)$ of waves travelling at the speed $c$ for the parabolic system \cref{parabolic_system}:
\begin{equation}
\label{syst_trav_front_appendix}
\phi'' = -c\phi' + \nabla V(\phi)
\,,
\quad\text{or equivalently}\quad
\begin{pmatrix} \phi' \\ \psi' \end{pmatrix} = \begin{pmatrix}\psi \\ - c \psi + \nabla V(\phi)  \end{pmatrix}
\,.
\end{equation}
\subsection{Asymptotics at the two ends of space}
Let $\xi\mapsto\phi(\xi)$ denote a solution of the differential system \cref{syst_trav_front_appendix}, defined on a maximal interval $(\xi_-,+\infty)$ (for same quantity $\xi_-$ in $\{-\infty\}\cup\rr$) and satisfying the following properties:
\[
\phi(\xi)\xrightarrow[\xi\to+\infty]{}e 
\quad\text{and}\quad
\phi\not\equiv e
\,.
\]
\begin{lemma}[asymptotics at the two ends of space]
\label{lem:asymptotics_at_the_two_ends_of_space}
The following statements hold. 
\begin{enumerate}
\item If $\phi(\cdot)$ is bounded on $(\xi_-,+\infty)$, then $\xi_-$ equals $-\infty$, the quantity $\sup_{\xi\in\rr} \abs{\phi(\xi)}$ is bounded from above by a quantity depending only on $V$, and there exists a negative quantity $V_{-\infty}$ such that the following limits hold as $\xi$ goes to $-\infty$:
\[
\phi'(\xi)\to 0
\quad\text{and}\quad 
\dist\Bigl(\phi(\xi),\SigmaCrit(V)\cap V^{-1}\bigl(\{V_{-\infty}\}\bigr)\Bigr) \to 0
\,.
\]
\label{item:lem_asymptotics_at_the_two_ends_of_space_left_end}
\item If $\phi$ is the profile of a pushed travelling wave (\cref{def:pushed_travelling_wave_front}), then there exists a unique quantity $\hat{\xi}$ in $(\xi_-,+\infty)$ such that
\begin{equation}
\label{abs_phi_of_hat_xi_minus_e_equals_deltaHess}
\abs{\phi(\hat{\xi})-e} = \deltaHess(c_0) 
\,,
\end{equation}
and such that, for every $\xi$ in $(\hat{\xi},+\infty)$,
\begin{equation}
\label{abs_phi_of_hat_xi_minus_e_less_than_deltaHess_and_decreasing}
\abs{\phi(\xi)-e} < \deltaHess(c_0) 
\quad\text{and}\quad
\bigl(\phi(\xi)-e\bigr)\cdot\phi'(\xi)<0
\,.
\end{equation}
\label{item:lem_asymptotics_at_the_two_ends_of_space_negative_scalar_product}
\end{enumerate}
\end{lemma}
\begin{proof}
The proof of statement \cref{item:lem_asymptotics_at_the_two_ends_of_space_left_end} is identical to the proof of the last statement of \cite[Lemma 9]{Risler_globCVTravFronts_2008} (see also the proof of \cite[\globalBehaviourAsymptProfilesLeftEnd]{Risler_globalBehaviour_2016}). Let us prove statement \cref{item:lem_asymptotics_at_the_two_ends_of_space_negative_scalar_product}. For $\xi$ in $(\xi_-,+\infty)$, let us consider the quantity $q(\xi)$ defined as
\[
q(\xi) = \frac{1}{2}\bigl(\phi(\xi)-e\bigr)^2 
\,.
\]
Thus, for every $\xi$ in $(\xi_-,+\infty)$, 
\[
q'(\xi) = \bigl(\phi(\xi)-e\bigr)\cdot\phi'(\xi) 
\quad\text{and}\quad
q''(\xi) +c q'(\xi)  = \phi'(\xi)^2 + \bigl(\phi(\xi)-e\bigr)\cdot\nabla V\bigl(\phi\xi\bigr)
\,.
\]
In accordance with the notation introduced in \cref{def_SigmaFar}, let us consider the set
\[
\SigmaFar{\deltaHess(c_0)} = \bigl\{\xi\in(\xi_-,+\infty): \abs{\phi(\xi)-e}> \deltaHess(c_0)\bigr\}
\,.
\]
If this set is empty, then $\phi$ is bounded, and as a consequence $\xi_-$ is equal to $-\infty$. If it is nonempty, let us consider the quantity 
\[
\hat{\xi} = \sup(\SigmaFar{\deltaHess(c_0)})
\,,
\]
and let us consider the interval $\Iclose{\deltaHess(c_0)}$ defined as
\[
\Iclose{\deltaHess(c_0)} = \left\{
\begin{aligned}
\rr \quad&\text{if $\SigmaFar{\deltaHess(c_0)}$ is empty,}\\
[\hat{\xi},+\infty) \quad&\text{if $\SigmaFar{\deltaHess(c_0)}$ is nonempty.}
\end{aligned}
\right.
\]
It follows from the inclusion \cref{spetrum_of_D2V_of_u_larger_than_mu_0_appendix} that, for every $\xi$ in $\Iclose{\deltaHess(c_0)}$, 
\[
q''(\xi) +c q'(\xi) \ge \phi'(\xi)^2 + \mu_0 \bigl(\phi(\xi)-e\bigr)^2
\,,
\]
or, equivalently, 
\[
\frac{d}{d\xi}\bigl(e^{c\xi}q'(\xi)\bigr) \ge e^{c\xi}\Bigl( \phi'(\xi)^2 - \abs{\mu_0} \bigl(\phi(\xi)-e\bigr)^2\Bigr)
\,.
\]
Since $\phi$ is assumed to be the profile of a pushed travelling wave (\cref{def:pushed_travelling_wave_front}), it follows that, for every $\xi_0$ in $\Iclose{\deltaHess(c_0)}$,  
\[
- e^{c\xi_0} q'(\xi_0) \ge \int_{\xi_0}^{+\infty} e^{c\xi}\Bigl( \phi'(\xi)^2 - \abs{\mu_0} \bigl(\phi(\xi)-e\bigr)^2\Bigr)\, d\xi
\,,
\]
and it follows to Poincaré inequality \cref{Poincare_inequality_gamma_equals_c_over_2} applied to the function $\phi(\cdot)-e$ that
\[
- e^{c\xi_0} q'(\xi_0) \ge 2\left( \frac{c^2}{4} - \abs{\mu_0}\right) \int_{\xi_0}^{+\infty} e^{c\xi} q(\xi) \, d\xi
\,,
\]
so that 
\[
q'(\xi_0) \le - 2\left( \frac{c^2}{4} - \abs{\mu_0}\right) \int_0^{+\infty} e^{c\zeta}  q(\xi_0+\zeta) \, d\zeta
\,,
\]
and since $c$ is assumed to be greater than the quantity $c_0$ introduced in \cref{def_c0_appendix}, it follows that $q'(\xi_0)$ is negative, so that $q(\cdot)$ is strictly decreasing on $\Iclose{\deltaHess(c_0)}$. If in addition the set $\SigmaFar{\deltaHess(c_0)}$ is empty, then $q(\xi)$ must go to some finite positive limit $q_{-\infty}$ as $\xi$ goes to $-\infty$, and it would follow from the previous inequality that 
\[
\limsup_{\xi\to-\infty} q'(\xi)\le - 2\left( \frac{c^2}{4} - \abs{\mu_0}\right) \frac{q_{-\infty}}{c} < 0
\,,
\]
a contradiction. Thus the set $\SigmaFar{\deltaHess(c_0)}$ is nonempty, equality \cref{abs_phi_of_hat_xi_minus_e_equals_deltaHess} follows from the definition of $\hat{\xi}$, and inequalities \cref{abs_phi_of_hat_xi_minus_e_less_than_deltaHess_and_decreasing} from the fact that $q(\cdot)$ is strictly decreasing on $[\hat{\xi},+\infty)$. Statement \cref{item:lem_asymptotics_at_the_two_ends_of_space_negative_scalar_product} is proved. 
\end{proof}
\subsection{Uniform convergence at the right end of space}
Let us keep the notation and assumptions introduced at the beginning of \cref{sec:properties_profiles_pushed_trav_waves}, and let us consider a positive quantity $\delta$, smaller than or equal to $\deltaHess(c_0)$, such that the conclusions of \cref{prop:local_steep_stable_manifold} (local steep stable manifold) hold. The following lemma calls upon the notation $\phi_{c,u}(\cdot)$ introduced in \cref{initial_condition_phi_c_u}. 
\begin{lemma}[uniform convergence at the right end of space]
\label{lem:uniform_convergence_towards_e_wrt_u}
The convergence
\begin{equation}
\label{phi_c_u_goes_to_e_as_xi_goes_to_plus_infty}
\phi_{c,u}(\xi) \to e 
\quad\text{as}\quad
\xi\to+\infty
\end{equation}
is uniform with respect to $u$ in $\partial B_{\rr^d}(e,\delta)$.
\end{lemma}
\begin{proof}
For every $u$ in $\partial B_{\rr^d}(e,\delta)$, the limit \cref{phi_c_u_goes_to_e_as_xi_goes_to_plus_infty} follows from the definition \cref{initial_condition_phi_c_u} of $\phi_{c,u}$ (the only thing to prove is that this convergence is uniform). As a consequence, for every positive quantity $\varepsilon$, there exists a positive time $\xi(u,\varepsilon)$ such that $\abs{\phi_{c,u}\bigl(\xi(u,\varepsilon)\bigr)}$ is smaller than $\varepsilon/2$. Thus, by continuity of the solutions of system \cref{syst_trav_front_appendix} with respect to initial conditions, there exists an open neighbourhood $\nu(u,\varepsilon)$ of $u$ in $\partial B_{\rr^d}(e,\delta)$ such that, for every $u'$ in $\nu(u,\varepsilon)$, 
\[
\abs{\phi_{c,u'}\bigl(\xi(u,\varepsilon)\bigr)-e} < \varepsilon
\,.
\]
Since this set $\partial B_{\rr^d}(e,\delta)$ is compact, there exist a finite set $\{u_1,\dots,u_n\}$ of points of this set such that
\[
\partial B_{\rr^d}(e,\delta) \subset \bigcup_{i=1}^n \nu(u_i,\varepsilon)
\,.
\]
According to statement \cref{item:lem_asymptotics_at_the_two_ends_of_space_negative_scalar_product} of \cref{lem:asymptotics_at_the_two_ends_of_space} above, for every $u$ in $\partial B_{\rr^d}(e,\delta)$, the function $\xi\mapsto\abs{\phi_{c,u}(\xi)-e}$ is decreasing on $[0,+\infty)$. It follows that, for every time $\xi$ greater than $\max\bigl(\xi(u_1,\varepsilon),\dots,\xi(u_n,\varepsilon)\bigr)$, and for every $u$ in $\partial B_{\rr^d}(e,\delta)$, 
\[
\abs{\phi_{c,u}(\xi)-e} < \varepsilon
\,,
\]
which is the intended conclusion. 
\end{proof}
\subsection{Extension of the local steep stable manifold until the radius \texorpdfstring{$\deltaHess(c_0)$}{deltaHess(c0)}}
Following the notation of \cref{subsubsec:parametrization_pushed_tw_and_pushed_tf}, let us consider the set $\Wssloc{e}{\deltaHess(c_0)}{c}$ defined as
\[
\begin{aligned}
&\Wssloc{e}{\deltaHess(c_0)}{c} = \Bigl\{(\phi_0,\psi_0)\in\rr^d\times\rr^d: \text{ the solution $\xi\mapsto\phi(\xi)$ of the} \\
&\qquad \text{differential system \cref{syst_trav_front_appendix} with initial condition $\bigl(\phi(0),\phi'(0)\bigr) = (\phi_0,\psi_0)$ satisfies:} \\
&\qquad \text{$\abs{\phi(\xi)-e}\le \deltaHess(c_0)$ for all $\xi$ in $[0,+\infty)$ and $\phi(\xi)-e = o_{\xi\to+\infty}\bigl(e^{-\frac{1}{2}c\xi}\bigr)$}\Bigr\}
\,.
\end{aligned}
\]
The following proposition is the analogue \cite[Lemma~10]{Risler_globCVTravFronts_2008}. The proof is similar, however, for sake of completeness and since the context and the notation significantly differ, a comprehensive proof is provided below. Concerning the proof of the main result provided in \cref{sec:proof}, this \cref{prop:extension_local_steep_stable_manifold} shows that the quantity $\deltaLocMan(\widebar{c})$ introduced in \cref{subsec:convergence} could actually be chosen equal to the quantity $\deltaHess(c_0)$ introduced in \cref{subsec:invasion_point_defined_by_smaller_radius} (allowing a slightly simpler presentation without any significant benefit), see the remark following the notation \cref{abs_u_of_tilde_x_of_t_t_equals_tilde_delta}.
\begin{proposition}[extension of the local steep stable manifold until the radius $\deltaHess(c_0)$]
\label{prop:extension_local_steep_stable_manifold}
The set $\Wssloc{e}{\deltaHess(c_0)}{c}$ is the graph of a $\ccc^1$-map: $\widebar{B}_{\rr^d}\bigl(e,\deltaHess(c_0)\bigr)\to\rr^d$. 
\end{proposition}
In other words, for $c$ greater than $c_0$, the conclusions of \cref{prop:local_steep_stable_manifold} (defining the local steep stable manifold of $e$) hold for a parameter $\delta$ equal to $\deltaHess(c_0)$. 

The proof of this \cref{prop:extension_local_steep_stable_manifold} will follow from the next two lemmas. Let us consider the projectors 
\begin{equation}
\label{projectors}
\pi_1: \rr^d\times\rr^d\to\rr^d
\,,\quad
(u,v)\mapsto u 
\quad\text{and}\quad
\pi_2: \rr^d\times\rr^d\to\rr^d
\,,\quad
(u,v)\mapsto v
\,,
\end{equation}
and the map 
\[
\piOneRes: \Wssloc{e}{\deltaHess(c_0)}{c} \to \widebar{B}_{\rr^d}\bigl(e,\deltaHess(c_0)\bigr)
\]
defined as the restriction of $\pi_1$ to the departure set $\Wssloc{e}{\deltaHess(c_0)}{c}$ and the arrival set $\widebar{B}_{\rr^d}\bigl(e,\deltaHess(c_0)\bigr)$. 
\begin{lemma}[surjectivity]
\label{lem:surjectivity}
The map $\piOneRes$ is surjective. 
\end{lemma}
\begin{proof}
For every $u$ in $\partial B_{\rr^d}(e,\delta)$, it follows from statement \cref{item:lem_asymptotics_at_the_two_ends_of_space_negative_scalar_product} of \cref{lem:asymptotics_at_the_two_ends_of_space} that there exists a negative quantity $\hat{\xi}_u$ such that the function 
\[
\xi\mapsto \abs{\phi_{c,u}(\xi)-e}
\]
defines a a one-to-one correspondence between the interval $[\hat{\xi}_u,+\infty)$ and the interval $\bigl(0,\deltaHess(c_0)\bigr]$, see \cref{fig:one_to_one_correspondence}. 
\begin{figure}[!htbp]
\centering
\includegraphics[width=.7\textwidth]{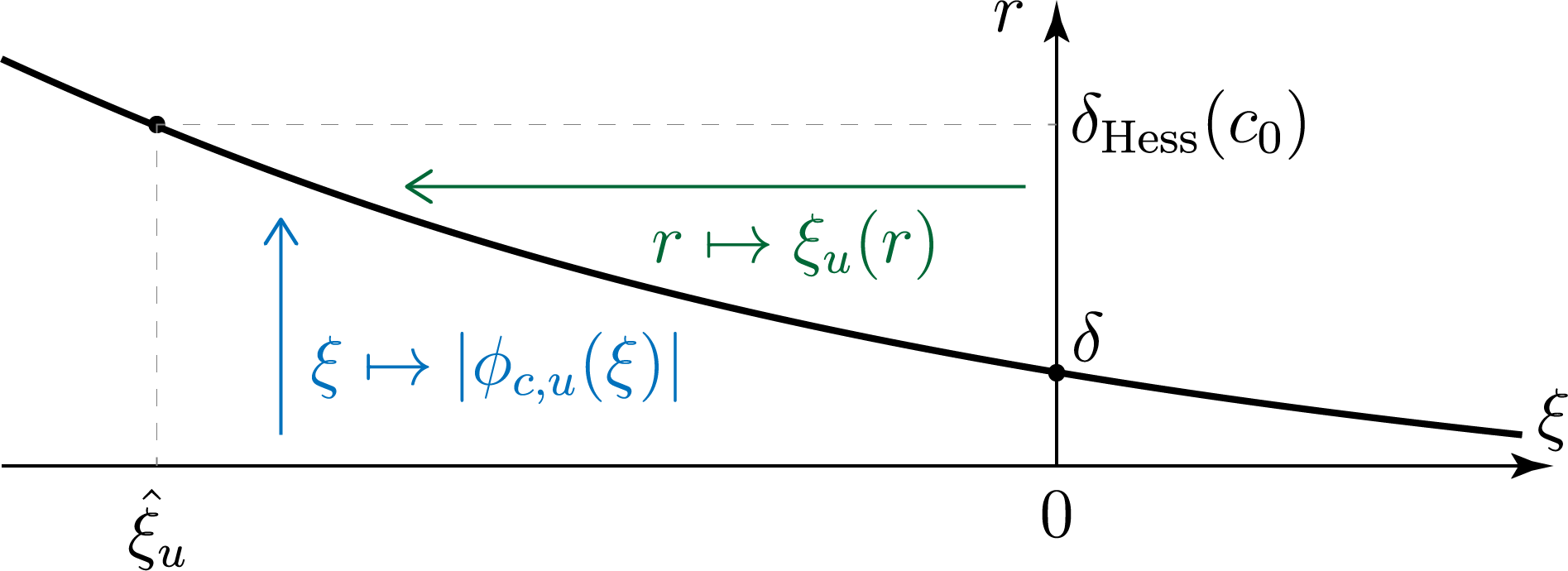}
\caption{One-to-one correspondence $\xi\mapsto\abs{\phi_{c,u}(\xi)}$ and inverse correspondence $r\mapsto\xi_u(r)$.}
\label{fig:one_to_one_correspondence}
\end{figure}
Let 
\[
\bigl(0,\deltaHess(c_0)\bigr]\to [\hat{\xi}_u,+\infty)\,,
\quad
r\mapsto \xi_u(r)
\]
denote the inverse correspondence. Then, for every $r$ in $\bigl(0,\deltaHess(c_0)\bigr]$, 
\[
\Bigl(\phi_{c,u}\bigl(\xi_u(r)\bigr),\phi_{c,u}'\bigl(\xi_u(r)\bigr)\Bigr) \in \Wssloc{e}{\deltaHess(c_0)}{c}
\,,
\]
so that
\[
\phi_{c,u}\bigl(\xi_u(r)\bigr) \in \piOneRes(\Wssloc{e}{\deltaHess(c_0)}{c})
\,.
\]
Let us consider the one-parameter family $(h_r)_{r\in[\delta,\deltaHess(c_0)]}$ of maps from $\mathbb{S}^{d-1}$ to $\mathbb{S}^{d-1}$ defined as
\[
h_r(v) = \frac{1}{r} \Bigl(\phi_{c,\delta v}\bigl(\xi_{\delta v}(r)\bigr)-e\Bigr)
\,.
\]
For every $v$ in $\mathbb{S}^{d-1}$, the quantity $\xi_{\delta v}(\delta)$ is equal to $0$, so that $\phi_{c,\delta v}\bigl(\xi_{\delta v}(r)\bigr)$ is equal to $\delta v$, and as a consequence $h_\delta(v)$ is equal to $v$. Thus $h_\delta$ is the identity of $\mathbb{S}^{d-1}$, so that, for every $r$ in $\bigl[\delta,\deltaHess(c_0)\bigr]$, $h_r$ is isotopic to the identity of $\mathbb{S}^{d-1}$, thus surjective (for topological reasons there is no retraction of $\mathbb{S}^{d-1}$ to a point, and the property ``$h_r$ non surjective'' would lead to the existence of such a retraction).  This shows that, for every $r$ in $\bigl[\delta,\deltaHess(c_0)\bigr]$, the set $\partial B_{\rr^d}(e,r)$ belongs to the image of $\piOneRes$, which is is therefore surjective. 
\end{proof}
Let us denote by $\Wss{e}{c}$ the (global) steep stable manifold of the equilibrium $(e,0_{\rr^d})$ for the differential system \cref{syst_trav_front_appendix}. This set is a $d$-dimensional $\ccc^1$-submanifold of $\rr^{2d}$, containing $\Wssloc{e}{\deltaHess(c_0)}{c}$. 
\begin{lemma}[transversality]
\label{lem:transversality}
For every $(\phi_0,\psi_0)$ in $\Wssloc{e}{\deltaHess(c_0)}{c}$, the intersection between the tangent space $T_{(\phi_0,\psi_0)}\Wss{e}{c}$ and $\{0_{\rr^d}\}\times\rr^d$ is transverse in $\rr^{2d}$. 
\end{lemma}
\begin{proof}
Take $(\phi_0,\psi_0)$ in $\Wssloc{e}{\deltaHess(c_0)}{c}$. If $(\phi_0,\psi_0)$ equals $(e,0_{\rr^d})$, then the conclusion follows from the expression \cref{eigenvectors} of the eigenvectors of the linearized differential systems \cref{syst_trav_front_order_1_2_linearized}. Let us assume that $(\phi_0,\psi_0)$ differs from $(e,0_{\rr^d})$, let us take a vector $(\tilde{\phi}_0,\tilde{\psi}_0)$ of $\rr^d\times\rr^d$, and let us consider the solution $\xi\mapsto \phi(\xi)$ of the differential systems \cref{syst_trav_front_appendix} for the initial condition $\bigl(\phi(0),\phi'(0)\bigr)$ equals $(\phi_0,\psi_0)$, and the solution $\xi\mapsto \tilde{\phi}(\xi)$ of the differential system
\[
\tilde{\phi}'' = -c \tilde{\phi} + D^2 V\bigl(\phi\bigr) \cdot \tilde{\phi}
\,,\quad\text{for the initial condition}\quad
\bigl(\tilde{\phi}(0),\tilde{\phi}'(0)\bigr)=(\tilde{\phi}_0,\tilde{\psi}_0)
\,.
\]
For every $\xi$ in $[0,+\infty)$, let us write
\[
\tilde{q}(\xi) = \frac{1}{2} \tilde{\phi}(\xi)^2
\,.
\]
Thus, for every $\xi$ in $[0,+\infty)$,
\[
\tilde{q}'(\xi) = \tilde{\phi}(\xi) \cdot \tilde{\phi}'(\xi)
\quad\text{and}\quad
\tilde{q}''(\xi) + c \tilde{q}'(\xi) = \tilde{\phi}'(\xi)^2 + D^2V\bigl(\phi(\xi)\bigr)\cdot \tilde{\phi}(\xi)\cdot \tilde{\phi}(\xi)
\,,
\] 
and it follows from the inclusion \cref{spetrum_of_D2V_of_u_larger_than_mu_0_appendix} that
\[
\tilde{q}''(\xi) + c \tilde{q}'(\xi) \ge \tilde{\phi}'(\xi)^2 + \mu_0 \tilde{\phi}(\xi)^2 
\,,
\]
or equivalently 
\[
\frac{d}{d\xi} \bigl(e^{c\xi}\tilde{q}'(\xi)\bigr) \ge e^{c\xi}\bigl(\tilde{\phi}'(\xi)^2 - \abs{\mu_0} \tilde{\phi}(\xi)^2\bigr)
\,.
\]
The vector $(\tilde{\phi}_0,\tilde{\psi}_0)$ belongs to the tangent space $T_{(\phi_0,\psi_0)}\Wss{e}{c}$ if and only if 
\begin{equation}
\label{condition_for_beeing_in_tangent_space_to_steep_stable_manifold}
\tilde{\phi}(\xi) = o\bigl(e^{-\frac{1}{2}c\xi}\bigr)
\quad\text{as}\quad
\xi\to+\infty
\,.
\end{equation}
If this equality \cref{condition_for_beeing_in_tangent_space_to_steep_stable_manifold} holds, then it follows that, 
\[
\tilde{q}'(0) \ge \int_{0}^{+\infty} e^{c\xi}\Bigl( \tilde{\phi}'(\xi)^2 - \abs{\mu_0} \tilde{\phi}^2\Bigr)\, d\xi
\,,
\]
and it follows from Poincaré inequality \cref{Poincare_inequality_gamma_equals_c_over_2} applied to the function $\tilde{\phi}(\cdot)-e$ that
\[
\tilde{q}'(0) \ge 2\left( \frac{c^2}{4} - \abs{\mu_0}\right) \int_{0}^{+\infty} e^{c\xi} \tilde{q}(\xi) \, d\xi
\,,
\]
and since $c$ is assumed to be greater than the quantity $c_0$ introduced in \cref{def_c0_appendix}, it follows that $\tilde{q}'(0)$ is negative, so that $\tilde{\psi}(0)$ is nonzero, which is the intended conclusion.
\end{proof}
\begin{proof}[Proof of Proposition \ref{prop:extension_local_steep_stable_manifold}]
According to \cref{lem:surjectivity,lem:transversality}, the map $\piOneRes$ defines a covering of $\widebar{B}_{\rr^d}\bigl(e,\deltaHess(c_0)\bigr)$ by $\Wssloc{e}{\deltaHess(c_0)}{c}$, and since $\Wssloc{e}{\deltaHess(c_0)}{c}$ is connected and $\widebar{B}_{\rr^d}\bigl(e,\deltaHess(c_0)\bigr)$ is simply connected, this covering must be a one-to-one correspondence. Let us denote by $\piOneRes^{-1}$ the inverse correspondence. Then, with the notation $\pi_2$ introduced in \cref{projectors}, the local steep stable manifold $\Wssloc{e}{\deltaHess(c_0)}{c}$ is the graph of the $\ccc^1$-map
\[
\pi_2\circ\Wssloc{e}{\deltaHess(c_0)}{c}: \widebar{B}_{\rr^d}\bigl(e,\deltaHess(c_0)\bigr) \to \rr^d
\,,
\]
which is the intended conclusion. 
\end{proof}
\section{An additional upper bound on the speeds of pushed fronts}
\label{sec:upper_bound_speed_pushed_front}
Let us keep the notation and assumptions of the beginning of \cref{sec:properties_profiles_pushed_trav_waves} (until the conditions \cref{cLinMax_less_than_c0_less_than_cQuadHull_appendix}, including these conditions). As in \cref{subsubsec:max_radius_stability_pushed_invasion}, let $\deltaStab(c_0)$ denote the maximal radius of stability of $e$ for pushed invasion at the speed $c_0$, and let us consider the quantity $\cUpp(c_0)$ defined as
\begin{equation}
\label{def_cUpp_preclusion_pushed_fronts_statement}
\cUpp(c_0) = \frac{2\sqrt{\abs{\Vmin}}}{\deltaStab(c_0)}
\,.
\end{equation}
\begin{figure}[!htbp]
\centering
\includegraphics[width=\textwidth]{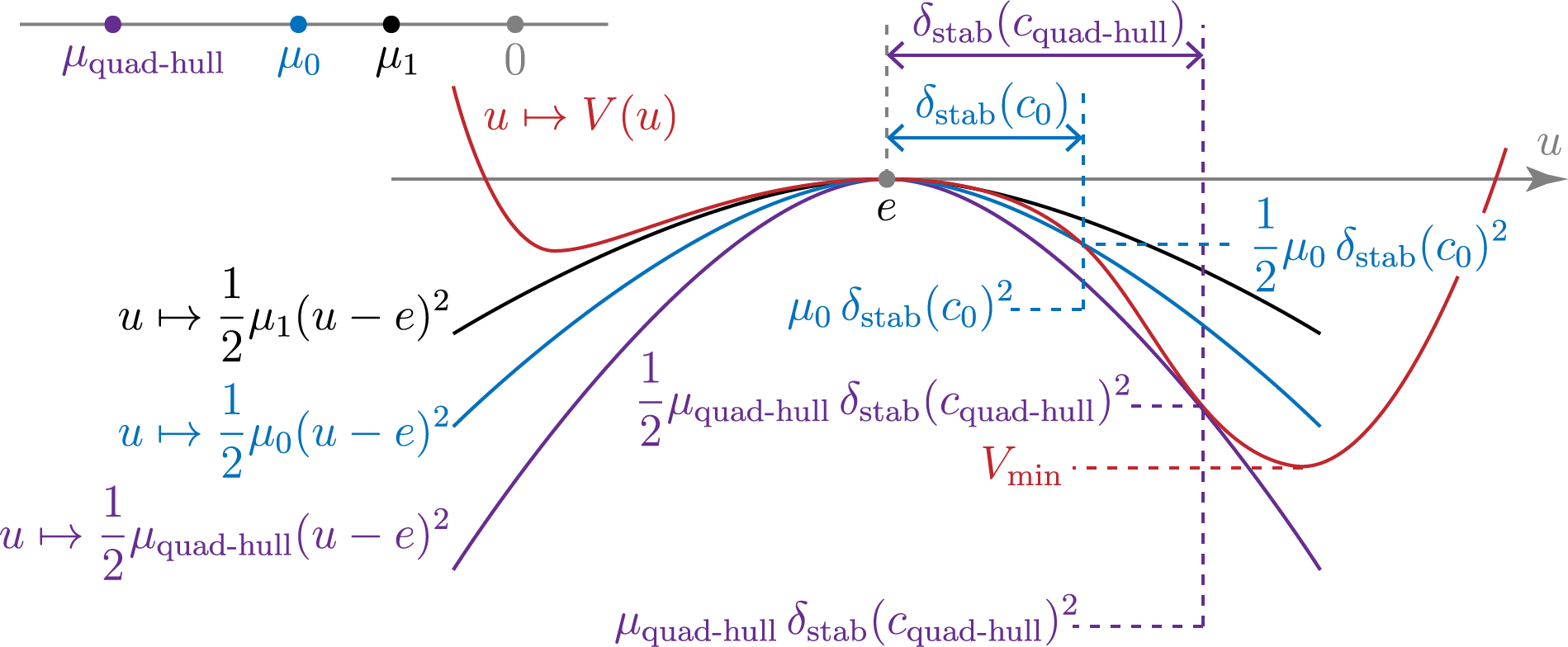}
\caption{Completion of \cref{fig:graph_of_V} with the quantity $\deltaStab(\cQuadHull)$ introduced in \cref{def_deltaStab_of_cQuadHull} and the quantity $\muQuadHull\deltaStab(\cQuadHull)^2$ appearing in \cref{condition_for_cUppDiag_less_than_cQuadHull}.}
\label{fig:graph_of_V_for_appendix}
\end{figure}
Both quantities $\deltaStab(c_0)$ and $\cUpp(c_0)$ can be viewed as functions of the parameter $c_0$, defined on the interval $(\cLinMax,\cQuadHull)$. On this interval, these functions are positive, bounded, and monotone (the function $\deltaStab(\cdot)$ is strictly increasing and the function $\cUpp(\cdot)$ is strictly decreasing), but not necessarily continuous, see \cref{fig:graph_delta_of_c_0}.
\begin{figure}[!htbp]
\centering
\includegraphics[width=.7\textwidth]{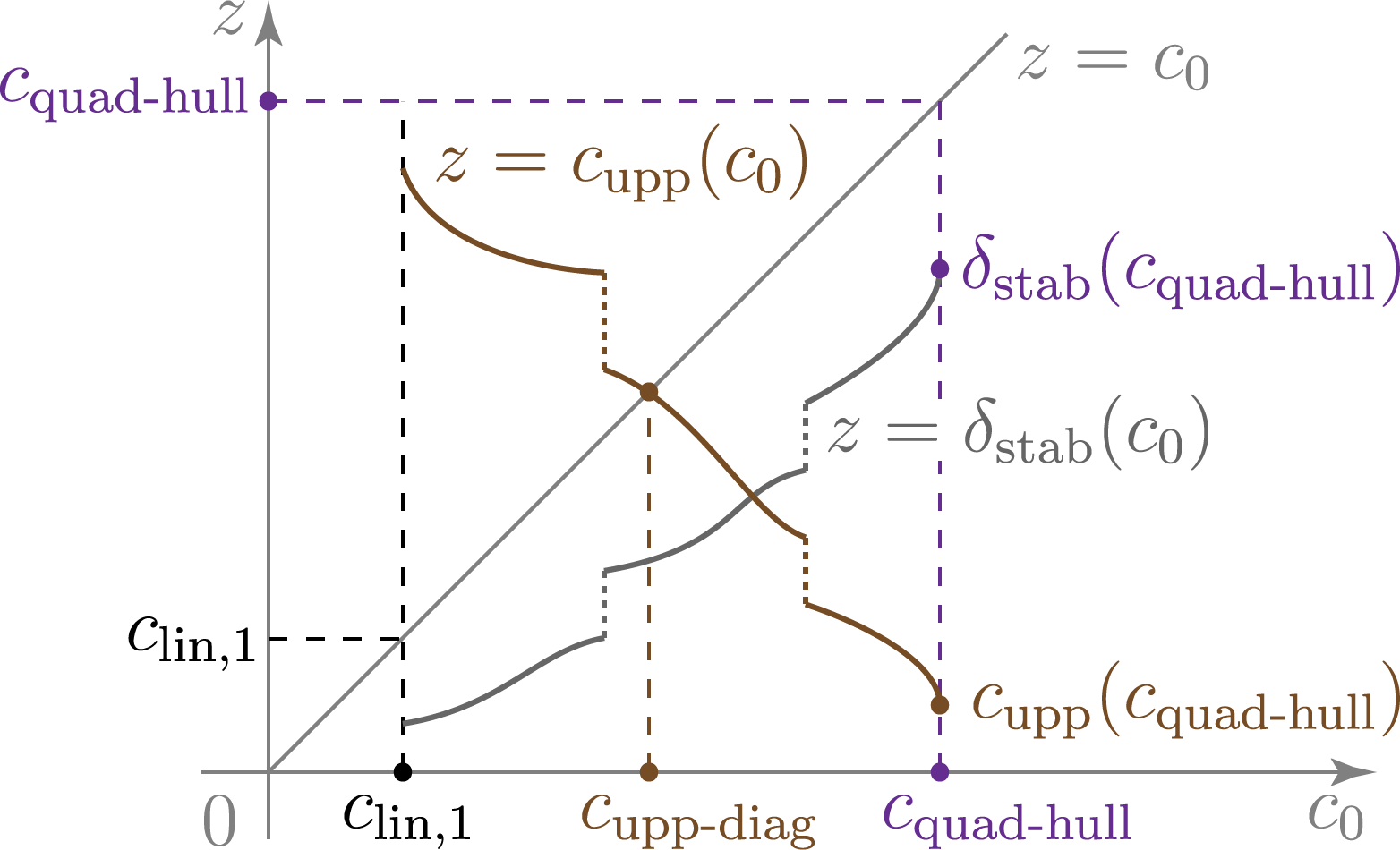}
\caption{Graphs of the functions $c_0\mapsto \deltaStab(c_0)$ and $c_0\mapsto\cUpp(c_0)$ and quantity $\cUppDiag$ defined in \cref{set_defining_cUppDiag}.}
\label{fig:graph_delta_of_c_0}
\end{figure}
The following proposition provides an additional constraint on the speed of a pushed front invading $e$. 
\begin{proposition}
\label{prop:c_smaller_than_cUpp_of_c0}
Let $c$ denote the speed of a pushed front invading $e$. Then, 
\[
c_0 \le c \implies c \le \cUpp(c_0)
\,.
\]
\end{proposition}
\begin{proof}
Let $c$ denote the speed of a pushed front invading $e$, let $v$ denote the profile of this front, and let us assume that $c$ is greater than or equal to $c_0$. According to equality \cref{energy_of_pushed_front_vanishes} of \cref{prop:energy_pushed_front_travelling_frame}, the quantity $\eee_c[v]$ vanishes. As a consequence, it follows from inequality \cref{lower_bound_energy_trav_frame_without_bar_xi} of statement \cref{item:lem_lower_bound_energy_trav_frame_without_bar_xi} of \cref{lem:lower_bound_energy_trav_frame} that inequality \cref{V_greater_than_one_half_mu0_v_of_xi_square} cannot hold for every $\xi$ in $\rr$, so that the set 
\[
\bigl\{\xi\in\rr:\abs{v(\xi)-e}>\deltaStab(c_0)\bigr\}
\]
is nonempty. On the other hand, since $v$ is in $H^1_c(\rr,\rr^d)$, this set is bounded from above. Let us denote by $\widebar{\xi}$ the supremum of this set. Then, $\abs{v(\widebar{\xi})-e}$ is equal to $\deltaStab(c_0)$ and inequality \cref{V_greater_than_one_half_mu0_v_of_xi_square} (with $e$ instead of $0_{\rr^d}$) holds for every $\xi$ in $[\widebar{\xi},+\infty)$, so that, according to inequality \cref{lower_bound_energy_trav_frame_with_bar_xi} of statement \cref{item:lem_lower_bound_energy_trav_frame_with_bar_xi} of \cref{lem:lower_bound_energy_trav_frame}, 
\[
\eee_c[v] \ge e^{c\widebar{\xi}}\left(-\frac{1}{c}\abs{\Vmin}+\frac{1}{2}\abs{\lambda_{c,-}(\mu_0)}\deltaStab(c_0)^2\right)
\,.
\]
Since $\eee_c[v]$ vanishes and since $\abs{\lambda_{c,-}(\mu_0)}$ is greater than or equal to $c/2$, it follows that
\[
c\le \frac{2\sqrt{\Vmin}}{\deltaStab(c_0)} = \cUpp(c_0)
\,,
\]
which is the intended conclusion. 
\end{proof}
It follows from this \cref{prop:c_smaller_than_cUpp_of_c0} that the condition
\[
c_0 \le \cUpp(c_0)\,,
\quad\text{or equivalently}\quad
c_0 \le \frac{2\sqrt{\abs{\Vmin}}}{\deltaStab(c_0)} 
\iff
\Vmin \le \mu_0\deltaStab(c_0)^2
\]
(see \cref{fig:graph_of_V_for_appendix}) is mandatory in order pushed fronts invading $e$ at some speed $c$ greater than or equal to $c_0$ to exist. In particular, if the condition
\begin{equation}
\label{condition_cUpp_larger_than_cLin1_for_c_close_to_cLin1}
\cLinMax < \lim_{c_0\to\cLinMax^+}\cUpp(c_0) 
\end{equation}
(see \cref{fig:graph_delta_of_c_0}) is not satisfied, then there exists no pushed front invading $e$ at a speed $c$ greater than $\cLinMax$. Let us therefore assume that this condition \cref{condition_cUpp_larger_than_cLin1_for_c_close_to_cLin1} is fulfilled, and let us consider the quantity $\cUppDiag$ defined as the supremum of the (nonempty) set
\begin{equation}
\label{set_defining_cUppDiag}
\bigl\{ c_0\in(\cLinMax,\cQuadHull): c_0 \le \cUpp(c_0) \bigr\}
\,.
\end{equation}
The following corollary is an immediate consequence of \cref{prop:c_smaller_than_cUpp_of_c0}.
\begin{corollary}
\label{cor:speed_of_a_front_cannot_exceed_cUppDiag}
The speed of a pushed front invading $e$ cannot be greater than the quantity $\cUppDiag$. 
\end{corollary}
Let us consider the continuous extensions of the functions $\deltaStab(\cdot)$ and $\cUpp(\cdot)$ to the interval $(\cLinMax,\cQuadHull]$ defined by
\begin{align}
\label{def_deltaStab_of_cQuadHull}
\deltaStab(\cQuadHull) &= \lim_{c_0\to\cQuadHull^-} \deltaStab(c_0)\,,\\
\nonumber
\text{and}\quad
\cUpp(\cQuadHull) &= \lim_{c_0\to\cQuadHull^-} \cUpp(c_0) = \frac{2\sqrt{\abs{\Vmin}}}{\deltaStab(\cQuadHull)}
\,,
\end{align}
see \cref{fig:graph_of_V_for_appendix,fig:graph_delta_of_c_0}. Observe that, if the following condition holds:
\[
\cUpp(\cQuadHull) < \cQuadHull
\,,
\]
see \cref{fig:graph_delta_of_c_0}, or equivalently:
\begin{equation}
\label{condition_for_cUppDiag_less_than_cQuadHull}
\frac{2\sqrt{\abs{\Vmin}}}{\deltaStab(\cQuadHull)} < \cQuadHull
\iff
\Vmin > \muQuadHull\deltaStab(\cQuadHull)^2
\,,
\end{equation}
see \cref{fig:graph_of_V_for_appendix}, then it follows that
\[
\cUppDiag < \cQuadHull
\,.
\]
In this case, the upper bound on the speeds of pushed fronts provided by \cref{cor:speed_of_a_front_cannot_exceed_cUppDiag} is better than the one provided by conclusion \cref{item:prop_basic_properties_variational_structure_inclusion_ccc_0} of \cref{prop:basic_properties_variational_structure}.
\section{Pulled and pushed travelling fronts in Fisher's model}
\label{sec:Fischer_model}
In the scalar case ($d$ equals $1$), the speed of a pushed front invading $e$ is necessarily greater than $\cLinMax$ (see the expression \cref{def_lambda_c_pm_of_mu} of the eigenvalues of the linearized system \cref{syst_trav_front_order_1_2_linearized}), and it follows from conclusion \cref{item:prop_basic_properties_variational_structure_inclusion_ccc_0} of \cref{prop:basic_properties_variational_structure} that, if $\muQuadHull$ is not less than $\mu_1$, then there is no pushed travelling front invading $e$. This result is well known and goes back (at least) to \cite[Corollary~9]{HadelerRothe_travellingFrontsNonlinearDiffEqu_1975}; in the same reference, Hadeler and Rothe consider the following reformulation of Fisher's (scalar) model \cite{Fisher_waveAdvanceAdvatageousGenes_1937}:
\begin{equation}
\label{Fisher_reaction_term}
u_t = f(u) + u_{xx}\,, 
\quad 
f(u) = u(1-u)\left(1+\frac{u}{\nu}\right) = u + \left(\frac{1}{\nu}-1\right) u^2 - \frac{1}{\nu} u^3
\,, 
\end{equation}
where $\nu$ is a positive parameter. The reaction term $f(u)$ derives from (that is, is equal to minus the derivative of) the potential $V$ defined as
\begin{equation}
\label{Fisher_potential}
V(u) = -\frac{1}{2}u^2 + \frac{1}{3}\left(1-\frac{1}{\nu}\right)u^3 + \frac{1}{4\nu} u^4
\,,
\end{equation}
see \cref{fig:Fisher_potential,fig:Fisher_reaction_term}. 
\begin{figure}[!htbp]
\centering
\begin{subfigure}{.45\textwidth}
\centering
\includegraphics[width=\textwidth]{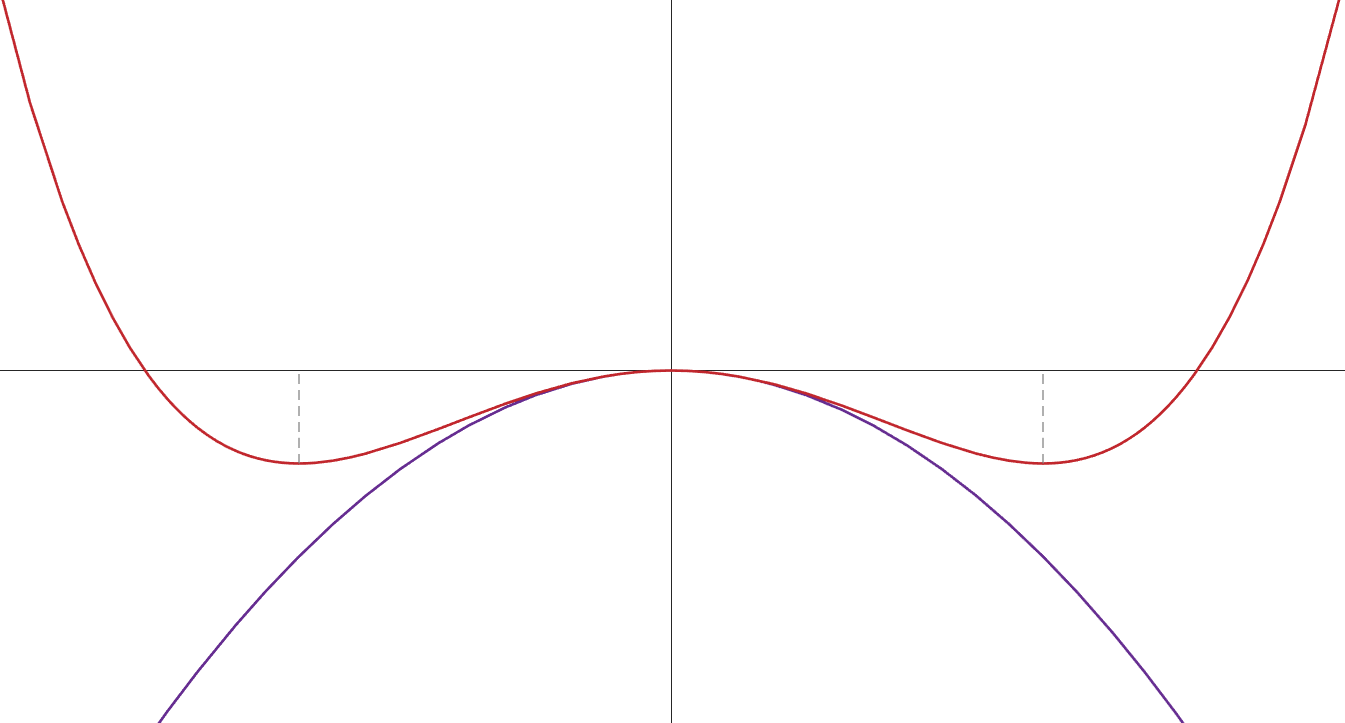}
\caption{$\nu = 1$}
\end{subfigure}
\hspace{.05\textwidth}
\begin{subfigure}{.45\textwidth}
\centering
\includegraphics[width=\textwidth]{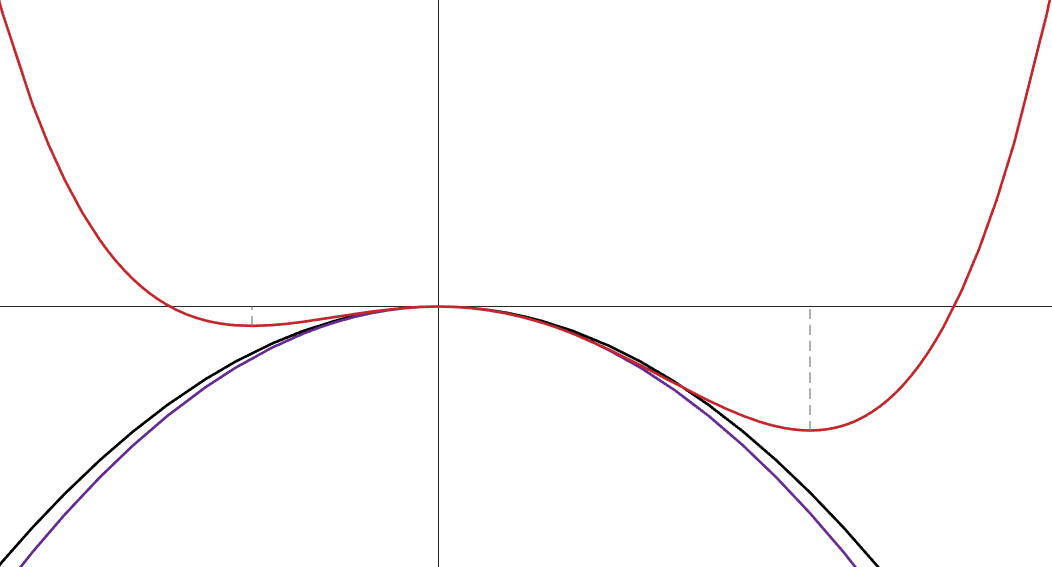}
\caption{$\nu = 1/2$}
\end{subfigure}
\begin{subfigure}{.45\textwidth}
\centering
\includegraphics[width=\textwidth]{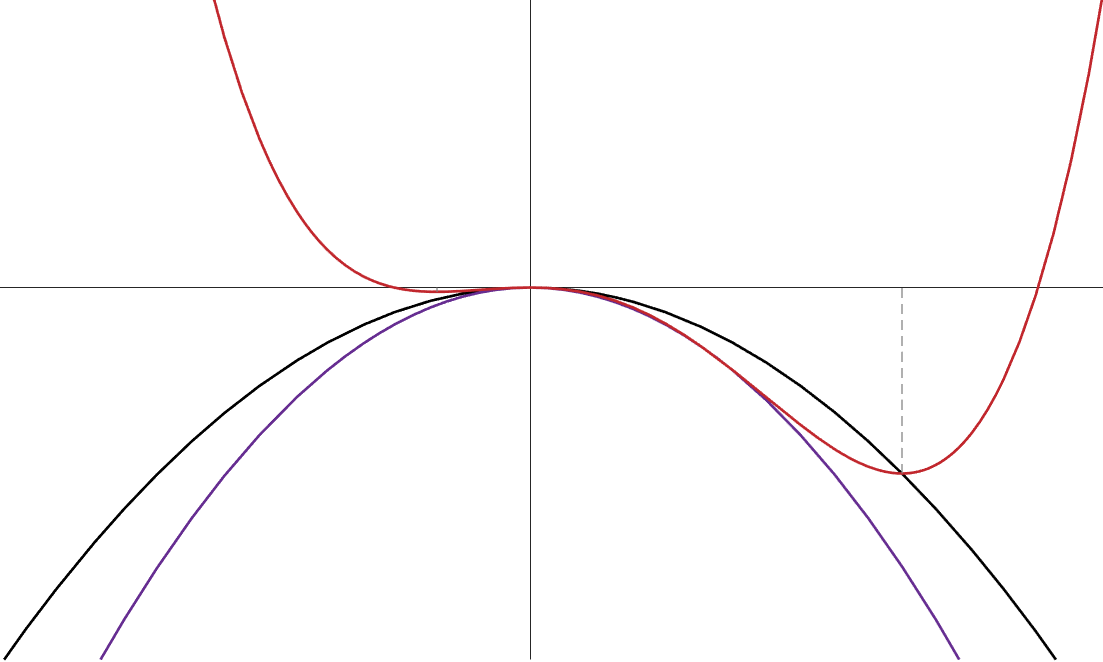}
\caption{$\nu = 1/4$}
\end{subfigure}
\hspace{.05\textwidth}
\begin{subfigure}{.45\textwidth}
\centering
\includegraphics[width=\textwidth]{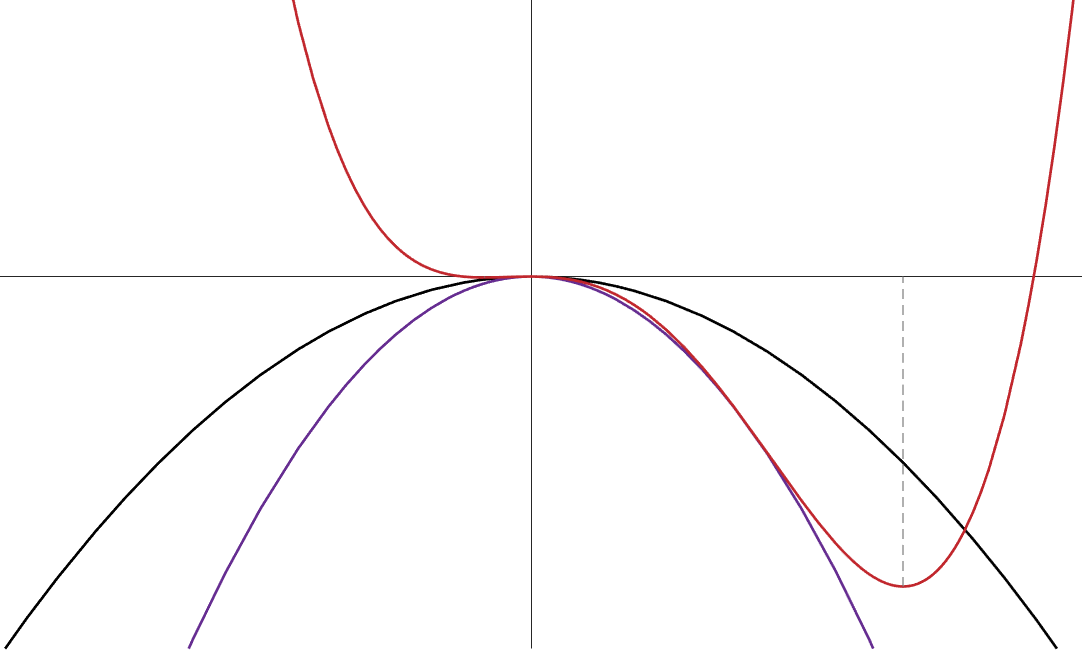}
\caption{$\nu = 1/8$}
\end{subfigure}
\caption{Graphs of the Fisher potential \cref{Fisher_potential} (in red) for various values of the parameter $\nu$. As on \cref{fig:graph_of_V_for_appendix}, the graph of $u\mapsto \frac{1}{2}\mu_1 u^2 = - \frac{1}{2} u^2$ is drawn in black, and the graph of $u\mapsto \frac{1}{2} \muQuadHull u^2$ is drawn in purple.}
\label{fig:Fisher_potential}
\end{figure}
\begin{figure}[!htbp]
\centering
\begin{subfigure}{.28\textwidth}
\centering
\includegraphics[width=\textwidth]{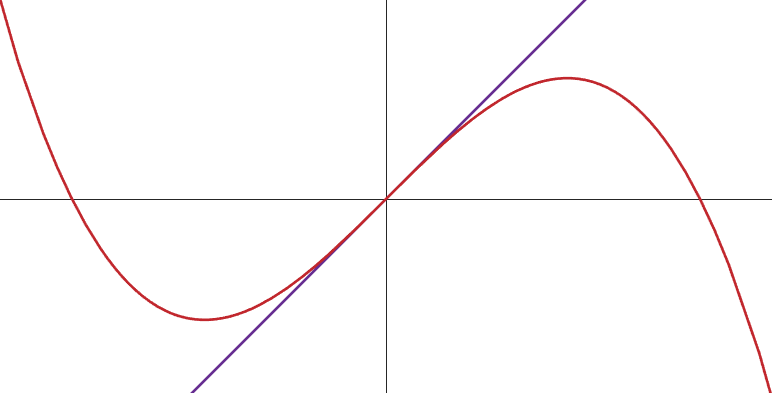}
\caption{$\nu = 1$}
\end{subfigure}
\hspace{.02\textwidth}
\begin{subfigure}{.22\textwidth}
\centering
\includegraphics[width=\textwidth]{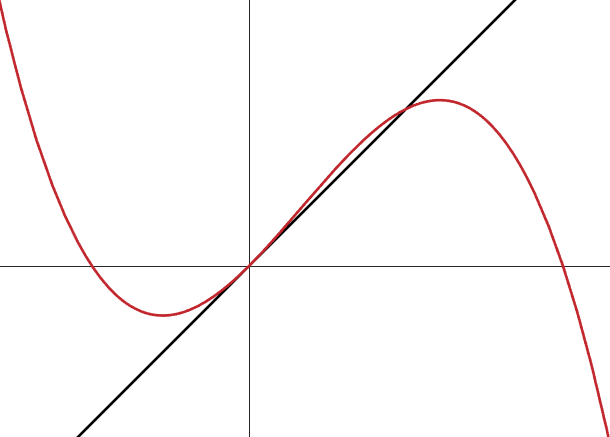}
\caption{$\nu = 1/2$}
\end{subfigure}
\hspace{.02\textwidth}
\begin{subfigure}{.2\textwidth}
\centering
\includegraphics[width=\textwidth]{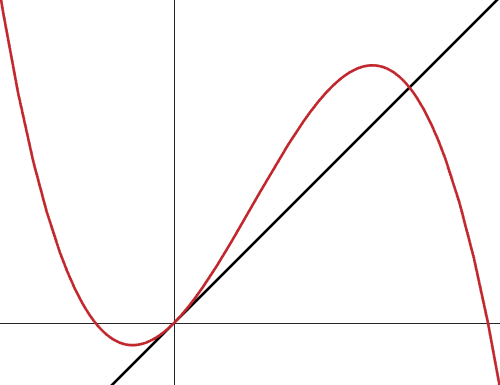}
\caption{$\nu = 1/4$}
\end{subfigure}
\hspace{.02\textwidth}
\begin{subfigure}{.15\textwidth}
\centering
\includegraphics[width=\textwidth]{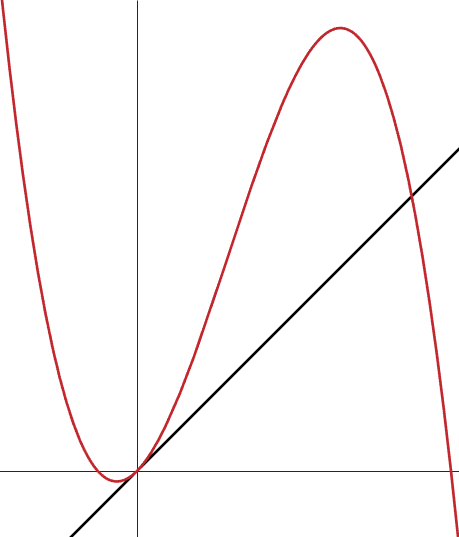}
\caption{$\nu = 1/8$}
\end{subfigure}
\caption{Graphs of the Fisher reaction term \cref{Fisher_reaction_term} (in red), for various values of the parameter $\nu$, and of the linear function $u\mapsto -\mu_1 u = u$ (in black) (this representation is more familiar than the potentials of \cref{fig:Fisher_potential} in the scalar case). The graphs of $u\mapsto -\muQuadHull u$, corresponding to the purple parabolas of \cref{fig:Fisher_potential}, are irrelevant, and therefore not represented.}
\label{fig:Fisher_reaction_term}
\end{figure}
This potential satisfies the coercivity assumption \cref{hyp_coerc} and has three critical points: a local maximum point at $u=0$ and two local minimum points at $u=-\nu$ and $u=1$. Thus, the role of the critical point $e$ considered insofar is played by $0$, and the quantity $\mu_1$ is equal to $V''(0)$, that is to $-1$. For every positive quantity $\nu$, there exist exactly two (up to translation) travelling fronts with monotone profiles invading $0$ and which are either pulled or pushed: one to the right of $0$ (with $1$ as the invading equilibrium) and one to the left of $0$ (with $-\nu$ as the invading equilibrium). When $\nu$ is equal to $1$ the potential $V$ is even and in this case the quantity called upon as $\muQuadHull$ is also equal to $-1$, so that no pushed front exists (see conclusion \cref{item:prop_basic_properties_variational_structure_inclusion_ccc_0} of \cref{prop:basic_properties_variational_structure} and comment above), and both fronts are pulled. As shown in \cite[Theorem~11]{HadelerRothe_travellingFrontsNonlinearDiffEqu_1975}, for $\nu$ between $0$ and $1$, the front to the left of $0$ is still pulled, and the front to the right of $0$ is: pulled if $1/2\le\nu\le1$ (and more precisely, pulled ``variational'' if $\nu$ equals $1/2$, and pulled ``non-variational'' if $1/2<\nu\le1$, \cite{Muratov_globVarStructPropagation_2004,JolyOliverBRisler_genericTransvPulledPushedTFParabGradSyst_2023}), and pushed if $0<\nu < 1/2$; see \cref{table:Fisher_model}.
\begin{table}
\centering
\begin{tabular}{|c|c|c|c|}
\hline
$\nu = 1$ & $0<\cLinMax = \cNonLinMax = \cQuadHull$ & pulled non variational & \cref{item:four_cases_KPP} \\ \hline
$1/2 < \nu<1$ & $0<\cLinMax = \cNonLinMax < \cQuadHull$ & pulled non variational & \cref{item:four_cases_not_enough_pushed} \\ \hline
$\nu = 1/2$ & $0<\cLinMax = \cNonLinMax < \cQuadHull$ & pulled variational & \cref{item:four_cases_not_enough_pushed} \\ \hline
$0 < \nu < 1/2$ & $0<\cLinMax < \cNonLinMax < \cQuadHull$ & pushed & \cref{item:four_cases_strict_inequalities} \\ \hline
\end{tabular}
\caption{Relations between the three quantities $\cLinMax$ and $\cNonLinMax$ and $\cQuadHull$ and nature of the monotone travelling front with minimal speed (pulled or pushed) to the right of $0$ in Fisher's model \cref{Fisher_reaction_term}, depending on the value of the parameter $\nu$. The fourth column contains the number of the ``case'' introduced in \cref{subsubsec:maximal_nonlinear_invasion_speed} and displayed on \cref{fig:line_of_speeds}.} 
\label{table:Fisher_model}
\end{table}
For a similar discussion on the subcritical quintic Ginzburg--Landau equation
\[
u_t = -\mu_1 u + u^3 - u^5  + u_{xx}
\,,
\]
see \cite[section~5]{Muratov_globVarStructPropagation_2004}.
\subsubsection*{Acknowledgements} 
The authors are indebted to Thierry Gallay and Romain Joly for their interest and support through numerous fruitful discussions. 
\emergencystretch=2em 
\printbibliography 
\bigskip
\RamonsSignature
\bigskip

\mySignature
\end{document}